\documentclass[10pt]{article}

\usepackage[ngerman,french,english]{babel}

\usepackage{pdfsync}

\usepackage{amsmath}                                                        
\numberwithin{equation}{section}
\usepackage{graphicx}                                                       
\usepackage{subfigure}
\usepackage{enumitem}                                                    
\usepackage{hyperref}
\usepackage{amssymb}                                                        
\usepackage[mathscr]{eucal}                                             
\usepackage{cancel}                                                             
\usepackage[normalem]{ulem}                                                                 
\usepackage{pstricks}
\usepackage{rotating}
\usepackage{lscape}
\usepackage[paperwidth=8.5in,paperheight=11in,top=1.00in, bottom=1.00in, left=1.00in, right=1.00in]{geometry}
\usepackage{mathtools}                                                      
\mathtoolsset{showonlyrefs=true}                                    
\usepackage{fixltx2e,amsmath}                                           
\MakeRobust{\eqref}

\usepackage{dsfont}
\usepackage{comment}

\usepackage{mathdots}
\usepackage{amsthm}                                                             
\allowdisplaybreaks

\theoremstyle{plain}
\newtheorem{theorem}{Theorem}
\numberwithin{theorem}{section}

\newtheorem{lemma}[theorem]{Lemma}                              
\newtheorem{proposition}[theorem]{Proposition}

\newtheorem{corollary}[theorem]{Corollary}
\theoremstyle{definition}
\newtheorem{definition}[theorem]{Definition}
\newtheorem{example}[theorem]{Example}
\newtheorem{notation}[theorem]{Notation}
\newtheorem{remark}[theorem]{Remark}
\newtheorem{assumption}[theorem]{Assumption}

\def \s {{\sigma}}

\def \R {\mathbb{R}}
\def \p {\partial}

\newcommand{\CT}[1]{\mathit{C}_T \mathcal{C}_b\cap \Leb^\infty_T\mathcal{C}_B^{#1}}
\newcommand{\Linf}[1]{ \Leb^\infty_T\mathcal{C}_B^{#1}}

\newcommand{\<}{\langle}
\renewcommand{\>}{\rangle}
\renewcommand{\(}{\left(}
\renewcommand{\)}{\right)}

\newcommand{\Leb}{L}

\newcommand\Eb{\mathbb{E}}
\newcommand\Pb{\mathbb{P}}

\newcommand\Rb{\mathbb{R}}

\newcommand\Nzero{\mathbb{N}_0}

\newcommand\Ac{\mathscr{A}}

\newcommand\Bc{\mathscr{B}}
\newcommand\Cc{\mathscr{C}}

\newcommand\Lc{\mathscr{L}}

\newcommand\Kc{\mathscr{K}}

\newcommand\eps{\varepsilon}

\newcommand\fh{\widehat{f}}

\newcommand\dd{d}

\def \rr {r}

\def \R  {{\mathbb {R}}}

\def \eps {{\varepsilon}}

\def \n {{\nu}}

\def \p {{\partial}}

\def \It\^o {It\^o }
\def \s {{\sigma}}

\def \R {{\mathbb {R}}}
\def \N {{\mathbb {N}}}
\def \C {{\mathcal {C}}}

\def \eps {{\varepsilon}}

\def \n {{\nu}}

\def \div {{\text{\rm div}}}

\def \tilde {\widetilde}

\def \A {\mathcal{A}}

\def \F {\mathcal{F}}

\def \BB {{\bf B}}

\def \Ã  {{\`a }}
\def \Ã¨ {{\`e }}
\def \Ã² {{\`o }}
\def \Ã¹ {{\`u }}

\begin{document}

\title{Degenerate McKean-Vlasov equations with  drift in anisotropic negative Besov spaces}

\author{Elena Issoglio \thanks{Dipartimento di Matematica G. Peano, Universit\`a di Torino, Torino, Italy.  \textbf{E-mail}: elena.issoglio@unito.it} \and
Stefano Pagliarani
\thanks{Dipartimento di Matematica, Universit\`a di Bologna, Bologna, Italy.
  \textbf{E-mail}: stefano.pagliarani9@unibo.it.} \and Francesco Russo
\thanks{Unit\'e de Math\'ematiques Appliqu\'ees, ENSTA Paris,
  Institut Polytechnique de Paris, Palaiseau, France.
  \textbf{E-mail}: francesco.russo@ensta.fr
} \and Davide Trevisani
\thanks{Facultad de Informática, Universitade da Coruña, A Coruña, Spain.
\textbf{E-mail}: davide.trevisani@udc.es}}

\date{This version: 16th March 2026}

\maketitle

\begin{abstract}
  The paper is concerned with a McKean-Vlasov type SDE with  drift   in anisotropic Besov spaces with negative regularity and with
   degenerate diffusion matrix
  under the weak H\"ormander condition.
The main result is of existence and uniqueness of a solution  in law for the McKean-Vlasov equation, which is formulated as a suitable martingale problem.
  All analytical tools needed are derived in the paper, such as the well-posedness of the Fokker-Planck and Kolmogorov PDEs with   distributional drift, as well as continuity dependence on the coefficients. The solutions to these PDEs naturally live in anisotropic Besov spaces, for which we developed suitable analytical inequalities, such as Schauder estimates.
\end{abstract}


%
%

\tiny
\tableofcontents
\normalsize
\section{Introduction and preliminaries}
\subsection{Statement of the problem}\label{subsec:statement}
In this work we treat the {weak} well-posedness for a class of of 
degenerate  McKean-Vlasov-type SDEs, 
whose drift contains an {{anisotropic Besov distribution} 
of negative order $-\beta$, and $\beta \in (0,1/2)$.} Let $0< d < N$ be two given natural numbers, and let us denote by  $Z_t:=(V_t,X_t)$ {a stochastic process valued on $\R^{N}=\R^{d}\times \R^{N-d}$.}

The class of equations that we consider in this paper is given as
\begin{equation}\label{eq:mkv_singular}
\begin{cases}
\dd V_t = \Big( F\big( u_t(Z_t)  \big) b_t(Z_t) + \Bc_0 Z_t \Big) \dd t   +     \dd W_t       \\
\dd X_t = \Bc_1 Z_t \dd t \\
\mu_{Z_t} (\dd z) = u_t(z) \dd z,
\end{cases}
\end{equation}
for a given random initial condition
\begin{equation}\label{eq:initial_cond_Z}
Z_0=(V_0,X_0) \sim\mu_0 = u^0(z) (dz). 
\end{equation}
Here  
$W$ is a $d$-dimensional Brownian motion, $\Bc_0\in\mathcal{M}^{d \times N}$ and 
$\Bc_1\in\mathcal{M}^{(N-d)\times N}$ are constant matrices, and $F:[0,+\infty) \to \mathcal{M}^{d\times m}$ is a  matrix-valued Borel function. 
 In order for the process $(V,X)$ to take values on $\R^{N}$ one can formally consider $b:[0,T] \times \R^N \to \R^{m}$ in a way that the product in \eqref{eq:mkv_singular} can be made sense of as a usual multiplication \emph{row-by-column}. One of the novel features here, however, is that 
 $b_t(\cdot)$ is a distribution, hence the {\em product} $ F(u) b$ needs to be defined with care.
Another trait of \eqref{eq:mkv_singular} is that the unknown $Z$ has a density $u$ which appears in the drift of the $V$-component, and this makes the problem a McKean-Vlasov type SDE. Notice moreover that the SDE is degenerate, in the sense that the noise $W$ appears only in the  dynamics of the $V$-component. 
In order for the SDE to be well-posed,  we will then require the hypoellipticity of the Kolmogorov operator $\Kc$ associated to \eqref{eq:mkv_singular} when $F\equiv 0$ (see Assumption \ref{ass:kolmogorov_op}).

We point out that, when $N =2d$, the system \eqref{eq:mkv_singular} reduces to the  {\em singular  kinetic} McKean-Vlasov SDE
\begin{equation}\label{eq:mkv_kinetic_singular}
\begin{cases}
\dd V_t = F\big( u_t(V_t,X_t)  \big) b_t(V_t , X_t) \dd t   +     \dd W_t   ,    \\
\dd X_t = V_t \dd t , \\
\mu_{(V_t,X_t)} ( \dd v \dd x )= u_t(v,x) \dd v \dd x ,
\end{cases}
\end{equation}
where the $\R^{2d}$-valued pair $(V_t,X_t)$ describes {\em velocity} and {\em position} of a particle in the phase-space. When $F\equiv 1$ and $b_t$ is a linear function of $V_t$ (in particular there is no McKean-Vlasov interaction and $b_t$ is not singular), \eqref{eq:mkv_kinetic_singular} reduces to the classical Langevin model \cite{langevin1908theorie}, which serves as a pilot example of more complex kinetic models (see \cite{imbert2021schauderSilvestre}, \cite{imbert2021schauder} among others). Recently, kinetic McKean-Vlasov models
attracted the interest of several authors (e.g. \cite{bossy,hao2024singular,veretennikov2023weak,zhang2021second}). As it is pointed out in \cite{hao2024singular},\eqref{eq:mkv_kinetic_singular} reflects the macroscopic behavior of a system of particles obeying Newton's second law with Gaussian noise, interacting through a mean-field force coefficient, with the singular nature of the latter accounting for environmental noise that acts on all particles. In \cite{bossy} the authors study a conditional version of \eqref{eq:mkv_kinetic_singular}, as a Langrangian stochastic model alternative to Navier-Stokes equations for the simulation of turbulent flows. In \cite{zhang2021second} (see also \cite{veretennikov2023weak}), the author studies a non-singular version of \eqref{eq:mkv_kinetic_singular} where the McKean-Vlasov interaction also affects the diffusion term, motivated by the fact that the non-linear Fokker-Planck PDE associated to this class includes to the kinetic Landau equation as a prototype example, the latter being the most popular mathematical model in the theory of collisional plasma. In this regards, the $2$nd order part of the Fokker-Planck PDE associated to \eqref{eq:mkv_kinetic_singular}, which reads as
\begin{equation}\label{eq:fokker_plank_kinetic_ex}
(\partial_t + \langle v, \nabla_x \rangle ) u_t +  \div_v\big(u_t F(u_t)b_t\big) = \Delta_v u_t ,
\end{equation}
differs from the one in the Landau equation, in particular it is linear. 
However, it is argued in \cite{alexandre2004landau} that a satisfactory description of a plasma is obtained when one also takes into account not only binary collisions (which are missing in \eqref{eq:fokker_plank_kinetic_ex}) but also collective effects modelled by a mean-field force term, which is present in \eqref{eq:mkv_kinetic_singular} and reflected by the non-linear $1$st order term in \eqref{eq:fokker_plank_kinetic_ex}. 

The main aim of this paper is to give a mathematical meaning to equation \eqref{eq:mkv_singular} and to prove existence and uniqueness of a solution, together with some regularity results on its time-marginal probability densities (see Theorem \ref{thm:ex_uniq_McKeanSDE_intro} and its proof in Section \ref{sec:MKMP}). To define the notion of solution to the McKean-Vlasov equation, which is a non-linear equation in the sense that the drift of the SDE depends on the solution  via the density of the solution itself, a natural step is to define and investigate its linear  counterpart, which is a kinetic-type SDE where the drift  is singular but does not depend on the law of the solution itself. Such SDE is formulated as a martingale problem in Section \ref{sec:linear MP}, and results on  its well-posedness 
are the first by-product of this work.
 We denote by $\Kc$ the parabolic generator of the underlying
  Gaussian process solving the SDE \eqref{eq:mkv_singular},
  see \eqref{eq:SDEgeneralLangevin}.             
 Another by-product is the study of the nonlinear Fokker-Planck (FP) PDE given by
\begin{equation}\label{eq:Kc'}
\begin{cases}
\Kc'  u_t = \div_v\big(u_t F(u_t)b_t\big), \qquad t\in (0,T],\\
u_0=u^0,
\end{cases}
\end{equation}
where $\Kc'$ is the  formal adjoint of $\Kc$,
 given in \eqref{eq:FP operator kinetic}-\eqref{eq:FP operator kinetic_bis},
$u^0 $ is the law density of $\mu_{(V_0, X_0)}$ and $\div_v$ denotes the divergence in the variable $v$ of the vector $(v,x)$.  Indeed, we prove that the density of the 
solution to \eqref{eq:mkv_singular} does satisfy \eqref{eq:Kc'} and for this, analytical results for the latter are  crucial (see Section \ref{sc:FP}), which might have separate interest. These include the existence and uniqueness of a weak solution in an anisotropic H\"older space, as well as stability results for the regularized version of \eqref{eq:Kc'}, obtained by replacing the distribution $b_t$ with a smooth function. Finally, we mention a third by-product, namely analytical results  on the singular kinetic-type backward Kolmogorov PDE  
\begin{equation}
\label{eq:back_cauchy probl_intro}
\begin{cases}
\Kc  u + \langle \Bc , \nabla_v \rangle u  = \lambda u + g  , \qquad t\in (0,T),\\
u_T=\ell,
\end{cases}
\end{equation}
with $\Kc$ being defined in \eqref{eq:kolm_const},  $\Bc$  and $g$ being  suitable distributions, and $\ell$ a suitable function. This PDE is the fundamental tool used to set up the linear martingale problem,  and  in Section \ref{sec:Kolmogorov PDE} we prove results that are analogous to those for the Fokker-Planck PDE \eqref{eq:Kc'}. We also note that all the PDE results of this work rely on suitable Schauder estimates in anisotropic Besov spaces for the semigroups of $\Kc$ and $\Kc'$, which 
are stated in Theorem \ref{th:schauder_first} and proved in Appendix \ref{app:proof_Schauder}. 

\subsection{Literature review}
As we have mentioned above, the system \eqref{eq:mkv_singular} that we will investigate combines three features: distributional coefficients (hence the term {\em singular}), the fact that the drift depends on  the density of the law of the unknown itself (hence the term {\em McKean-Vlasov}), the fact that the Brownian motion drives only the $V$-component (hence  the term {\em degenerate} or {\em kinetic}). All these  features have previously been studied in the literature, as we see below, but this is,  to the best of our knowledge, the first time they are combined together when the drift depends in a pointwise fashion on the density law.

The first examples of SDEs with distributional drifts date back some 20 years ago with the works  \cite{bass_chen, frw2,frw1,russo_trutnau07} in one dimension. Some years later, 
\cite{flandoli_et.al14} treat the case of any dimension $d\geq 1$ and with a drift in a fractional Sobolev space of negative regularity index $-\beta>-1/2$, in the so-called Young regime. 
Immediately after and independently   
\cite{diel} treat the case of  singular coefficients in Besov spaces with negative regularity index $-\beta>-2/3$ in dimension 1 (i.e  beyond the Young regime). In the following years 
\cite{catellier_chouk}   investigate extensions  in dimension $d$ and $-\beta>-2/3$. 
Another generalization is given simultaneously in  
\cite{athreya2020} and
\cite{chaudru_menozzi}, where the authors  consider the extension to $\alpha$-stable noises in place of Brownian motion, in the Young regime and dimension $d$, producing also an explicit description of the dynamics of the solution using a suitable semigroup. In \cite{issoglio_russoMPb} 
they frame the problem as a martingale problem and  derive a description of the dynamics in terms of weak Dirichlet processes.  
One of the common tools that can be found in many of these works is a singular Kolmogorov equation like \eqref{eq:back_cauchy probl_intro} in the 
uniformly parabolic case, i.e. when $N=d$, which is needed to cast the singular SDE as a  martingale problem. A solution for the latter is a process $Z$ and a probability $\mathbb P$ such that  for all $f\in \mathcal D$ we have $f(t, Z_t) - f(0, X_0) - \int_0^t \mathcal L f (s, X_s) ds $ is a $\mathbb P$-martingale, where $\mathcal L$ is the
`generator' of $Z$ (formally given by $\mathcal L = \Kc + \langle \Bc , \nabla_v \rangle $). The class of functions $\mathcal D$ is constructed as the set of solutions to \eqref{eq:back_cauchy probl_intro}  for smooth forcing terms $g$ and terminal conditions $\ell$. 
Once the martingale problem is defined, one uses results on the Kolmogorov PDE \eqref{eq:back_cauchy probl_intro} (such as well-posedness and continuity with respect to the coefficients) combined with classical tools of stochastic analysis such as tightness and uniqueness results of marginal laws to show existence and uniqueness of a martingale solution.
After these early foundational works on SDEs with singular coefficients, people started to explore this topic further: by studying numerical schemes for singular SDEs \cite{deangelis.et.al, Goudenege}; allowing for different type of noises \cite{kremp_perkowski}; studying heat kernel estimates for singular SDEs \cite{perkowski_vanzuijlen}; studying {\em singular kinetic} SDEs, {\em singular McKean-Vlasov} SDEs  and even {\em singular kinetic McKean-Vlasov}  SDE, as we see below.

{\em McKean-Vlasov} equations were originally introduced by McKean in the mid 1960s and have become very popular, in recent years, after the well-known Saint-Flour lecture notes by A.S. Sznitman \cite{Sznitman}, and they 
typically arise as limit of interacting particle systems when the number of particles goes to infinity and when the dynamics of each particle depends on the empirical measure of all other particles. 
Depending on the type of interaction, one obtains different limiting equations.
In the  so-called {moderate interaction}, first introduced by \cite{Oelschlaeger} for smooth drifts, the particles interact through a mollified version of the empirical measure, and the limiting equation turns out to depend on the density of the law in a pointwise way (rather than the law itself). 
A similar McKean SDE whose marginal laws solve the   porous media PDE was provided in \cite{Ben_Vallois}.
  Other McKean SDEs with more general diffusion and drift coefficients
  were investigated by \cite{BRR, BRR2, BCR2, BarbuRockSIAM}.
In the framework of moderate interaction, but with singular drift, and in the special non-degenerate case, the  well-posedness of the limiting {\em singular McKean-Vlasov} SDE \eqref{eq:mkv_singular} has been recently studied in 
\cite{issoglio_russoMK}. More precisely, in \cite{issoglio_russoMK} the authors prove weak existence and uniqueness to the equation \eqref{eq:mkv_singular} in the special case when $ \Bc_0 =0$ and $d=N$, that is  when the vector $(V,X)$ in \eqref{eq:mkv_singular} reduces to the first component $V$, and for $-\beta>-1/2$. 
Similar equations have been studied  
in \cite{chaudru_jabir_menozzi22, chaudru_jabir_menozzi23} (see \cite{hao2023second} for analogous kinetic models) where they analyse a McKean-Vlasov type SDE with singular coefficients with a drift of convolution form $b \ast \mu_{V_t}$ and a general $\alpha$-stable process as noise.
They prove well-posedness of said SDE and heat kernel estimates. Notice that their threshold for the singularity is $-\beta > -1 $ but they are still in the Young regime due to the smoothing effect of the convolution, as opposed to the pointwise product  in \eqref{eq:mkv_singular}, which is the same as in \cite{issoglio_russoMK}.

In the realm of {\em kinetic}-type SDEs, the well-posedness of the \emph{linear} martingale problem, together with Gaussian estimates for the density, have been established in \cite{menozzi2018martingale} and \cite{menozzi}, assuming the coefficients in suitable H\"older spaces induced by an anisotropic norm that takes into account the different time-scaling properties of the degenerate components, namely  the component $X$ of $(V,X)$ according to our notation. In the latter references, the  drift of the degenerate component $X$ may be non-linear and the diffusion coefficient is allowed to be non-constant, provided that it satisfies a uniform ellipticity condition on $\mathbb{R}^d$.
In \cite{chaudru_et_al2023}, the authors 
allow the drift of $V$  to be an  unbounded discontinuous function and derive heat kernel and gradient kernel estimates.
 In \cite{hao2024singular}, the drift of $V$ is a convolution with a kernel that may be a distribution and the authors study the singular martingale problem in order to tackle a convolutional kinetic McKean-Vlasov SDE (see also the next paragraph).
In the {\em kinetic McKean-Vlasov} setting, we mention a few contributions. In \cite{bossy} a conditional Langevin-type dynamics is considered as an alternative approach to Navier-Stokes equations for turbulent flows. The weak well-posedness of the McKean-Vlasov SDE was proved via propagation of chaos, by considering a smoothed SDE and proving that the corresponding interacting particle system propagates chaos. The weak existence for equations with McKean-Vlasov coefficients in the form of $\Eb[b(t,z,Z_t)]$, both in the drift and the diffusion, was recently proved in \cite{pascucci2024existence} for the strictly kinetic case ($N=2d$), under minimal regularity conditions (uniform continuity w.r.t.\ to the position variable in both the state and the measure component) on the function $b$. Similar results were recently proved in \cite{zhang2021second}.

Finally, we  mention \cite{hao2024singular} in which the authors combine {all three aspects  considered in this paper:} {\em singular, kinetic} and {\em McKean-Vlasov}. In particular, they study  a kinetic McKean-Vlasov equation which differs from \eqref{eq:mkv_kinetic_singular} in that the dynamics of the velocity component $V_t$ reads as
\begin{equation}\label{eq:SDEzhang}
\dd V_t = b_t(Z_t) \dd t +( K \ast \mu_{X_t}) (X_t) \dd t   +     \dd W_t ,
\end{equation}
where the drift component is the sum of a 
term 
$b_t$ in a weighted anisotropic Besov space with regularity index $-\beta$ with $\beta \in (1/2, 2/3)$ (hence this is  outside the Young regime and requires the use of tools like paracontrolled calculus) and one convolution part involving the law of the position $X$  with a singular kernel $K$  with regularity index $\alpha > (\beta -1)/3$ (notice that $\alpha$ can be negative). 
They prove existence  of a martingale solution, and under the stronger  assumptions that $K$ is bounded and measurable (thus non-singular) they also prove uniqueness. 
One of the main differences between the setting in \cite{hao2024singular} and that of the present work regards the way the drift depends on the law: (i) in \eqref{eq:SDEzhang} the dependence 
appears through the convolution with a singular term, hence it is linear, while here, as in \eqref{eq:mkv_singular}, we have  a nonlinear function evaluated pointwisely in the law and multiplied by a Schwartz  distribution; (ii) in \eqref{eq:SDEzhang} the dependence on the law, and thus the mean-field interaction of the associated particle system, only involves the position variable $X$. 
 
 Finally we stress that in the present work we derive our results in the general degenerate setting of equation \eqref{eq:mkv_singular} (with $0<d<N$) and not only in the purely kinetic case ($N=2d$).

 As mentioned in Section \ref{subsec:statement}, a key step to prove our results on the SDE \eqref{eq:SDEzhang}
 is the study of the singular non-linear Fokker-Planck equation \eqref{eq:Kc'} and the backward Kolmogorov equation  \eqref{eq:back_cauchy probl_intro}. From the PDE point of view, our results in Sections \ref{sc:FP} and \ref{sec:Kolmogorov PDE} on the well-posedness, regularity and stability of these equations are natural extensions of those in \cite{issoglio_russoMK} and \cite{issoglio2022pde} from the elliptic to the hypoelliptic setting. As the proofs in the latter references rely on Besov-Schauder estimates for the heat semigroup (see \cite{gubinelli2015paracontrolled}), our proofs rely on their anisotropic counterpart for the semigroup of the operators $\Kc$ and $\Kc'$ (defined in \eqref{eq:kolm_const} and \eqref{eq:FP operator kinetic}), for which we provide a complete proof. To the best of our knowledge, in the hypoelliptic framework, these estimates only appeared in \cite{zhang2021cauchy} for the strictly kinetic ($N=2d$) and homogeneous case, namely when the operator $\Kc$ is invariant with respect to a family of dilations. We also mention that our PDE results are somehow complementary to those in \cite{hao2024singular}, where the paracontrolled case ($\beta\in(1/2,2/3)$) was considered,  in the strictly kinetic framework. In general, there is a large body of literature, also very recent, concerned with the regularity properties of degenerate-type PDEs involving $\Kc$ or its adjoint $\Kc'$ as a principal operator. In the H\"older setting, we recognize two
main approaches. In the semigroup approach, initiated by 
\cite{lunardi1997schauder}, solutions are defined in the distributional sense: in this framework, solutions do not benefit from the time-smoothing effect that is typical of parabolic equations, and the regularity structures are typically defined in terms of spatial \emph{anisotropic} H\"older spaces (see Section \ref{sec:holder_anisot_intrinsic}), which do not require H\"older-regularity with respect to the time-variable. With respect to this approach, the anisotropic Besov spaces defined in Section \ref{sec:anis_besov}, and the related Schauder estimates for the semigroups of $\Kc$ and $\Kc'$ in Section \ref{sec:schauder}, are the natural distributional extensions. 
On the other hand, in the stream of research started by 
\cite{pol}, solutions in the Lie sense are defined in terms of directional derivatives along the H\"ormander vector fields. 
In this approach, the regularity of the coefficients and of the solutions is typically stated in so-called \emph{instrinsic} H\"older spaces (see Section \ref{sec:holder_anisot_intrinsic}), where the regularity properties in space and time are strictly intertwined.  Some authors (see \cite{biagi2022schauder}, \cite{dong2022global}, \cite{lucertini2023optimal}) have recently combined these two approaches to obtain intrinsic regularity of the solution (with respect to all the variables) by only assuming anisotropic regularity for the data (and only measurability). We refer to \cite{lucertini2022optimal} and \cite{lucertini2023optimal} for a more exhaustive overview of the literature regarding these two approaches and their combination. 
In particular, we refer to \cite{lucertini2022optimal} and to the previous work \cite{difpas} for the study of fundamental solutions in the degenerate H\"older setting. We note that, although the theory of functional intrinsic spaces has already been developed in the Sobolev setting (e.g. \cite{garofalo2022hardy}, \cite{pascucci2022sobolev}), the distributional Besov counterpart of the intrinsic H\"older spaces has not yet been studied and it is an open problem for future research. This will be necessary to establish sharp regularity estimates for the singular PDEs considered in this paper.

\subsection{Setting and main result}
 Throughout the paper, we set the matrix
\begin{equation}\label{B_tot}
  B:=\begin{pmatrix}
  \Bc_0  \\ \Bc_1
  \end{pmatrix} \in \mathcal{M}^{N\times N},
\end{equation}
where $\Bc_0 \in  \mathcal{M}^{d\times N}$ and $\Bc_1 \in  \mathcal{M}^{(N-d)\times N}$ are the real matrices in \eqref{eq:SDEzhang}. We also denote by $\Delta_v$ and $\nabla_{z}$, respectively, the Laplacian operator on $\R^d$ and the gradient on $\R^{N}$, namely
\begin{equation}
 \Delta_v = \sum_{i=1}^d \partial^2_{v_i}, \qquad \nabla_z = (\partial_{z_1}, \cdots, \partial_{z_{N}}).
\end{equation}
\begin{assumption}\label{ass:kolmogorov_op}
The Kolmogorov operator
\begin{equation}\label{eq:kolm_const}
 \Kc := \Ac +\partial_t := \frac{1}{2} \Delta_{v}  + \underbrace{{\langle B z}, \nabla_{z} \rangle + \partial_t}_{=:Y} ,\qquad  {z= (v, x)}\in \mathbb{R}^N,
\end{equation}
is hypoelliptic on $\R^{N+1}$.
\end{assumption}

\begin{remark}
By H\"ormander theorem, $\Kc$ in \eqref{eq:kolm_const} is hypoelliptic if and only if the vector fields $\partial_{v_1},\dots,\partial_{v_{d}}$ and $Y$ satisfy 
\begin{equation}\label{horcon}
\text{rank } \text{Lie}(\partial_{v_1},\dots,\partial_{v_{d}},Y)  = N+1.
\end{equation}
\end{remark}
\begin{example}\label{ex:kinetic}
The kinetic equation \eqref{eq:mkv_kinetic_singular} can be obtained from \eqref{eq:SDEzhang} by setting $N=2 d$, $m=d$ and 
\begin{equation}
  B:=\begin{pmatrix}
  0 & 0 \\ 
  I_d & 0 
  \end{pmatrix}.
\end{equation}
In this case we have $Y = \partial_t + \langle v , \nabla_x \rangle$ and 
\begin{equation}
[\partial_{v_i}, Y ] = \partial_{x_i}, \qquad i=1,\cdots, d,
\end{equation}
thus \eqref{horcon} is satisfied.
\end{example}
In \cite{lanpol} it was shown that \eqref{horcon}, known as weak (or parabolic) \emph{H\"ormander} condition, is equivalent to $\Bc_1$ having the block-form
\begin{equation}\label{B}
  \Bc_1=\begin{pmatrix}
 B_1 & \ast &\cdots& \ast & \ast \\
 0 & B_2 &\cdots& \ast& \ast \\
 \vdots & \vdots &\ddots& \vdots&\vdots \\
 0 & 0 &\cdots& B_{\rr}& \ast
  \end{pmatrix}
\end{equation}
where the $*$-blocks are arbitrary and $B_j$ is a $(d_{j-1}\times d_j)$-matrix of rank $d_j$ with
\begin{equation}\label{eq:elements_d}
d\equiv d_{0}\geq d_1\geq\dots\geq d_{\rr}\geq1,\qquad \sum_{i=0}^{\rr} d_i=N.
\end{equation}
It is convenient to set 
\begin{equation}
\bar d_i := \sum_{k =0}^i d_k, \qquad i=0,\cdots, \rr, \qquad \text{and} \quad \bar d_{-1}:=0.
\end{equation}

 Next we introduce the notation for function spaces used in this paper. 
For a given $T>0$, $\alpha \in \R$, and for any $(U, \|\cdot\|_U)$ generic normed space endowed with the Borel $\s$-algebra generated by the topology induced by $\|\cdot\|_U$, we introduce the following notations.
\begin{itemize}
\item  $\mathcal{S}$ the space of Schwartz functions and $\mathcal{S}'$ the space of tempered distribution on $\R^N$.
\item $\C_b$  the space of bounded continuous functions on $\R^N$, equipped with the norm $\| \cdot \|_{L^{\infty}}=\| \cdot \|_{L^{\infty}(\R^N)} $. 
\item $\C_c^\infty$  the space of $\C^\infty$ functions on $\R^N$ with  compact support.
\item $\C_B^\alpha$ the anisotropic Besov space of order $\alpha$ and by $ \|\cdot\|_\alpha$ its norm, see Definition \ref{def:anisotrBesov}.
\item $ \Cc^{\alpha}_{B,T}$ the space of intrinsically $\alpha$-H\"older continuous functions for $\alpha>0$, and by $\|\cdot\|_{\Cc^{\alpha}_{B,T}}$ its norm, see  Definition \ref{def:intrinsicHolder}.
 \item ${ L}^{\infty}_{T} U
$,  the space of {Borel-}measurable functions $f:[0,T]\to U$
such that
\[ \|f\|_{T,U}:=\operatorname*{ess\,sup}\limits_{t\in [0,T]} \|f_t\|_U<+ \infty .\]
{We have two special cases: 
\begin{itemize}
\item when $U=\C_B^\alpha$ the norm $\|f\|_{T,U}$ is denoted by $\|f\|_{T,\alpha}$,
\item when $U=\C_b$, the norm $\|f\|_{T,U}$ is denoted by $\|f\|_{T}$.
\end{itemize}}
\item  $\mathit{C}_T U $ 
the space of continuous functions $f:[0,T]\to U$. 
\item  $\mathit{C}_T \mathcal S $ 
the space of continuous functions $f:[0,T]\to \mathcal S$, 
  where $\mathcal S$ is equipped with the usual strong topology.
\end{itemize}

We use the notation $\langle \cdot, \cdot \rangle $ to indicate the scalar product in $\R^N$, and $\langle \cdot \vert\cdot \rangle $ for the distributional dual pairing in $\mathcal S$ and $\mathcal S'$.

We now specify the regularity assumptions on the coefficients $b$, $F$ in \eqref{eq:mkv_singular} and on the density of the initial condition $u^0$ in \eqref{eq:initial_cond_Z}, under which we will prove the main result.
\begin{assumption}
\label{ass:beta and b}
Let $T>0$, $\beta\in (0,\frac{1}{2})$.  We have
\begin{itemize}
\item[(i)] {$b\in \mathit{C}_T \C_B^{-\beta}$}. 

\item[(ii)]  $u^0 \in \C_B^{\beta+\eps}$  for some $\eps \in (0,1-2 \beta)$, and is a probability density on $\mathbb{R}^N$. 
\end{itemize}
\end{assumption}

Let $F:\mathbb{R}\to \mathcal{M}^{d\times m}$ and $\tilde F:\mathbb{R}\to \mathcal{M}^{d\times m}$, defined as 
\begin{equation}\label{eq:FFF}
\tilde F{(s)}:= s F{(s)}.
\end{equation}
We will make use of the following assumptions for  $\Phi= F$ or  $\Phi= \tilde F$.
\begin{assumption}\label{ass:phi}
The function $\Phi:\mathbb{R}\to \mathcal{M}^{d\times m}$ is differentiable, with  bounded and globally Lipschitz-continuous derivative.
\end{assumption}

Postponing until Section  \ref{sec:MKMP}
(Definition \ref{def:MKMP}) the precise definition  of  solution (in law)
to \eqref{eq:mkv_singular} (denoted as solution to   MKMP$( F,b, u^0)$), we formulate the following main result, which will be restated and proved in Theorem \ref{thm:ex_uniq_McKeanSDE}.  
Concerning the regularity of the density of the solution to the MKMP$( F,b,u^0)$, we refer to Remark \ref{rem:regularity_ut}.

\begin{theorem}\label{thm:ex_uniq_McKeanSDE_intro} 
Let Assumptions \ref{ass:kolmogorov_op}  and 
 \ref{ass:beta and b} hold. Let also Assumption \ref{ass:phi} hold with $\Phi=F,\tilde F$. 
Then there exists a unique  solution to  MKMP$( F, b,u^0)$.
\end{theorem}

In order to help the reader better  understand the paper, we now explain 
the key steps and overall strategy for the proof of the main result.  
The idea in a nutshell is to construct 
a reasonably good ``guess" of solution to the McKean problem and then verify that such a guess is the unique solution of our problem. However, the construction of the guess solution is not trivial. 

The {\em first step}, in order to do this, is to view the MKMP$( F,b, u^0)$ as a particular case in a  
class of {\em linear} martingale problems, denoted by  MP$( \Bc, \mu_0)$, where the drift $\Bc$ does not contain any unknowns. Here $\Bc$ must be any given Besov distribution and $\mu_0 $ any initial law.
We will construct our guess solution by solving the linear martingale problem, which is an easier environment to work in. 
 It can be easily verified that a solution of  MKMP$( F, b,u^0)$ is a solution of MP$(F(u)b, u^0(z)dz)$, so in order to move from the McKean-Vlasov framework to the linear framework we need to know the density of the solution itself, because  the unknown density $u$ of \eqref{eq:mkv_singular} has to be given as an input of the linear problem.  
  In the linear case,  by applying the It\^o formula, it is easily verified that the marginals of an SDE are (distributional) solutions of the associated  FP equation.   
 This fact also holds for a solution of MKPM$(F,b,u^0)$, as proved in Proposition \ref{prop:MP sol FP}, even though the associated FP equation is non-linear in this case. The results on well-posedness for the nonlinear FP equation are proved in Section \ref{sc:FP}, and are a first hint about the well-posedness of \eqref{eq:mkv_singular}.
Once we are in the linear framework, the {\em second step} consists of finding a (unique) solution to the singular linear martingale problem, which is the content of Theorem \ref{thm:ex_un_MP}. 
To do so, the singular drift is smoothed 
 and the result \cite[Theorem 2.6]{figalli} allows us to find the solution of the Stroock-Varadhan martingale problem. 
 At this point, it is worth noticing that in the smoothed case  our notion of solution coincides with  the classical one \`a l\`a Stroock-Varadhan, as proven in 
Proposition \ref{prop:MP-Strook}.  To show {\em existence} of a solution to the non-smoothed linear problem we have to 
 verify that the sequence of measures that solves the regularized problem
 is tight, and that the limit of a subsequence is a solution.  To prove tightness (see Lemma \ref{lm:tighness})  we make use of a uniform-boundness result of solutions to a Kolmogorov-type PDE, whose analysis is done in  Section \ref{sec:Kolmogorov PDE}. 
To prove {\em uniqueness}, we extend a well-known uniqueness property of martingale problems to time-inhomogeneous coefficients, which is referred to as `Property M'  in this work (see Appendix \ref{sec:MP Ethier-Kurtz}).   
 
 The {\em final step} is now easy and consists of checking that the solution constructed as described above is indeed the unique solution to the singular McKean-Vlasov  
 {martingale problem}. This can be seen using tightness as well as continuity results of the  Zvonkin-type PDE and the Fokker-Planck PDE, together with the fact that for the smoothed coefficient the result is valid. This is done in Section \ref{sec:MKMP}. 
 
 \vspace{10pt}
 
  The rest of the paper is organized as follows. In Section \ref{sec:analytical_tools} we introduce the anisotropic Besov and H\"older spaces we will work in and the main analytical tools that are needed to study the non-linear Fokker-Planck and the backward Kolmogorov PDEs. In particular we introduce the semigroups associated to the Kolmogorov backward operator and to the Fokker-Planck forward operator and provide key Schauder estimates for both, the proof of which is postponed to the Appendix. 
  In Section \ref{ssc:FP-well} we study the well posedness of the nonlinear Fokker-Planck equation via local fixed point arguments and using Schauder estimates. We also show continuity results for the FP equation in Section \ref{ssc:FP-cont}. Section \ref{sec:Kolmogorov PDE} is devoted to showing well-posedness and continuity results for the linear Kolmogorov equation. Section \ref{sec:linear MP} contains the study of the linear martingale problem and Section  \ref{sec:MKMP} the well posedness of the McKean-Vlasov problem. Appendix \ref{sec:MP Ethier-Kurtz} contains the extension of the uniqueness results for time-inhomogeneous martingale problems knowing the uniqueness of the marginals  and Appendix  \ref{app:proof_Schauder} the proof of Schauder estimates.

\section{Analytical tools}\label{sec:analytical_tools}

 In this section we provide the main analytical results that are needed for the study of the singular non-linear Fokker-Planck PDE \eqref{eq:Kc'} and backward Kolmogorov PDE \eqref{eq:back_cauchy probl_intro}. We first introduce the relevant anisotropic Besov spaces and recall some of their basic properties (Section \ref{sec:anis_besov}), then we state the inherent Schauder estimates for the semigroup kernels of $\Kc$ and $\Kc'$ (Section \ref{sec:schauder}). We then recall the definition of anisotropic H\"older spaces, their characterization as Besov spaces with non-integer positive regularity index, and briefly compare them with their intrinsic counterparts (Section \ref{sec:holder_anisot_intrinsic}). Finally we state some useful regularization results (Section \ref{sec:regul}).

Hereafter we will always assume that Assumption \ref{ass:kolmogorov_op} is satisfied and thus $ \Bc_1$ admits a representation as in \eqref{B}-\eqref{eq:elements_d}, which factorizes $\R^N$ in $r+1$ components of dimensions $d_0\geq \dots \geq d_r\geq 1$. Accordingly, we will employ the notation $z=(z_0,\cdots, z_r)\in\R^N$ with each $z_i\in\R^{d_i}$.

\subsection{Anisotropic Besov Spaces}\label{sec:anis_besov}
 For the reader's convenience, we recall the classical Fourier-based construction of anisotropic distributional Besov spaces together with some useful results. We focus on the particular case of integrability and microscopic parameters equal to infinity  (typically denoted by $p$ and $q$ in the literature), as this work only covers this case. The interested reader can refer to \cite[Chapter 5]{triebel06} for the general definition, and to the references therein for an account on the historical development of these spaces.

Let us define the so-called \emph{anisotropic norm} as
\begin{equation}
|z|_B :
= \sum_{i=0}^{\rr} |z_i|^{\frac{1}{2 i + 1}}, \qquad z = (z_0, \cdots, z_{\rr}) \in \R^N,
\end{equation}
where $|z_i|$ is the Euclidean norm in $\R^{d_i}$.
It is immediate that 
$|\cdot|_B$ is not strictly a norm as it is not homogeneous with respect to usual scalar multiplication. However, homogeneity holds for the family of dilations $D_\lambda: \R^N \longrightarrow \R^N$, $\lambda>0$, defined as
\begin{equation}\label{eq:dilation_lambda}
\qquad D_{\lambda}z = \lambda . z  :=  ( \lambda z_0,  \lambda^{3} z_1  ,  \ldots, \lambda^{2\rr +1} z_{\rr}), \qquad z = (z_0, \cdots, z_{\rr})\in\R^N.
 \end{equation}
In this work $\BB_{\tau}(z)$ denotes the closed anisotropic disk, or ball, centred at $z\in \R^N$ with radius $\tau>0$, i.e. $ \BB_{\tau}(z):=\{ y\in \R^N : \; |y-z |_B\leq \tau\}$. The ball centred at the origin is denoted by $\BB_{\tau}:= \BB_{\tau}(0)$.
Moreover, we define the anisotropic annuli
\begin{equation}
 \A_{\tau,\tau'}:=\{x\in \R^N :\; \tau \leq |x|_B \leq \tau' \} , \qquad 0<\tau<\tau'.
\end{equation}
 We finally denote by $Q$ the so-called \emph{homogeneous dimension}, set as \begin{equation} 
  \label{eq:homog dim}
 Q := \sum_{i=0}^{\rr} d_i (2i + 1). 
 \end{equation}
As in the isotropic case, the anisotropic spaces are defined through a Paley-Littlewood decomposition. 
We start by considering a smooth, non-negative function $\rho_{-1}$ on $\R^N$ such that
\begin{equation}\label{eq:def_rho_minus_one}
\rho_{-1} \equiv 1 \text{ on } \BB_{\frac{1}{2}}, \qquad  \text{supp}\rho_{-1} \subseteq \BB_{\frac{2}{3}} \quad \text{  and  } \int_{\R^N}\rho_{-1}(\xi) d\xi=1.
\end{equation}
The sequence of functions $\{ \rho_j\}_{j\geq 0}$, defined by 
\begin{align}
 \rho_j(\xi):= \rho_{-1}(2^{-(j+1)}.\xi)-\rho_{-1}(2^{-j}.\xi),  \qquad \xi\in\R^N, 
 \end{align}
is called an anisotropic \emph{Paley-Littlewood partition}. 
\begin{remark}\label{rem:paley-littlewood}
By definition we have the following properties: 
\begin{align}
(a)&\quad\rho_j(\xi)  =\rho_0(2^{-j}.\xi), \qquad j\geq 0,\\
(b)&\quad\text{supp} \rho_j \subseteq 
\A_{2^{j-1},2^{j+1}}, \qquad j\geq 0,  \label{rmk:point (b) support partition} \\
(c)&\quad\text{supp} \rho_i  \cap \, \text{supp} \rho_j \neq \emptyset \Longleftrightarrow \, |i-j|\leq 1, \qquad i,j\geq-1,\\
(d)&\quad \sum_{j=-1}^{n} \rho_j (\xi) = \rho_{-1}\big(2^{-(n+1)}.\xi\big) ,\qquad  n\geq -1,\\
(e)&\quad \sum_{j=-1}^{+\infty} \rho_j \equiv 1 \text{  on } \R^N.
\end{align}
\end{remark}
For any $f\in\Leb^1(\R^N)$  we employ the following definitions for the Fourier transform and its inverse:
\begin{align}
\F(f)(\xi)&=\fh(\xi):=\left(2\pi\right)^{-\frac{N}{2}}\int_{\R^N} e^{-i\< z, \xi \>} f(z) dz,\\
\F^{-1}\left(f\right)(z)&=\check{f}(z):=\left(2\pi\right)^{-\frac{N}{2}}\int_{\R^N} e^{i\< z, \xi \>} f(\xi) d\xi. 
\end{align}
\begin{remark}
It can be directly verified that the scaling property of the Fourier transform for the $\{ \rho_j\}_{j\geq 0}$ reads as 
\begin{equation}
\label{eq: scaling propr part}
2^{jQ}
\check{\rho_0} \circ D_{2^{j}}  = \F^{-1}( \rho_0 \circ D_{2^{-j}} )  =  \check{\rho_j}, \qquad j\geq 0.
\end{equation}
Another important property is
\begin{equation}
\label{eq: mean propr part}
\left(2\pi\right)^{-\frac{N}{2}} \int_{\R^N} \check{\rho}_{-1}(z)dz=1, \qquad \left(2\pi\right)^{-\frac{N}{2}} \int_{\R^N} \check{\rho}_j(z)dz=0, \quad j\geq 0,
\end{equation}
which stems from
\begin{equation}
\left(2\pi\right)^{-\frac{N}{2}} \int_{\R^N} \check{\rho_j}(z) dz =\F\left( \F^{-1} \rho_j\right)(0) = \rho_j(0).
\end{equation}
\end{remark}
The Fourier transform of $f\in \mathcal{S}'$ is defined, as usual, by duality, i.e.
\begin{equation}\label{eq:fourier_distr}
\<\F(f)|\,\varphi\>=\<\fh|\,\varphi\>:= \<f|\, \widehat\varphi \>, \qquad \varphi \in \mathcal{S},
\end{equation}
and the same for $\F^{-1} = \check{f}$.
For $f \in \mathcal{S}'$, 
the $j$-th Paley-Littlewood block of $f$ is defined as
\begin{equation}\label{eq:fourier_Delta_j}
\Delta_j f:= \F^{-1}\big( \rho_j \fh\ \big),\qquad j\geq -1.
\end{equation}
\begin{remark}\label{rem:paley_little}
By the properties of the Fourier transform, we have 
\begin{equation}
 \Delta_j f =  \check{\rho_j}* f ,
\end{equation}
and in particular, $\Delta_j f$ is smooth. On the other hand, it is easy to check that $\sum_{j= - 1}^n \Delta_j f$ tends to $f$ in $\mathcal{S}'$ as $n\to \infty$.
\end{remark}

The following result, whose proof is completely analogous to that of \cite[Lemma 2.1]{bahouri2011fourier} in the isotropic case, is fundamental to comprehend the construction of Besov spaces used in the paper.
\begin{theorem}[Bernstein's type inequality]
For every $k\in \Nzero$, there exists a positive constant $C=C(k
)$ such that
\begin{equation}
\label{eq: general Bern inequality}
\sum_{l=\bar{d}_{i-1}+1}^{\bar{d}_i}\|\partial^k_{z_l} f\|_{\Leb^\infty}\leq C \lambda^{ (2i+1) k }\|f\|_{\Leb^\infty},\qquad \lambda>0, \quad i=0,\cdots, r,
\end{equation}
for any $f\in \Leb^\infty(\R^N)$ such that $\text{\emph{supp}}
 \fh \subseteq \BB_{ \lambda 
 }$. Note that, in particular, $C$ does not depend on $\lambda$.
\end{theorem}
By applying this result to the Paley-Littlewood partition, we have the
following.

\begin{corollary}
For any $f\in \mathcal{S}'$ and $k\in\N_0$, we have
\begin{equation}
\label{eq: Bernstein ineq part}
\sum_{l=\bar{d}_{i-1}+1}^{\bar{d}_i} \| \Delta_j (\partial^k_{z_l} f)\|_{\Leb^\infty}\leq C\, 2^{j(2i+1) k} \|\Delta_j f\|_{\Leb^\infty},  \qquad i=0,\cdots, r, \qquad j\geq -1,
\end{equation}
where $C=C(k)>0$.
\end{corollary}
\begin{proof}
 By \eqref{eq:fourier_Delta_j} we have 
\[ \F \left(\Delta_j f\right) = \rho_j  \fh, \qquad j\geq -1,\]
whose support is contained in $\BB_{2^{j+1}}$ by \eqref{rmk:point (b) support partition}.
If $\|\Delta_j f\|_{\Leb^\infty}=\infty$, there is nothing to prove. Otherwise, since the decomposition operators $\Delta_j$ commute with the derivatives, the result follows from \eqref{eq: general Bern inequality}.
\end{proof}
We can now give the definition of anisotropic Besov space.
\begin{definition}\label{def:anisotrBesov}
Given $\gamma\in \R$, the space $\C_B^\gamma$ of the distributions $f\in \mathcal{S}'$ such that
\begin{equation}
\|f\|_{\gamma}:= \sup_{j\geq -1} \(2^{j\gamma}\| \Delta_j f\|_{\Leb^\infty}\) <\infty
\end{equation}
is called \emph{anisotropic Besov space of order $\gamma$}.  In this context, $\gamma$ is called the \emph{regularity parameter}.
\end{definition}
It is straightforward to see that
\[\alpha<\gamma \Longrightarrow\C_B^\gamma \subset \C_B^\alpha.\]
Also, by applying inequality \eqref{eq: Bernstein ineq part}, 
for any $\gamma\in\R$ and $k\in\N_0$ we have
\begin{equation}
\label{eq: Bern ineq Besov norm}
\sum_{l=\bar{d}_{i-1}+1}^{\bar{d}_i}\|\partial^k_{z_l}  f\|_{\gamma}\leq C\| f\|_{\beta+(2i+1) k}, \qquad i=0,\cdots, r.
\end{equation}
The quantity $(2i+1)$ represents the reduction in terms of regularity whenever a derivative along the $i$-th block is required.
\begin{example}\label{ex:Linfinity}
We have that $\Leb^\infty(\R^N) \subset \C^0_B$. Indeed, 
for any $f\in \Leb^\infty(\R^N)$, Young's inequality yields
\begin{equation}
\|\Delta_j f\|_{L^{\infty}}\leq \|\check{\rho_j}\|_{L^1} \|f\|_{L^{\infty}}.
\end{equation}
Also, by \eqref{eq: scaling propr part} we have $\|\check{\rho_j}\|_{L^{1}}=\|\check{\rho_0}\|_{L^{1}}$, and thus 
\begin{equation}
\|f\|_0=\sup_{j\geq -1} \|\Delta_j f\|_{L^{\infty}} <\infty.
\end{equation}
\end{example}
\begin{example} The Dirac delta $\delta_0$ of $\R^N$ is an element of $\C^\gamma_B$ 
if and only if $\gamma\leq-Q$, where $ Q $ denotes the homogeneous dimension defined in \eqref{eq:homog dim}. 
To see this, note that
\begin{equation}
\Delta_j \delta_0= \check{\rho_j}*\delta_0=\check{\rho_j}, \qquad j\geq -1.
\end{equation} 
Therefore, by \eqref{eq: scaling propr part} we obtain
\begin{equation}
\sup_{j\geq 0} \left( 2^{j\gamma}\|\Delta_j \delta_0\|_{L^{\infty}}\right)=\sup_{j\geq 0}\left(2^{(j(\gamma+Q)} \| \check{\rho_0}\|_{L^{\infty}}\right),
\end{equation} 
which is finite if and only if $\gamma\leq -Q$.
\end{example}

 To make sense of the product appearing in the MKV SDE \eqref{eq:mkv_singular} and in the PDEs \eqref{eq:Kc'} and \eqref{eq:back_cauchy probl_intro}, we will consider the so-called \emph{Bony's product} (see \cite{bony1981interaction}) between Besov distributions, which extends the product between functions (which are also tempered distributions). Intuitively, this is defined if we can make sense of the formal product
\begin{equation}\label{eq:formal_product}
  f g = \sum_{j,i= - 1}^{\infty} \Delta_j f  \Delta_i g,
\end{equation}
which is possible if the sum between the regularity parameters of $f$ and $g$ is strictly positive. As in the isotropic framework, this is achieved by bounding separately the \emph{paraproduct} and the \emph{resonant} components of \eqref{eq:formal_product} (see \cite[Section 2.1]{gubinelli2015paracontrolled}). In particular, the following result can be proved by proceeding exactly as in \cite[Lemma 2.1]{gubinelli2015paracontrolled} (see also \cite{hao2024singular}).

\begin{proposition}\label{prop:bony_prod}
Let $\alpha,\gamma \in \R$  
with  
$\alpha+\gamma >0$.  
 Then Bony's product restricted to   $\C_B^\alpha \times \C_B^\gamma$ is a bilinear form from $\C_B^\alpha \times \C_B^\gamma$ to $\C_B^{\alpha \wedge \gamma}$ and 
there exists a positive constant $C
$ such that
\begin{equation}
\label{eq: prod estim}
\| fg\|_{\alpha \wedge \gamma}\leq C\, \|f\|_\alpha \|g\|_\gamma, \qquad f\in \C_B^\alpha,\ g\in \C_B^\gamma.
\end{equation}
\end{proposition}

\subsection{Kinetic-type semigroups and Schauder's estimates}\label{sec:schauder}

In this section we present Besov-Schauder's estimates for the fundamental solutions of the hypoelliptic operators $\Kc$ in \eqref{eq:kolm_const} and of its formal adjoint $\Kc'$ defined below, which appear in \eqref{eq:back_cauchy probl_intro} and \eqref{eq:Kc'}, respectively. Note that $\Kc$ is the Kolmogorov operator associated to the $\R^N$-valued linear SDE
\begin{equation}\label{eq:SDEgeneralLangevin}
\dd Z_t = B Z_t \dd t+ \sigma \dd W_t,
\end{equation}
with 
\begin{equation}\label{eq:sigma}
\sigma = \begin{pmatrix}
  I_d  \\ 
  0_{N-d,d}
  \end{pmatrix} \in \mathcal{M}^{N\times d},
\end{equation}
where $I_d$ and $0_{N-d,d} $
denote, respectively, the $d\times d$ identity
matrix and the $(N-d)\times d$ matrix with null entries.  We recall that $B$ is as in \eqref{B_tot}-\eqref{B}-\eqref{eq:elements_d}, so that Assumption \ref{ass:kolmogorov_op} is satisfied.
Note that \eqref{eq:SDEgeneralLangevin} coincides with \eqref{eq:mkv_singular} upon posing the McKean-Vlasov coefficient $ F\big( u_t(Z_t)  \big) b_t (Z_t)$   equal to zero. 
It is  known since \cite{hormander1967hypoelliptic} (see also \cite[Section 9.5]{pascucci} for a recent presentation) that $\Kc$ admits a unique smooth fundamental solution $\Gamma(T-t;y,z)$, $t<T$, $y,z\in\R^N$.  Namely, the function $(t,y)\mapsto \Gamma(T-t;y,z) $ solves $\Kc u = 0 $ on $(-\infty,T)\times \R^N$ and 
\begin{equation}
\Gamma(T-t;\eta,z) dz {\longrightarrow} \delta_y(dz) \quad \text{weakly as } (t,\eta)\to (T^-,y).
\end{equation}
In particular, 
$\Gamma$ is explicitly given by
\begin{equation}
\label{eq:fund sol general langevin}
\Gamma(s;y,z)=\Gamma_{s}(z-e^{s B}y), \qquad s>0, \quad z,y \in \R^N,
\end{equation}
with
\begin{equation}
\label{eq:Gamma general Langevin}
\Gamma_s(z):=\frac{(2\pi)^{-\frac{N}{2}}}{\sqrt{\det \Cc (s)}} \exp\( -\frac{1}{2}\< \Cc^{-1} (s)z,z\>
 \),
\end{equation}
and
\begin{equation}
\Cc (s):= \int_0^s e^{\tau B}A \left(e^{\tau B}\right)^\top d\tau, \qquad A:= \begin{pmatrix}
I_d & 0 \\
0 & 0 
\end{pmatrix}.
\end{equation}
Note that $\Cc (s)$ is strictly positive definite if (and only if) Assumption \ref{ass:kolmogorov_op} holds. 
 It is also known that $\Gamma(t-s;y,z)$, $t>s$, is the fundamental solution (with respect to the forward variables $t,z$) of the 
Fokker-Planck operator $\Kc'$ given by
\begin{align}
\label{eq:FP operator kinetic}
\Kc' &:=  - \partial_t + \Ac' \\ \label{eq:FP operator kinetic_bis}
 \Ac' &:=  \frac{1}{2}\Delta_v  -\<\nabla_z, Bz \> = \frac{1}{2}\Delta_v  -\<Bz,\,\nabla_z\>-Tr(B)  ,\qquad   z= (v, x)\in \mathbb{R}^N. 
\end{align}  
This means that the function $(t,z)\mapsto \Gamma(t-s;y,z) $ solves $\Kc' u = 0 $ on $(s,+\infty)\times \R^N$ and 
\begin{equation}\label{eq:delta_solfund}
\Gamma(t-s;y,\eta) dy {\longrightarrow} \delta_z (dy) \quad \text{weakly as } (t,\eta)\to (s^+,z).
\end{equation}
To fix the ideas, this implies that the function
\begin{equation}
\label{eq:classical PF cauchy solution}
v(s;t,z):= \int_{\R^N} \Gamma(t-s;y,z)\varphi(y)dy - \int_{s}^t  \int_{\R^N} \Gamma(t-\tau;y,z)g(s,y) \dd y \dd \tau ,\qquad t>s,\ z \in\mathbb{R}^N,
\end{equation}  
is a smooth solution to the Cauchy problem
\begin{equation}\label{eq:example_degenerate_pdes}
\begin{cases}
\Kc' v(s;\cdot,\cdot) = g, \qquad \text{on $(s,+\infty)\times \mathbb{R}^N$}, \\
v(s;s,\cdot) = \varphi ,
\end{cases}
\end{equation}
for any $\varphi \in C_b(\mathbb{R}^N)$ and $g
$ sufficiently smooth. An analogous representation holds for the solution to the backward Cauchy problem for $\Kc$.

In light of this, we want to extend the action of the fundamental solution to tempered distributions. For $t>0$, we consider  $P_t,P'_t: \mathcal{S} \to \mathcal{S}$ 
acting as
\begin{equation}
\label{eq:def semigr schwartz}
P_t  \varphi(y)=\int_{\R^N} \Gamma(t;y,z)\varphi(z)dz,\qquad P'_t  \varphi(z)=\int_{\R^N} \Gamma(t;y,z)\varphi(y)dy,
\end{equation}
which define semigroups. 
By duality, we then extend these semigroups $(P_t)_{t>0}$ and $(P'_t)_{t>0}$ of $\Kc$ and $\Kc'$, respectively, as the family of linear operators $P_t,P'_t:\mathcal{S}'\rightarrow \mathcal{S}'$ 
acting as
\begin{equation}\label{eq:def semigroup tempered}
 \<P_t  f |\, \varphi\> := \<f|\, P'_t \varphi \>, \qquad \<P'_t  f |\, \varphi\> := \<f|\, P_t \varphi \>, \qquad \varphi\in \mathcal{S},\quad f\in \mathcal{S}'.
\end{equation}

\begin{theorem}[Schauder's estimates]\label{th:schauder_first}
For any $\gamma \in \R$, $\alpha\geq 0$ and $T>0$ there exists a positive constant $C= C(\alpha, \gamma ,T)$ such that, for any $f\in \C_B^\gamma$, we have
\begin{equation}
\label{eq:schauder 1}
\|P_tf\|_{\gamma+\alpha} + \|P'_tf\|_{\gamma+\alpha} \leq C\, t^{-\frac{\alpha}{2}} \|f\|_{\gamma}, \qquad    t\in (0,T].
\end{equation}
\end{theorem} 
 The proof of Theorem \ref{th:schauder_first} is postponed to Appendix \ref{app:proof_Schauder}.
These estimates generalize those in \cite{gubinelli2015paracontrolled} for the semigroup of the heat operator, and those in \cite{zhang2021cauchy} for the strictly kinetic ($N=2d$) and homogeneous case, namely when the operator $\Kc$ is invariant with respect to the family of dilations in \eqref{eq:dilation_lambda}.

\subsection{Anisotropic and intrinsic H\"older spaces}\label{sec:holder_anisot_intrinsic}

As mentioned in the introduction, the anisotropic Besov spaces defined above represent a natural extension of the anisotropic H\"older spaces typically used in the study of $\Kc$ and $\Kc'$-type operators according to the semi-group approach. For $\gamma > 0$, a full characterization of $\C^\gamma_B$ in terms of an anisotropic H\"older space can be found in \cite[Lemma 2.8]{zhang2021cauchy}. To illustrate the idea, we provide a representation of such a characterization for some ranges of $\beta$.
\begin{definition}
  We denote by $\Cc^\gamma_B$ the linear space
  of 
  functions $f:\R^N\longrightarrow \R$ such that the following quantity
  is
  finite:
\begin{align}
\|f\|_{\Cc^\gamma_B}:= \begin{cases}
 \| f \|_{{L^{\infty}}}+ \sup\limits_{z\in \R^N}\sup\limits_{h\in\R^N
}\frac{|f(z+h)-f(z)|}{|h|_B^{\gamma}}, &\quad \gamma\in (0,1] ,\\
 \| f \|_{{L^{\infty}}}+ \sup\limits_{z\in \R^N}\sup\limits_{ h\in\R^{N-d} }\frac{|f(z+(0,h))-f(z)|}{|(0,h)|_B^{\gamma}}+ \sum\limits_{l= 1}^{d}\|\partial_{z_l}  f\|_{\Cc^{\gamma-1}_B}, &\quad \gamma\in (1,3], \\
 \| f \|_{{L^{\infty}}}+ \sup\limits_{z\in \R^N}\sup\limits_{h\in\R^{N-d-d_1}}\frac{|f(z+(0,0,h))-f(z)|}{|(0,0,h)|_B^{\gamma}} + \sum\limits_{l= 1}^{d}\|\partial_{z_l}  f\|_{\Cc^{\gamma-1}_B} + \sum\limits_{l= 1}^{d_1}\|\partial_{z_{d+l}}  f\|_{\Cc^{\gamma-3}_B}, &\quad \gamma\in (3,4].
\end{cases}
\end{align}
Previous quantity is a norm
  and $\Cc^\gamma_B$ is a Banach space.
\end{definition}
Note that, up to order $\gamma=3$, the only derivatives appearing are with respect to the first $d$ space-variables. The first-order derivatives with respect to the space variables in the $d_1$-block only appear when $\gamma>3$. We refer the reader to \cite[Subsection 2.1]{zhang2021cauchy} for the definition of $\Cc^\gamma_B$
for an arbitrary $\gamma>0$. 
The basic idea is that a derivative with respect to a space-variable within the $j$-th block is counted as a $(2j+1)$-th order derivative. 
\begin{remark}\label{rm:holder2}
If $\alpha>\gamma>0$ then the space ${\mathit C}^\alpha_B$ is topologically embedded in $\Cc^\gamma_B$. 
\end{remark}
The relation between the Besov spaces $\C^\gamma_B$ and $\Cc^\gamma_B$ is as follows (see \cite[Lemma 2.8]{zhang2021cauchy}, which is an adaptation of  \cite[Proposition 1.3.2]{danchin2005fourier}, see also \cite{chemin}). 
\begin{proposition}
\label{prop:Besov-hold}
{For any non-integer $\gamma>0$, 
$\Cc_B^\gamma=\C_B^\gamma$ and their norms are equivalent.}
\end{proposition}
\begin{remark}\label{rm:holder}
We observe that, if $f\in \Cc^\gamma_B$ with $\gamma\in(0,1]$, then  there exist two positive constants $C$ and $\nu \in (0, \gamma)$ (small enough depending on $B$)  such that 
$$\|f\|_{\C^\nu}:=
\| f \|_{L^{\infty}}+ \sup_{z \neq z' \in \R^N, \ |z-z'|\leq 1}\frac{ |f(z)- f(z')|}{|z-z'|^\nu} \leq C \|f\|_{\gamma}.$$
In other words, $f$ is $\nu$-H\"older continuous in the classical sense (see \cite[Definition 1.49]{bahouri2011fourier}).
\end{remark}

The anisotropic H\"older spaces above can be somehow ``extended" to functions defined on $\R\times\R^N$. These are the so-called \emph{intrinsic} H\"older spaces mentioned in the introduction, which are the natural choices to include the regularity along the time-variable in the study of degenerate $\Kc$ and $\Kc'$-type operators. 
In this framework, the correct translation to consider is not the Euclidean one used in the parabolic case, but the one with respect to the integral curves of $Y$. Intuitively, this follows from the fact that $Y$ plays the role played by the partial derivative $\partial_t$ in the parabolic case. For the same reason, the Lie derivative along $Y$, i.e.
\begin{equation}
\label{eq:def Lie der}
Yu (t,z):=\lim_{h\to 0} \frac{u(t+h, e^{hB}z)-u(t,z)}{h}
\end{equation}
 counts as a second-order derivative. 
 Once more, we provide the definition of these spaces only for some ranges of $\gamma$.  
\begin{definition}\label{def:intrinsicHolder}
For $T>0$, 
we denote by $\Cc_{B,T}^\gamma$ the space of 
functions $u:[0,T]\times\R^{N}\longrightarrow \R$  such that the following {exists and it} is finite {(in particular it is a norm)}:
\begin{align}
\|u\|_{\Cc^\gamma_{B,T}}:= \begin{cases}
\sup\limits_{(t,z)\in[0,T]\times\R^{N}} |u(t,z)|+ \sup\limits_{(t,z),(s,y)\in [0,T]\times\R^{N}} \frac{|u(t,z)-u(s,y)|}{|t-s|^{\gamma/2}+|z - e^{(t-s) B} y|_B^{\gamma}}  , &\quad \gamma\in (0,1] ,\\
 \sup\limits_{(t,z)\in [0,T]\times \R^N} |u(t,z)|+ \sup\limits_{\substack{(t,z)\in [0,T]\times\R^{N}\\ s\in[0,T]}} \frac{|u (s, e^{(s-t)B} z )-u(t,z)|}{|s-t|^{\gamma/2}} + \sum\limits_{l= 1}^{d}\|\partial_{z_l}  u\|_{\Cc^{\gamma-1}_{B,T}}   , &\quad \gamma\in (1,2], \\
\sup\limits_{(t,z)\in [0,T]\times \R^N} |u(t,z)| + \sum\limits_{l= 1}^{d}\|\partial_{z_l}  u\|_{\Cc^{\gamma-1}_{B,T}} + \| Y  u\|_{\Linf{\gamma-2}}   , &\quad \gamma\in (2,3].
\end{cases}
\end{align}
In the case $\gamma\in(2,3]$, we also require that $Yu$, which is the Lie derivative of $u$  in the sense of \eqref{eq:def Lie der}, is continuous on $[0,T]\times \R^N$. Note that, by the boundedness of the norm, $u, \partial_{z_i}u, \partial_{z_i z_j} u$ are uniformly continuous on $[0,T]\times \R^N$.  
We say that these functions are instrinsically $\gamma$-H\"older continuous.
\end{definition}

Obviously, by Proposition \ref{prop:Besov-hold}, for any $\gamma\in (0,1)$ we have 
\begin{equation}\label{eq:intrinsic_inclusion}
\Cc_{B,T}^\gamma \subset \CT{\gamma}.
\end{equation}
A full characterization of the intrinsic H\"older spaces at any order was given in \cite{pagliarani2016intrinsic} (see also the references therein for previous characterizations), where an intrinsic Taylor formula was proved. The latter allows to see that \eqref{eq:intrinsic_inclusion} actually holds for any non-integer $\gamma>0$.

\subsection{Regularizations}\label{sec:regul}
Let $\Phi$ be a standard mollifier on $\R^N$ with compact support. For $n\in \N$ we set $\Phi_n(z)=n^N \Phi(nz)$. For $\gamma\in \R$, $g\in \C^{\gamma}_B$, set the family
\begin{equation}
\label{eq:b_n bis}
g^{(n)}:=\Phi_n * g, \qquad n\in \N.
\end{equation}  

\begin{lemma}\label{lem:regul_g}
  Let $T>0$, $g\in \Linf{\gamma}$ with $\gamma \in\R$ (respectively $g\in    C_T \C^\gamma_B $),
  and denote by $g^{(n)}$ the function $([0,T]\ni t\mapsto g^{(n)}_t)$ given by \eqref{eq:b_n bis}. Then we have the following.
\begin{itemize}
\item[(i)]
For any $n\in\N$ and $\alpha>0$, we have  $g^{(n)}\in \Linf{\alpha}$ (respectively $g^{(n)}\in  \CT{\alpha}$). 
\item[(ii)] For any $\eta\in (0,1)$, there exists $C=C(
\gamma,\eta)>0$ such that 
\begin{equation}\label{eq:bound_gn_g}
\|g^{(n)}-g\|_{T,\gamma-\eta}
\leq C \|g\|_{T , \gamma} \int_{\R^N} \Big{|} \frac{y}{n}\Big{|}_B^\eta  \Phi(y)  dy, \qquad n\in\N,
\end{equation}
and in particular,
\begin{equation}
 \label{eq:conv sequence}
 \|g^{(n)} - g \|_{T,\gamma-\eta} \rightarrow 0, \qquad \text{as } n\to \infty.
\end{equation}
\item[(iii)] In the case
  $g\in  C_T \C_b $ 
  with  compact support, then $g^{(n)} \in C_T \mathcal S$ for any $n\in\N$.
\end{itemize}
\end{lemma}
To prove Lemma \ref{lem:regul_g}, we need the following result, which is a particular case of \cite[Corollary 2.9]{hao2024singular}.
\begin{lemma}
Let $g\in \C^{\gamma}_B$ with $\gamma\in\R$, and 
$\eta\in ( 0, 1)$. There exists a constant $C=C(
\gamma,\eta)>0$ such that 
\begin{equation}
\label{eq:hold between besov spaces}
\|\delta_h g\|_{\gamma-\eta}\leq C |h|_B^\eta \|g\|_{\gamma}, {\qquad h\in\R^N,}
\end{equation}
with $\delta_h g \in \C_B^{\gamma}$ being defined through
 \begin{equation}
 \langle \delta_h g | \varphi \rangle := \langle  g| \varphi(\cdot-h)- \varphi \rangle , \qquad \varphi\in \mathcal{S}(\R^N).
\end{equation}
\end{lemma}

\begin{proof}[Proof of Lemma \ref{lem:regul_g}]
 Let $n\in\N$ and $t\in[0,T]$ be fixed throughout the proof and dropped for ease of notation. 
 To prove Part (i) it is sufficient to show that all $g^{(n)}$ are smooth in space with bounded derivatives uniformly in $t$.
Smoothness is proved as in \cite[Theorem 4.1.1]{hormander2015analysis}.   
We need to prove that for every multi-index $k\in \N^N$
\begin{equation}
|\partial^k g^{(n)}(z)|=|g*\partial^k \Phi_n(z)|=|\langle g|\, \partial^k \Phi_n(z-\cdot)\rangle|<\infty , \qquad z\in\R^N.
\end{equation}
If $\varphi_z(y):=\partial^k \Phi_n(z-y)$, it follows from Remark \ref{rem:paley-littlewood}-(c) that
\begin{align}
\partial^k g^{(n)}(z)&=\sum_{i,j\geq -1} \langle \Delta_ig |\, \Delta_j \varphi_z \rangle= \sum_{i\geq -1} \langle \Delta_ig |\, \tilde{\Delta}_i \varphi_z \rangle
\end{align}
with
\begin{equation}
\tilde{\Delta}_i \varphi_z:=
\begin{cases}
(\Delta_{-1}+\Delta_0)\varphi_z,&  i=-1,\\
(\Delta_{i-1}+\Delta_i+\Delta_{i+1}) \varphi_z,&  i\geq 0 .
\end{cases}
\end{equation}
Therefore, using H\"older inequality we get
\begin{equation}\label{eq:est_gn_z}
|\partial^k g^{(n)}(z)|\leq \sum_{i\geq -1} \| \Delta_i g \|_{L^{\infty}} \|\tilde{\Delta}_i \varphi_z \|_{L^1} =  \sum_{i\geq -1} 2^{i\gamma}\| \Delta_i g \|_{L^{\infty}} 2^{-i\gamma}\|\tilde{\Delta}_i \varphi_z\|_{L^1} \leq \|g\|_\gamma \sum_{i\geq-1}2^{-i\gamma}\underbrace{\|\tilde{\Delta}_i \varphi_z\|_{L^1}}_{= \|\tilde{\Delta}_i \varphi_0\|_{L^1}} .
\end{equation}
Now, the latter sum is finite as it is bounded by the dual Besov norm $\sum_{i\geq -1}2^{-i\gamma}\| \Delta_i \varphi_0 \|_{\Leb^1}$, which is finite due to
$\varphi\in C^{\infty}_c$.
Respectively, under the additional assumption $g\in    C_T \C^\gamma_B $, with $k=0$ in \eqref{eq:est_gn_z} and replacing $g^{(n)}$ with $g^{(n)}_t - g^{(n)}_s$ yields $g^{(n)}\in    \mathit{C}_T \C_b$.

We now prove Part (ii). For any $j\geq -1$ 
we have
\[ 
\Delta_j g^{(n)}_t=\check{\rho}_j* \Phi_n * g_t= \Phi_n *\check{\rho}_j*g_t= \Phi_n *\Delta_j g_t, 
\] 
and thus 
\begin{align}
\big( \Delta_j\, (g^{(n)}_t-g_t) \big) (z)&=\int_{\R^N} (\Delta_j g_t)(z-y) \Phi_n(y) dy- \int_{\R^N} (\Delta_j g_t)(z) \Phi(y) dy\\
&=\int_{\R^N} \underbrace{\big( (\Delta_j g_t) (z-y/n)-(\Delta_j g_t)(z)\big)}_{=\, \delta_{y/n}( \Delta_j g_t)(z)} \Phi(y) dy , \qquad z\in\R^N,
\end{align}
which yields 
\begin{align}
\| \Delta_j\, (g^{(n)}_t-g_t )\|_{{L^{\infty}}}  & \leq \int_{\R^N}\| \delta_{y/n}( \Delta_j g_t)\|_{{L^{\infty}}}  \Phi(y)  dy
\intertext{(by \eqref{eq:hold between besov spaces} and definition of Besov norm)}
& \leq C \|g\|_{T,\gamma}\, 2^{-j(\gamma-\eta)}\int_{\R^N} \Big{|} \frac{y}{n}\Big{|}_B^\eta   \Phi(y) dy.
\end{align}
Again by definition of Besov norm, the latter inequality implies \eqref{eq:bound_gn_g} and concludes the proof.

{We now prove Part (iii). 
It is sufficient  to prove that $t\mapsto \mathcal F g^{(n)} (t, \cdot) $ is continuous with values in $ \mathcal S$. We have 
\[ \mathcal F g^{(n)} (t, \cdot)  = \mathcal F g(t, \cdot)\, \mathcal F  \Phi_n. \]
Since $g$ is a continuous function with compact support then $(t\mapsto \mathcal F g(t, \cdot) )\in C_T \mathcal S$. 
Since $ \mathcal F  \Phi_n \in \mathcal S$ the result easily  follows. }
\end{proof}
The next result will be used in Section \ref{sec:linear MP} to prove the equivalence between our formulation of the martingale problem with the classical one by Stroock-Varadhan.

\begin{lemma}\label{lem:density_intrinsic_spaces}
Let $T>0$, $\alpha\in (0,1]$. For any $u\in \Cc^{2+\alpha}_{B,T}$, there exists a sequence $(u^{(n)})_{n\in\N}$ such that the following hold. 
\begin{itemize}
 \item[(a)] For any $n\in\N$, $u^{(n)} \in C^{1,2}([0,T]\times \R^N)$, with $u^{(n)}, \partial_{z_i} u^{(n)}, \partial_{z_i z_j} u^{(n)}$ bounded on $[0,T]\times \R^N$.
\item[(b)] For any $i,j=1,\cdots, d$ and for any compact $K\subset \R^N$, we have  
\begin{equation}\label{eq:conv_regul_u_intrinsic}
\| u - u^{(n)} \|_{T}, \ \| \partial_{z_i} u - \partial_{z_i} u^{(n)} \|_{T}, \ \| \partial_{z_{i}z_j} u - \partial_{z_{i}z_j} u^{(n)} \|_{T}, \ \| (Y u - Y u^{(n)} ) {\bf 1}_{K} \|_{T} \longrightarrow 0, \qquad \text{as } n\to +\infty,
\end{equation}
where $Yu$ denotes the Lie derivative of $u$,  continuous on $[0,T]\times \R^N$, defined as in \eqref{eq:def Lie der}. 
\end{itemize}
\end{lemma}
\begin{proof}
Let $\Phi$ be a standard mollifier on $\R^N$ and set, for any $n\in\N$,
\begin{equation}\label{eq:eq:un_moll}
 u^{(n)}(t,z) : = \frac{1}{|\det e^{tB}|}  \int_{\R^N} \Phi_n \big( e^{- t B} (z -  y ) \big) u(t, y) \dd y, \quad \Phi_n(y ) := n^{N} \Phi( n y ), \qquad (t,z)\in [0,T]\times\R^N.
\end{equation}

We first prove Part (a). By the properties of the mollifier and by the fact that $u$ is continuous and bounded on $[0,T]\times \R^N$, it is standard to prove that 
$\nabla_z u^{(n)}(t,z)$ is continuous and bounded on $[0,T]\times \R^N$ and that the same holds true for the second order derivatives in space.
We now study $\partial_t u^{(n)}(t,z)$. First note that a simple change of variable yields
\begin{equation}\label{eq:proof_mollif_un}
u^{(n)}(t,z)  = \int_{\R^N} \Phi_n ( y ) u (t,z -  e^{t B} y ) \dd y.
\end{equation}
Note that the Lie derivative along $Y$ is invariant with respect to the translation along its integral curves, namely
\begin{align}
Y \big((t,z)\mapsto u\big(t ,  z -  e^{t B} y\big) \big) &= \lim_{\delta\to 0} \frac{u\big(t + \delta,  e^{\delta B} (z -  e^{t B} y)\big) - u\big(t ,  z -  e^{t B} y\big)}{\delta} \\
& =  \lim_{\delta\to 0} \frac{u\big(t + \delta ,  e^{\delta B} \xi \big) - u(t,  \xi )}{\delta}\bigg|_{\xi=  z -  e^{t B} y}
  = (Yu)\big( t , z -  e^{t B} y \big),\label{eq:dx2I2}
\end{align}
for any $0\leq t \leq T$ and $z,y\in\R^N$. Owing to this, and to the boundedness of $Y  u$ on $[0,T]\times\R^N$, one can move the Lie derivative under the integral in \eqref{eq:proof_mollif_un} and obtain 
\begin{equation}\label{eq:YI2}
Y u^{(n)}(t,z) 
 = \int_{\R^N} \Phi_n ( y ) (Y  u)\big(t,  z -  e^{t B} y\big) \dd y .
\end{equation}
 Therefore, $u^{(n)}(t,z) $ is differentiable with respect to $t$ and we have 
\begin{equation}
\partial_t u^{(n)}(t,z) = Yu^{(n)}(t,z) - \langle Bz , \nabla_z u^{(n)}(t,z) \rangle, \qquad (t,z) \in [0,T] \times \R^N,
\end{equation}
where $\nabla_z u^{(n)}(t,z)$ makes sense a standard Euclidean gradient. As $Y u^{(n)}, \nabla_z u^{(n)}$ are continuous on $[0,T]\times \R^N$, so is $\partial_t u^{(n)}$. This concludes the proof of Part (a).

We now prove Part (b). By the boundedness of $\partial_{z_i}  u, \partial_{z_{i}z_j}  u$ on $[0,T]\times\R^N$, one can move the derivatives under the sign of the integral in \eqref{eq:proof_mollif_un} and obtain
\begin{align}\label{eq:estI11}
\partial_{z_i} u^{(n)}(t,z) &  = \int_{\R^N} \Phi_n ( y ) \partial_{z_i}  u\big(t ,  z -  e^{t B} y\big) \dd y , 
\\ \label{eq:estI12}
\partial_{z_{i}z_j} u^{(n)}(t,z) &  = \int_{\R^N} \Phi_n ( y ) \partial_{z_{i}z_j}  u\big(t ,  z -  e^{t B} y\big) \dd y . 
\end{align}
 Noticing now that, by assumption, $u, \partial_{z_i}  u, \partial_{z_{i}z_j}  u$ are uniformly continuous on $[0,T]\times\R^N$ and $Yu$ is uniformly continuous on $[0,T]\times K$ for any compact $K\subset \R^N$, it is standard to show that \eqref{eq:proof_mollif_un}, \eqref{eq:YI2}, \eqref{eq:estI11} and \eqref{eq:estI12} imply \eqref{eq:conv_regul_u_intrinsic}. 
\end{proof}

\section{Non-linear Fokker-Planck singular PDE}\label{sc:FP}

In this section we study the non-linear Fokker-Planck PDE associated to the McKean-Vlasov SDE  \eqref{eq:mkv_kinetic_singular}, i.e.
\begin{equation}
\label{eq:FP cauchy probl}
\begin{cases}
\Kc'  u_t =  \div_v\big(u_t F(u_t)b_t\big), \qquad t\in (0,T],\\
u_0= u^0,
\end{cases}
\end{equation}
with $\Kc'$ being the forward Kolmogorov operator associated to the Markovian SDE \eqref{eq:SDEgeneralLangevin}, defined in \eqref{eq:FP operator kinetic}-\eqref{eq:FP operator kinetic_bis}, and where $\div_v$ is the divergence operator in $\R^d$. Namely 
$\div_v$ formally acts on a function $G:\mathbb{R}^N \to \mathbb{R}^d$ as
\begin{equation}
\div_v G(z) = \sum_{j=1}^d \partial_{v_j} G_j(v,x),\qquad z=(v,x)\in\R^{d}\times \R^{N-d} .
\end{equation}
We assume, throughout this section, that Assumptions \ref{ass:kolmogorov_op}, \ref{ass:beta and b} and \ref{ass:phi} for $\Phi = F$ and $\Phi = \tilde F$
(see \eqref{eq:FFF}) are satisfied.

\subsection{Well-posedeness of the PDE}\label{ssc:FP-well}

We prove that \eqref{eq:FP cauchy probl} has a unique mild (or, equivalently, weak) solution 
in $\CT{\beta+\eps}$ for  a suitable $\eps>0$. 
Note that, a priori, it is not obvious how the term
\[ \div_v(u_tF(u_t)b_t),\qquad t\in [0,T], \]
in \eqref{eq:FP cauchy probl} might be well-defined. In this respect and to prove well-posedness of \eqref{eq:FP cauchy probl} we need the following
\begin{lemma}
\label{lem:F lipsch est}
Let $\Phi$ satisfy Assumption \ref{ass:phi}.
For any $\alpha\in (0,1)$, the map
\begin{equation}
\Phi:\C_B^\alpha\longrightarrow \C_B^\alpha, \qquad f \longmapsto \Phi f:= \Phi \circ f, 
\end{equation}
 is well-defined 
and there exists a constant $C>0$, only depending on $\Phi$, such that
\begin{align}\label{eq:estim_Ftil_lip}
\|\Phi f-\Phi g\|_\alpha &\leq C(1+\|f\|_\alpha+\|g\|_\alpha) \|f-g \|_\alpha, && f,g\in \C_B^\alpha, \\ \label{eq:estim_Ftild_f}
\|\Phi f\|_\alpha &\leq C(1+\|f\|_\alpha), && f\in \C_B^\alpha.
\end{align}
\end{lemma}
\begin{proof}
We only show \eqref{eq:estim_Ftil_lip}, with \eqref{eq:estim_Ftild_f} being simpler.  Furthermore, without loosing generality, we can assume $\Phi$ being scalar valued.  As the function $\Phi$ is Lipschitz continuous, we have
\begin{equation}
\| \Phi f -\Phi g \|_{\Leb^{\infty}} \leq C \|f-g\|_{\Leb^{\infty}}.
\end{equation}
Now fix $z,z'\in\R^d$. To complete the proof of \eqref{eq:estim_Ftil_lip}, we need to show 
\begin{equation}\label{eq:eq:estim_Ftil_lip_proof}
| \Phi f (z)-\Phi g (z) - \Phi f (z') + \Phi g (z') | \leq C (1+\|f\|_\alpha+\|g\|_\alpha) \|f-g \|_\alpha |z-z'|^{\alpha}_B.
\end{equation}
We set
\begin{equation}
a := f(z) - g(z), \quad b := f(z') - g (z'), \quad c:= f(z) - f(z'), \quad d:=g(z) - g(z'),
\end{equation}
and proceed by considering two separate cases:

\vspace{2pt}
\fbox{$|a|+|b| \leq |c| + |d| $}\, By mean-value theorem,
there exists $\xi$ included between $f(z)$ and $g(z)$, and $\eta$ included between $f(z')$ and $g(z')$, such that
\begin{equation}
\Phi  f (z)  - \Phi  g (z) = \Phi ' (\xi) a, \qquad \Phi  f (z')  - \Phi  g (z') = \Phi ' (\eta) b.
\end{equation}
Therefore, by triangular inequality we obtain
\begin{align}
| \Phi f (z)-\Phi g (z) - \Phi f (z') + \Phi g (z') | & \leq |\Phi  ' (\xi) (a-b)| + \big| b \big( \Phi  ' (\xi) - \Phi  ' (\eta) \big) \big| 
\intertext{(by Assumption \ref{ass:phi})}
& \leq C  \| f-g \|_{\alpha} \big( |z-z'|^{\alpha}_B  + | \xi - \eta| \big). \label{eq:estimate_Ftilde_proof}
\end{align}
Now, as $|a|+|b| \leq |c| + |d| $, we obtain
\begin{equation}
| \xi - \eta| \leq |a|+|b|+|c|+|d| \leq 2(|c| + |d| ) \leq 2 (\| f \|_{\alpha} + \| g \|_{\alpha}) |z - z'|^{\alpha}_B,
\end{equation}
which, together with \eqref{eq:estimate_Ftilde_proof}, yields \eqref{eq:eq:estim_Ftil_lip_proof}.

\vspace{2pt}
\fbox{$|c|+|d| \leq |a| + |b| $}\, By mean-value theorem, there exist $\xi$ included between $f(z)$ and $f(z')$, and $\eta$ included between $g(z)$ and $g(z')$, such that 
\begin{equation}
\Phi  f (z)  - \Phi  f (z') = \Phi ' (\xi) c, \qquad \Phi  g (z)  - \Phi  g (z') = \Phi ' (\eta) d.
\end{equation}
Therefore, by triangular inequality we obtain
\begin{align}
| \Phi f (z)-\Phi g (z) - \Phi f (z') + \Phi g (z') | & \leq |\Phi  ' (\xi) (c-d)| + \big| d \big( \Phi  ' (\xi) - \Phi  ' (\eta) \big) \big| 
\intertext{(by Assumption \ref{ass:phi})}
& \leq C \big( \| f-g \|_{\alpha}  + \|  g \|_{\alpha} | \xi - \eta| \big)  |z-z'|^{\alpha}_B . \label{eq:estimate_Ftilde_proof_bis}
\end{align}
Now, as $|c|+|d| \leq |a| + |b| $, we obtain
\begin{equation}
| \xi - \eta| \leq |a|+|b|+|c|+|d| \leq 2(|a| + |b| ) \leq 4 \| f - g \|_{\alpha} ,
\end{equation}
which, together with \eqref{eq:estimate_Ftilde_proof_bis}, yields \eqref{eq:eq:estim_Ftil_lip_proof} and completes the proof.
\end{proof}

\begin{remark}\label{rem:div_reg}
By Lemma \ref{lem:F lipsch est}  applied to $\Phi = \tilde F$, for any $\alpha\in (\beta,1)$ we have
\begin{equation}
\div_v (f F( f )b_t) \in  \C^{-\beta-1}_B , \qquad f \in C_B^{\alpha}, \quad t\in[0,T].
\end{equation}
Indeed, by the Bernstein inequality \eqref{eq: Bern ineq Besov norm}, the estimate for the Bony's product \eqref{eq: prod estim} and by \eqref{eq:estim_Ftild_f}, we obtain
\begin{equation}
\label{eq:est_div_sing}
\|\div_v(f  F(f)b_t)\|_{-\beta-1} \leq C \sum_{i=1}^d \big\| \big( (\tilde{F} f)  b_t \big)_i\big\|_{-\beta}
\leq C (1+\|f\|_{\alpha})\|b_t\|_{-\beta}< +\infty .
\end{equation}
\end{remark}
\begin{definition}\label{def:weak_sol_FP}
We say that a function $u\in \CT{\beta+\eps}$ is a \emph{weak solution} for the singular Fokker-Planck \eqref{eq:FP cauchy probl} if it is a distributional solution. Namely, for all $\varphi\in \mathcal{S}$ we have
\begin{equation}
\label{eq:weak sol def}
\< u_t|\, \varphi\>= \< u^0|\, \varphi\>+ \int_0^t \<u_s|\, \Ac \, \varphi\> \, ds +\sum_{i=1}^d \int_0^t \big\< u_s (F(u_s)b_s)_i|\, \partial_{v_i} \varphi\big\> ds , \qquad t\in [0,T],
\end{equation} 
where $\Ac$ is the generator of $Z$ in \eqref{eq:SDEgeneralLangevin}, defined through \eqref{eq:kolm_const}.
\end{definition}
We now introduce the notion of \emph{mild solution}, which turns out to be equivalent to the one of \emph{weak solution}.
\begin{definition}\label{def:mild_sol}
We say that a function $u\in \CT{\beta+\eps}$ is a \emph{mild solution} for the singular Fokker-Planck equation \eqref{eq:FP cauchy probl} if the integral equation 
\begin{equation}
\label{eq:def mild sol}
u_t=P'_t  u^0 {-} \int_0^t P'_{t-s}\big( \div_v \( u_s F(u_s)\, b_s\) \big) ds,\qquad t\in [0,T],
\end{equation}
is satisfied.
\end{definition}
\begin{remark}
Inequality \eqref{eq:est_div_sing} in Remark \ref{rem:div_reg} together with the Schauder's estimate \eqref{eq:schauder 1} yield
\begin{equation}\label{eq:integrability_Est}
\big\|P'_{t-s}\big( \div_v ( u_s F(u_s)\, b_s)\big) \big\|_{\beta+\eps}\leq C (t-s)^{-(\beta+\frac{\eps+1}{2})} \|\div_v ( u_s F(u_s)\, b_s)\|_{-\beta-1}\leq
C (t-s)^{-(\beta+\frac{\eps+1}{2})},
\end{equation}
where
\begin{equation}
\beta+\frac{\eps+1}{2}< 1,
\end{equation}
in light of $\beta\in(0,1/2)$ and $\eps \in (0,1-2 \beta)$. This shows that the LHS in \eqref{eq:integrability_Est} is integrable as a function of $s$, in $[0,t]$, and thus Definition \ref{def:mild_sol} is well-posed.  Indeed, the integral in \eqref{eq:def mild sol} can be understood as a pointwise Lebesgue integral, namely as 
\begin{equation}
\int_0^t P'_{t-s}\big( \div_v \( u_s F(u_s)\, b_s\) \big) (x) ds, \qquad x\in\R^N,
\end{equation}
and the standard triangular inequality yields
\begin{equation}
\| u_t \|_{\beta+\eps} \leq \| P'_t  u^0 \|_{\beta+\eps} + \int_0^t \big\|P'_{t-s}\big( \div_v ( u_s F(u_s)\, b_s)\big) \big\|_{\beta+\eps} ds.
\end{equation}

\end{remark}
\begin{remark} \label{Sprime}
  \begin{enumerate}
    \item
The right-hand side of $\eqref{eq:def mild sol}$ can  also be considered as a Bochner's integral with values in
$\mathcal{S}'$ equipped with its strong topology, being the strong dual of the vector space $\mathcal{S}$, see e.g. \cite[Chapter 4, Example 1 a]{walsh86}. 
For $k\in \N$, let $(E_k,\|\cdot \|_k)$ be the normed space
\begin{equation}\label{eq:k}
E_k:=\{\varphi\in C^\infty(\R^N)|\, \|\varphi\|_k:=\sup_{z\in \R^N,\, |\alpha|\leq k} |D^\alpha \varphi(z)| \left(|z|^2+1\right)^{k}<\infty \}.
\end{equation}
Then $\mathcal{S}$ is a vector (metrizable) space given by the intersection of the $E_k$'s, while $\mathcal{S'}$ coincides with the union of the spaces $E^*_k$, which are 
the strong duals of  $(E_k,\|\cdot\|_k)$.
The (strong) topology of ${\mathcal S}'$ is the one corresponding with the inductive limit of the $E^*_k$'s. 
\item
Since $E^*_k$ endowed with the strong topology is a Banach space, it is then possible to define the Bochner's integral $\int_0^T g_s ds$ for $g\in\Leb^\infty_T \mathcal{S}'$. The latter 
is the space of strongly measurable $g:[0,T]\rightarrow \mathcal{S}'$
such that there is $k$ with $g(t) \in E_k^*$ for a.e. $ t\in [0,T]$ and that
$ t \mapsto \Vert g(t) \Vert_{E_k^*}$  belongs to $L^\infty([0,T])$. 
\item   Specifically, if $g\in \Leb^\infty_T \mathcal{S}'$
then we have the following.
\begin{itemize}
\item[i)] The map $t\mapsto \langle u_t|\, \varphi\rangle$ is Borel for all $\varphi\in \mathcal{S}$.
\item[ii)] There exists $C=C(g)>0$, $k\in \N$ such that  for all $\varphi\in \mathcal{S}$
\begin{equation}
|\langle g_t |\, \varphi\rangle| \leq C \|\varphi\|_k  \qquad \text{a.e. in }[0,T]. 
\end{equation} 
\end{itemize}
\item We say that $g:[0,T] \to {\mathcal S}'$ belongs to
  $C^{AC}([0,T]; {\mathcal S}')$ if there is $g':[0,T] \to {\mathcal S}'$
  continuous with $g(t) = g(0) + \int_0^t g'(s) ds$ for any $t \in [0,T]$.
  \end{enumerate}
\end{remark}
The next result shows the equivalence between mild and weak solutions.  Although the statement assumes $u\in \CT{\beta+\eps}$ in accordance with Definitions \ref{def:weak_sol_FP} and \ref{def:mild_sol}, its proof relies on the weaker assumption that $u\in\Leb^\infty_T \mathcal{S}'$, the latter being the natural framework in order to make sense of the the concepts of mild and weak solutions to the linear version of the Fokker-Planck equation.
\begin{proposition}
  \label{prop:mild_weak_equiv}
Let $u\in \CT{\beta+\eps}$. Then $u$ is a weak solution of \eqref{eq:FP cauchy probl} if and only if it is a mild solution.
\end{proposition}
\begin{proof}
For $g\in \Leb^{\infty}_T \mathcal{S}'$, $t\in [0,T]$, and $\varphi\in \mathcal{S}$ we have
\begin{equation}
|\langle P'_{t-s}g_s|\, \varphi\rangle |=|\langle g_s|\, P_{t-s} \varphi \rangle|\leq C(g) \|\varphi \|_k,\qquad s\in [0,t],
\end{equation}
for some $k\in \N$, where $\|\cdot\|_k$ was defined in \eqref{eq:k}. Then $s\mapsto P'_{t-s}g_s \in \Leb^{\infty}_t \mathcal{S}'$ and, for $f_0\in \mathcal{S}'$, 
 the mild function $u:[0,T]\to \mathcal{S}'$ defined by
\begin{equation}
\label{eq:mild g eq}
u_t:=P'_t u^0 - \int_0^t P'_{t-s}g_s ds,\qquad t\in [0,T],
\end{equation}
is well-defined in $\Leb^\infty_T \mathcal{S}'$. By inspection, it can be verified that such a $u$ is a weak solution of
\begin{equation}
\label{eq:FP cauchy g}
\begin{cases}
\Kc' \, u_t = g_t,\qquad t\in(0,T),\\
u_0=u^0,
\end{cases}
\end{equation}
namely, for every $\varphi\in \mathcal{S}$
\begin{equation}
\label{eq:weak sol g}
\< u_t|\, \varphi\>= \<u^0|\, \varphi\>+ \int_0^t \big( \<u_s|\, \Ac \, \varphi\> 
+ \< g_s |\,  \varphi \> \big) ds , \qquad t\in [0,T].
\end{equation}
By setting $g_s:=\div_v(u_s F(u_s) b_s)$, it holds that a mild solution of \eqref{eq:FP cauchy probl} is also a weak solution.
To prove the opposite implication it is then sufficient to prove uniqueness of weak solutions to \eqref{eq:FP cauchy g} in the sense above.
This fact can be established in the same way as in
\cite[Proposition 3.18]{LucasOR}, even though that
result was stated for  finite measure-valued functions.
\end{proof}

We now state one of the main results of this section.
\begin{theorem}\label{th:well_posedness_FP}
  Let Assumptions \ref{ass:kolmogorov_op}, \ref{ass:beta and b},   and Assumption \ref{ass:phi}  for $\Phi = \tilde F$,
  hold.
 For all $\varepsilon \in (0, 1-2\beta)$ such that $u^0 \in \C_B^{\beta+\varepsilon}$ 
  there exists a unique mild solution (or equivalently, by Proposition \ref{prop:mild_weak_equiv}, a weak solution) $u\in \CT{\beta+\varepsilon}$ to \eqref{eq:FP cauchy probl}.
\end{theorem}
\begin{remark}\label{rem:existence_uniq}
Notice that, given an initial condition  $u^0 \in \C_B^{\beta+\varepsilon}$ for some $\varepsilon \in (0, 1-2\beta)$,  the solution  $u$ to equation \eqref{eq:FP cauchy probl} exists in the space $\CT{\beta+\varepsilon}$ but it is unique in any larger space of the form   $\CT{\beta+\varepsilon'}$ for $0<\varepsilon' <\varepsilon$.
\end{remark}

\begin{remark}\label{rem:well_posedness_FP}
The result of Theorem \ref{th:well_posedness_FP} remains true, with identical proof, under a weaker version of Assumption \ref{ass:beta and b}, namely replacing $b\in \mathit{C}_T \C_B^{-\beta}$ by $b\in \Leb^{\infty}_T \C_B^{-\beta}$.
\end{remark}

To prove Theorem \ref{th:well_posedness_FP}, 
the idea is to find a fixed point $w^*\in\Linf{\beta+\varepsilon}$ for the map
\begin{equation}\label{eq:def_Jt}
J_t(w):= - \int_0^t P'_{t-s}[ \div_v G_s(w)]ds,\quad  \quad G_s(w):=\tilde{F}(w_s+P'_s u^0)b_s, \qquad t\in [0,T].
\end{equation}
A mild solution to \eqref{eq:FP cauchy probl} can be then obtained by setting $u_t:=w^*_t+P'_t u^0$, after verifying that $u \in \mathit{C}_T \C_b$. The uniqueness of the solution will stem from the fact that the fixed point $w^*$ is proved via contraction.
\begin{remark}
Proceeding as in Remark \ref{rem:div_reg} one can see that, if $w\in\Linf{\beta+\varepsilon}$, the integral in \eqref{eq:def_Jt} can be understood as a standard pointwise Lebesgue integral, namely
\begin{equation}
J_t(w)(x) = - \int_0^t P'_{t-s}[{\div_v} G_s(w)](x)\, ds, \qquad x\in\R^N.
\end{equation}
The standard triangular inequality also yields
\begin{equation}
\| J_t(w) \|_{\beta+\varepsilon} \leq \int_0^t  \| P'_{t-s}[{ \div_v} G_s(w)]  \|_{\beta+\varepsilon}  \, ds .
\end{equation} 
\end{remark}

For 
any $\varrho \geq 0$ we define   
the norm on $\Linf{\beta+\eps}$ given by
\begin{equation}\label{eq:equiv_norm_rho}
\|w\|_{\varrho,T,\beta+\eps}:=\sup_{t\in [0,T]}e^{-\varrho t}\|w_t\|_{\beta+\eps},
\end{equation}
and we consider the $\varrho$-closed balls
\begin{equation}
E_{\varrho,M}^{\beta+\varepsilon}:=\{w\in \Linf{\beta+\eps}\;: \|w\|_{\varrho,T,\beta+\eps}\leq M\}, \qquad M>0.
\end{equation}
\begin{remark}\label{rem:complete}
For any $\varrho \geq 0$, the norm $\|\cdot\|_{\varrho,T,\beta+\eps}$ is equivalent to $\|\cdot\|_{T,\beta+\eps}$.
Furthermore, the $\varrho$-closed balls are closed sets of $\Linf{\beta+\eps}$, and thus $(E^{\beta+\varepsilon}_{\varrho,M}, \|\cdot\|_{\varrho,T,\beta+\varepsilon})$ is a complete metric space for any $M>0$.
\end{remark}
The choice of these spaces 
is justified by the following lemma.
\begin{lemma}
\label{prop:J space mapping}
 There exist $\varrho_0, M_0>0$, depending on $T$, $\beta$, $\eps$, $\|u^0\|_{\beta+\eps}$, $\|b\|_{T,-\beta}$,
such that 
\begin{equation}
\label{eq:J-fixed ball}
J:E^{\beta+\varepsilon}_{\varrho_0,M}\longrightarrow E^{\beta+\varepsilon}_{\varrho_0,M}, \qquad M>M_0.
\end{equation}
Furthermore, for any $M>M_0$ there exists $\varrho_M>0$, also dependent on $T$, $\beta$, $\eps$, $B$, $\|u^0\|_{\beta+\eps}$, $\|b\|_{T,-\beta}$, such that 
\begin{equation}\label{eq:J is a contraction}
\|J(w)-J(w')\|_{\varrho_M,T,\beta+\varepsilon}\leq \frac{1}{2} \|w-w'\|_{\varrho_M,T,\beta+\varepsilon}, \qquad w,w' \in E^{\beta+\varepsilon}_{\varrho_0,M}.
\end{equation}
\end{lemma}

\begin{proof} Owing to the anisotropic Schauder's estimate of Theorem \ref{th:schauder_first},   the proof is identical to the one of \cite[Proposition 3.5 and Lemma 3.6]{issoglio_russoMK}.
\end{proof}
\begin{lemma}\label{lem:continuity_J}
  We have
\begin{align}
  (t\mapsto P'_{t} g ) \in  \mathit{C}_T \C_b,& \qquad g \in \C_b,
                    \label{eq:contin_semigroup} \\ 
J(w) \in  \mathit{C}_T \C_b,& \qquad w\in \Linf{\beta+\eps}.  \label{eq:contin_convol}
\end{align}
\end{lemma}
\begin{proof}
The proof of \eqref{eq:contin_semigroup} is standard: indeed $\Gamma(t-s, y , z)$ is smooth on $\{s>t\}$, and \eqref{eq:delta_solfund} holds true uniformly with respect to $z\in\R^N$. 
  Let now $w\in \Linf{\beta+\eps}$ be fixed. Lemma \ref{prop:J space mapping} implies that
  $J_t(w) \in \C_b$ for any $t\in [0,T]$. To show \eqref{eq:contin_convol},   note that applying  \eqref{eq:est_div_sing} in Remark \ref{rem:div_reg} with 
$f = w_s + P'_s u^0 \in \C_B^{\beta+\eps}$ and Schauder's estimate \eqref{eq:schauder 1} with $\alpha=0$ yields
\begin{equation}
\label{eq:est divG_s}
\|\div_v\, G_s(w)\|_{-\beta-1}
\leq C(1+\|w_s\|_{\beta+\varepsilon}+\|u^0\|_{\beta+\eps})\|b_s\|_{-\beta},\qquad s\in[0,T],
\end{equation} 
and thus $\div_v G (w)\in L_T^\infty \C_B^{-\beta-1}$. Therefore, it is sufficient to show that
\begin{equation} \label{eq:suff}
t\longmapsto \int_0^t P'_{t-s}f_s ds\in \mathit{C}_T\C_b,\qquad f\in L_T^\infty \C_B^{-\beta-1}.
\end{equation}
We study the continuity in time, by splitting the integrals as follows. For $h\in (0,T-t)$, using the semigroup property, we have
\begin{align}
\int_0^{t+h} P'_{t+h-s} f_s ds-\int_0^t P'_{t-s}f_s ds 
&=\int_{0}^t \Big(P'_h P'_{t-s}f_s- \underbrace{P'_{t-s}f_s}_{=:g_s}\Big)+\int_t^{t+h} P'_{t+h-s}f_{s}ds\\
&=\int_{0}^t \left(P'_hg_s- g_s\right)ds+\int_t^{t+h} P'_{t+h-s}f_{s}ds=:I^h_1+I^h_2.
\end{align}
To estimate $\|I^h_2\|_\infty$, we set $\delta\in (0,\frac{1}{2})$. By 
the Schauder's estimate \eqref{eq:schauder 1} we obtain
\begin{equation}
\|I^h_2\|_\infty\leq \int_t^{t+h} \|P'_{t+h-s}f_s\|_\infty ds \leq C \int_t^{t+h} \|P'_{t+h-s}f_s\|_{\delta}\, ds\leq C\|f\|_{T,-\beta-1}\int_0^{h} (h-s)^{-\frac{\beta+1+\delta}{2}} ds \xrightarrow[h\to 0^+]{}0 .
\end{equation}
To estimate $\|I^h_1\|_\infty$, note that \eqref{eq:contin_semigroup} yields
\begin{equation}
 \| P'_h g_s- g_s \|_\infty \to 0, \qquad \text{as} \quad h\to 0^+,
\end{equation}
for any $s\in [0,t)$. Also, $\| P'_h g_s \|_\infty \leq C \|g_s \|_\infty$, the latter being summable on $[0,t]$ (once more by Schauder's estimate \eqref{eq:schauder 1}). Therefore, Lebesgue's dominated convergence theorem yields $\|I^h_1\|_\infty \to 0$ as $h\to 0^+$. This proves continuity from the right. Continuity from the left can be shown with analogous computations.
\end{proof}

We are now in the position to prove Theorem \ref{th:well_posedness_FP}.
\begin{proof}[Proof of Theorem \ref{th:well_posedness_FP}]
Since $P_{\cdot} u^0 \in \Linf{\beta+\varepsilon}$, a function $u\in\Linf{\beta+\varepsilon}$ satisfies \eqref{eq:def mild sol} if and only if $u_t=w^*_t+P_t u^0$ with 
\begin{equation}
J(w^*) = w^* \in\Linf{\beta+\varepsilon}. 
\end{equation}
Lemma \ref{prop:J space mapping} and Remark \ref{rem:complete} allow us to apply Banach fixed point theorem and obtain that such $w^*$ exists and is unique. Finally, Lemma \ref{lem:continuity_J} ensures that both $(t\mapsto P_{t} u^0), w^* \in \mathit{C}_T \C_b$, and this completes the proof.
\end{proof}

\subsection{The regularized equation}\label{ssc:FP-cont}

For any $n\in\N$, consider the regularized Fokker-Planck Cauchy problem
\begin{equation}
\label{eq:FP cauchy probl_regul}
\begin{cases}
\Kc'  u_t =  \div_v\big(u_t F(u_t)b^{(n)}_t\big), \qquad t\in (0,T],\\
u_0= u^0,
\end{cases}
\end{equation}
where $b^{(n)}$ is the smooth approximation of $b$ defined  in  \eqref{eq:b_n bis}.

\begin{theorem}\label{thm:conv_un}
  Let Assumptions \ref{ass:kolmogorov_op}, \ref{ass:beta and b} with related $\varepsilon$, and \ref{ass:phi} for $\Phi =\tilde F$  be in force. 
 Then the following
  properties hold.
\begin{itemize}
\item[(a)]
For any $n\in\N$ and $\alpha> 0$, we have $b^{(n)}\in  \CT{\alpha} $, and the regularized Fokker-Planck Cauchy problem \eqref{eq:FP cauchy probl_regul} admits a unique mild solution $u^{(n)}\in \CT{\beta+\eps}$, in the sense of Definition \ref{def:mild_sol}, with 
\begin{equation}\label{eq:L1_bound_u}
u^{(n)}_t\geq 0,\qquad 
{ \| u^{(n)}_t \|_{\Leb^1(\R^N)} \leq 1, \qquad t\in[0,T].}
\end{equation}
\item[(b)]
We have
\begin{equation}\label{eq:conv_un_u_FP}
\|u^{(n)} - u \|_{T,\beta+\eps} \rightarrow 0, \qquad \text{as } n\to \infty,
\end{equation}
where
$u \in \CT{\beta+\eps}$ is the unique solution to \eqref{eq:FP cauchy probl}. 
\end{itemize}
\end{theorem}
The rest of the section is devoted to the proof of Theorem \ref{thm:conv_un}. We start by proving the first part.
\begin{proof}[Proof of Theorem \ref{thm:conv_un}, Part (a)] 
Here $n$ is fixed. 
By Lemma \ref{lem:regul_g}-(i), $b^{(n)}\in \Linf{\alpha} $ for any $\alpha> 0$. In particular, we have $b^{(n)}\in \Linf{-\beta} $, and, by  
Remark \ref{rem:well_posedness_FP}, 
the regularized FP Cauchy problem \eqref{eq:FP cauchy probl_regul} admits a unique mild solution $u^{(n)} $ such that  
\begin{equation}\label{eq:unCb}
u^{(n)}\in \CT{\beta+\eps}.
\end{equation}
We now show \eqref{eq:L1_bound_u}. Setting 
\begin{equation}
{ \Bc^{(n)}_t(z):= F\big(u^{(n)}_t(z)\big)b^{(n)}_t(z)}, \qquad t\in[0,T],\quad z\in \R^N,
\end{equation}
the function $u^{(n)}$ solves the linear FP Cauchy problem 
\begin{equation}
\label{eq:FP cauchy probl_regul_lin}
\begin{cases}
\Kc'  u_t = \div_v\big(u_t { \Bc^{(n)}_t}\big), \qquad t\in (0,T],\\
u_0= u^0.
\end{cases}
\end{equation}  
   We can  show that
   ${ \Bc^{(n)}}\in \CT{\beta+\eps}$:  by Lemma \ref{lem:regul_g}-(i), $b^{(n)}\in C_T \C_b\cap \Linf{\beta+\eps}$, and  $  F\big(u^{(n)}_t(z)\big)\in \Linf{\beta+\eps} $ thanks to Lemma \ref{lem:F lipsch est} with $\Phi=F$, and $  F\big(u^{(n)}_t(z)\big)\in  C_T \C_b\cap \Linf{\beta+\eps} $ thanks to \eqref{eq:unCb},
   thus $\Bc^{(n)} \in C_T \C_b\cap \Linf{\beta+\eps}$.
 Since the PDE in \eqref{eq:FP cauchy probl_regul_lin} is written in divergence form, the results about classical solutions for degenerate operators in non-divergence form are not directly applicable since $\Bc^{(n)} $ is not regular enough.
 We thus need a further regularization step. For any $m\in\N$, consider the linear Cauchy problem
\begin{equation}
\label{eq:FP cauchy probl_regul_lin_reg}
\begin{cases}
\Kc'  u_t {- \div_v\big(u_t { \Bc^{(n,m)}_{t}}\big) =0}, \qquad t\in (0,T],\\
u_0= u^0,
\end{cases}
\end{equation}
where ${ \Bc^{(n,m)}}$ denotes the regular approximation of ${ \Bc^{(n,m)}}$ defined as in Lemma \ref{lem:regul_g}.
Obviously $\Bc^{(n,m)} \in C_T \C_b$ and by Lemma \ref{lem:regul_g}-(i),  $\Bc^{(n,m)}\in\CT{\alpha}$ for any $\alpha> 0$. Therefore, by \cite[Theorem 2.7 and  Remark 2.9]{lucertini2023optimal}, \eqref{eq:FP cauchy probl_regul_lin_reg} has a strong Lie solution $u^{(n,m)}$, which can be represented  (see  \cite[equation (3.1)]{lucertini2023optimal}) as
\begin{equation}\label{eq:duhamel_u_mn}
u^{(n,m)}_t = \int_{\R^N} \Gamma^{(m)} (0,y; t, \cdot) u^0(y) d y, \qquad t\in [0,T],
\end{equation}
where $\Gamma^{(m)}(s,y;t,z)$ is the fundamental solution of the operator 
\begin{equation}
\Kc' -  \langle  { \Bc^{(n,m)} }, \nabla_v   \rangle -  \div_v{ \Bc^{(n,m)} } 
\end{equation}
in the variables $(t,z)$. Also, proceeding like in the proof of \cite[Theorem 1.5]{difpas}, it can be shown that $\Gamma^{(m)}(s,y;t,z)$ is also the fundamental solution of the operator 
\begin{equation}
{\Kc + \langle { \Bc^{(n,m)} }, \nabla_v   \rangle }
\end{equation}
in the variables $(s,y)$. In particular $\Gamma^{(m)}>0$ { (see \cite[Theorem 1.1]{lucertini2022optimal})} and $\int_{\R^N} \Gamma^{(m)}(0,y; t, z) dz =1$ { (by \cite[Theorem 2.7]{lucertini2023optimal})}. The latter and \eqref{eq:duhamel_u_mn} imply  \eqref{eq:L1_bound_u} for $u^{(n,m)}$.  Therefore, owing to the positivity of the functions and to Fatou's lemma, to show \eqref{eq:L1_bound_u} for $u^{(n)}$ it is enough to prove
\begin{equation}\label{eq:conv_unm_un}
\|u^{(n,m)} - u^{(n)} \|_{T,\beta+\eps} \rightarrow 0, \qquad \text{as } m\to \infty.
\end{equation}
 Owing once more to 
\cite[(3.1)]{lucertini2023optimal}, it is {straightforward} to see that $u^{(n,m)}$ is also a mild solution to \eqref{eq:FP cauchy probl_regul_lin_reg} in the sense of Definition \ref{def:mild_sol}, and the proof of \eqref{eq:conv_unm_un} is a simpler version of the proof of \eqref{eq:conv_un_u_FP}: we omit the details for brevity.
\end{proof}

To prove the second part of Theorem \ref{thm:conv_un}, we need the following stability result.

\begin{lemma}\label{lem:comparison_mild_sol_FP}
Let the assumptions of Theorem \ref{thm:conv_un} be in force. Let $\tilde b \in \Linf{-\beta}$, and denote by $\tilde u \in \CT{\beta+\eps}$ the mild solution of the nonlinear FP Cauchy problem obtained by replacing $b$ with $\tilde b$ in \eqref{eq:FP cauchy probl}, which exists in light of Remark \ref{rem:well_posedness_FP}. Then, for any $0\leq \eta< \min(\eps,1-2\beta-\eps)$ there exists a function  $l_{\eps,\eta}:\R^+\times\R^+\longrightarrow \R^+$, independent of $b$ and $\tilde b$, increasing in the second variable, such that
\begin{equation}
\| u-\tilde u\|_{T,\beta+\eps}\leq l_{\eps,\eta}\big(\|u^0\|_{\beta+\varepsilon}, \|b\|_{T,-\beta-\eta}\vee \|\tilde b\|_{T,-\beta-\eta}\big)\|b-\tilde b\|_{T,-\beta-\eta},
\end{equation}
where $u \in \CT{\beta+\eps}$ is the unique solution to \eqref{eq:FP cauchy probl}. 
\end{lemma}
\begin{proof}
Owing to the anisotropic Schauder's estimate of Lemma \ref{th:schauder_first}, and the continuity of the function $\tilde F$ by Lemma \ref{lem:F lipsch est}, the proof is a straightforward modification of the one of \cite[Proposition 4.7-(i)]{issoglio_russoMK}, upon substituting $\| u\|_{T,\alpha}$, $\|\tilde u\|_{T,\alpha}$ with $\| u\|_{T,\beta + \eps}$, $\| \tilde u\|_{T,\beta + \eps}$ and $\| b\|_{T,-\beta}$, $\|\tilde b\|_{T,-\beta}$ with $\| b\|_{T,-\beta - \eta}$, $\| \tilde b\|_{T,-\beta - \eta}$.
\end{proof}
We are now in the position to prove Part (b) of Theorem \ref{thm:conv_un}.

\begin{proof}[Proof of Theorem \ref{thm:conv_un}, Part (b)]\vspace{2pt}
The result stems from combining Lemma \ref{lem:regul_g}-(ii) with  Lemma \ref{lem:comparison_mild_sol_FP}  with $\tilde u = u^{(n)}$ and $\tilde b = b^{(n)}$ and using the fact that $$l_{\varepsilon, \eta} (\|u^0\|_{\beta+\varepsilon},  \|b\|_{T,-\beta-\eta}\vee \| b^{(n)}\|_{T,-\beta-\eta}) \leq l_{\varepsilon, \eta} (\|u^0\|_{\beta+\varepsilon} ,  \|b\|_{T,-\beta-\eta}\vee  \sup_n\|b^{(n)}\|_{T,-\beta-\eta}). $$ 
\end{proof}
\begin{remark}
The result of Theorem \ref{thm:conv_un}  remains true, with identical proof, under a weaker version of Assumption \ref{ass:beta and b}, namely replacing $b\in \mathit{C}_T \C_B^{-\beta}$ by $b\in \Leb^{\infty}_T \C_B^{-\beta}$.
\end{remark}

\section{Backward Kolmogorov PDE with singular drift}\label{sec:Kolmogorov PDE}

In this section we study the backward Kolmogorov Cauchy problem associated to the Markovian kinetic SDE with singular drift, obtained by 
\eqref{eq:mkv_singular} with   $F(u)b$ replaced by $\Bc$. 
Namely, for any ${\lambda\geq 0}$, we consider the backward Cauchy problem
\begin{equation}
\label{eq:back_cauchy probl}
\begin{cases}
\Kc  u + \langle \Bc , \nabla_v \rangle u  = \lambda u + g  , \qquad t\in (0,T),\\
u_T=\ell,
\end{cases}
\end{equation}
with $\Kc$ being the backward Kolmogorov operator associated to the Markovian SDE \eqref{eq:SDEgeneralLangevin}, defined in \eqref{eq:kolm_const}. 
Throughout this section we let Assumption \ref{ass:kolmogorov_op} be in force, and we fix  $g, \Bc \in \Linf{-\beta}$ and $\ell \in \C_B^{1+\beta+\eps}$, with $T>0$, $\beta \in (0,1/2)$ and $ \eps \in (0,1-2\beta)$.

\subsection{Well-posedness of the PDE}
\begin{definition}\label{def:mild_sol_back}

  We say that a function $u\in \CT{1+\beta+\eps}$ is a \emph{mild solution} for the backward Kolmogorov Cauchy problem \eqref{eq:back_cauchy probl} if the integral equation 
\begin{equation}
\label{eq:back_cauchy probl_def}
u_t = P_{T-t} \ell - \int_t^T P_{s-t}\big(  \lambda u_s + g_s - \langle \Bc_s , \nabla_v \rangle u_s \big) ds, \qquad t\in [0,T],
\end{equation}
is satisfied.
\end{definition}
\begin{remark}
If $u\in \CT{1+\beta+\eps}$, then Schauder's estimate \eqref{eq:schauder 1} yields
\begin{equation}\label{eq:integrability_Est_bis}
\big\| P_{s-t}\big(  \lambda u_s + g_s - \langle \Bc_s , \nabla_v \rangle u_s \big) \big\|_{1+\beta+\eps}\leq C (s-t)^{-(\beta+\frac{\eps+1}{2})} \| \lambda u_s + g_s - \langle \Bc_s , \nabla_v \rangle u_s \|_{-\beta}\leq
C (s-t)^{-(\beta+\frac{\eps+1}{2})},
\end{equation}
where
\begin{equation}
\beta+\frac{\eps+1}{2}< 1 
\end{equation}
in light of $\beta\in(0,1/2)$ and $\eps \in (0,1-2 \beta)$. This shows that the LHS in \eqref{eq:integrability_Est_bis} is integrable as a function of $s$, in $[t,T]$, and thus Definition \ref{def:mild_sol_back} makes sense.
Indeed, the integral in \eqref{eq:back_cauchy probl_def} can be understood as a pointwise Lebesgue integral, namely as 
\begin{equation}
\int_t^T P_{s-t}\big(  \lambda u_s + g_s - \langle \Bc_s , \nabla_v \rangle u_s \big)(x) ds, \qquad x\in\R^N,
\end{equation}
and the standard triangular inequality yields
\begin{equation}
\| u_t  \|_{1+\beta+\eps} \leq \| P_{T-t} \ell  \|_{1+\beta+\eps} + \int_t^T \big\| P_{s-t}\big(  \lambda u_s + g_s - \langle \Bc_s , \nabla_v \rangle u_s \big) \big\|_{1+\beta+\eps} ds.
\end{equation}
\end{remark}

\begin{definition}\label{def:weak_sol_defBis}
We say that a function $u\in \CT{1+\beta+\eps}$ is a \emph{weak solution} to the backward Kolmogorov Cauchy problem \eqref{eq:back_cauchy probl} if it is a distributional solution. Namely, for all $\varphi\in \mathcal{S}$ we have
\begin{equation}
\label{eq:weak_sol_defBis}
\< \ell\, |\, \varphi\>= \<u_t|\, \varphi\>- \int_t^T \<u_s|\, \Ac' \, \varphi\> \, ds -\sum_{i=1}^d \int_t^T \big\< (\Bc_s)_i (\partial_{v_i}  u_s) |\, \varphi \big\> ds + \int_t^T \< \lambda u_s +g_s |\, \varphi \> ds
 , \qquad t\in [0,T],
\end{equation} 
where $\Ac'$ is the formal adjoint of $\Ac$ as defined in \eqref{eq:FP operator kinetic_bis}.
\end{definition}
The proof of the next statement is fully analogous to that of Proposition \ref{prop:mild_weak_equiv}, 
and thus is omitted.
\begin{proposition}\label{prop:equivalence_sol_backward}
A function $u\in \CT{1+\beta+\eps}$ is a {weak solution} to the backward Kolmogorov Cauchy problem \eqref{eq:back_cauchy probl} if and only if it is a mild solution.
\end{proposition}

We now state the first main result of this section. 
\begin{theorem}\label{thm:sol_back_CP}
Let Assumption \ref{ass:kolmogorov_op}  hold, and let $g, \Bc \in \Linf{-\beta}$ and $\ell \in \C_B^{1+\beta+\eps}$, for some  $\eps \in (0,1-2\beta) $. Then there exists a unique mild solution $u\in \CT{1+\beta+\eps}$ to \eqref{eq:back_cauchy probl}.
\end{theorem}

\begin{remark}
If $0<\eps' <\eps$ then clearly $\ell \in \C_B^{1+\beta+\eps'}$ and the unique solution $u \in \CT{1+\beta+\eps'}$ coincides with the solution $u \in \CT{1+\beta+\eps}$, by uniqueness in the larger space $ \CT{1+\beta+\eps'} $.
\end{remark}

\begin{proof}[Proof of Theorem \ref{thm:sol_back_CP}] 
To prove Theorem \ref{thm:sol_back_CP},
we follow the same strategy employed in the parabolic case (\cite[Theorem 4.7]{issoglio2022pde}), namely $d=N$ and $B\equiv 0$. The idea is to find a unique fixed point $u\in\Linf{1+\beta+\eps}$ for the solution map
\begin{equation}
I_t(w):= P_{T-t} \ell - \int_t^T P_{s-t}\big(  \lambda w_s + g_s - \langle \Bc_s , \nabla_v \rangle w_s \big) ds, \qquad t\in [0,T],
\end{equation}
and to show that $u\in \mathit{C}_T \C_b$.

Relying on the anisotropic Schauder estimate of Theorem \ref{th:schauder_first}, one can proceed exactly as in the proof of \cite[Theorem 4.7]{issoglio2022pde} and obtain that $I$ is a contraction from the complete metric space 
$(\Linf{1+\beta+\eps}, \| \cdot \|_{\varrho,T,1+\beta+\eps})$ onto itself for $\varrho>0$ suitably large, where $\| f \|_{\varrho,T,1+\beta+\eps}  : = \sup_{t\in[0,T]} e^{-\varrho(T-t)} \|f_t\|_{ 1+\beta+\eps}$ is the analogous backward of the norm defined in \eqref{eq:equiv_norm_rho}. Banach fixed point theorem then yields that there exists a unique fixed point $u\in\Linf{\beta+\eps+1}$ for $I$. Eventually, a straightforward modification of the proof of Lemma \ref{lem:continuity_J} yields $u\in \mathit{C}_T \C_b$ and completes the proof.
\end{proof}

Then next result shows that, if $\Bc_t,g_t$ are H\"older-continuous functions, then the mild solution $u$ is a solution also in some classical sense and enjoys additional regularity.
\begin{proposition}\label{prop_lie_sol}
Let Assumption \ref{ass:kolmogorov_op} be in force. If $\ell \in \C^{2+\alpha}_B$, $g \in \CT{\alpha}$ for $\alpha\in(0,1)$ and $ \Bc\in \CT{\alpha+\nu}$ for some $\nu>0$,  then the mild solution $u$ to \eqref{eq:back_cauchy probl} belongs to $\Cc^{2+\alpha}_{B,T}$ (see Definition \ref{def:intrinsicHolder}) 
and is also a strong Lie solution in the sense of \cite[Definition 2.6]{lucertini2023optimal}.
\end{proposition} 
\begin{remark} In particular, the function $u$ solves the PDE pointwise, everywhere on $[0,T]\times\mathbb{R}^N$, where the term $Yu$ that appears in $\Kc u$ makes sense as a Lie derivative (see \eqref{eq:def Lie der}). 
\end{remark}
\begin{proof}[Proof of Proposition \ref{prop_lie_sol}]
  By \cite[Theorem 2.7 and Remark 2.9]{lucertini2023optimal} (see also \cite{brampol}), \eqref{eq:back_cauchy probl} admits a (unique) Lie solution $u\in \Cc{2+\alpha}_{B,T}$. Furthermore, by \cite[(3.1)]{lucertini2023optimal}, the latter
 verifies equation \eqref{eq:back_cauchy probl_def}, and thus $u$ coincides with the unique mild solution to \eqref{eq:back_cauchy probl}.
\end{proof}

\subsection{Continuity properties}
We recall it is assumed that $g,\Bc \in \Linf{-\beta}$ and $\ell \in \C_B^{1+\beta+\eps}$, with  $\beta\in(0,1/2)$ and $\eps \in (0,1-2\beta)$. 
Let $\Bc^{(n)},g^{(n)}\in  \Linf{-\beta}$ and $\ell^{(n)}\in\C_B^{1+\beta+\eps}$ denote any  approximations of $b$, $g$ and $\ell$, respectively, such that 
\begin{equation}
  \|\Bc-\Bc^{(n)}\|_{T, -\beta-\eta},  \|g-g^{(n)}\|_{T, -\beta-\eta}, \|\ell-\ell^{(n)}\|_{1+\beta+\eps-\eta} \rightarrow 0 \qquad \text{as } n\to \infty,
\end{equation}
for all $\eta>0$. An example of how to construct said approximation  is given in by \eqref{eq:b_n bis} and  Lemma \ref{lem:regul_g} ensures the convergence. 
For any fixed $n\in\N$, we consider the unique {mild (or,
  equivalently, weak)} solution $u^{(n)}$ to the approximate backward Kolmogorov Cauchy problem
\begin{equation}
\label{eq:back_cauchy probl_reg}
\begin{cases}
\Kc  u + \langle \Bc^{(n)} , \nabla_v \rangle u  = \lambda u + g^{(n)}  , \qquad t\in (0,T),\\
u_T=\ell^{(n)}.
\end{cases}
\end{equation}

\begin{theorem}\label{thm:conv_un_back}
Let Assumption \ref{ass:kolmogorov_op} be in force and let $g, \Bc \in \Linf{-\beta}$ and $\ell \in \C_B^{1+\beta+\eps}$ for some $\eps \in (0,1-2\beta) $.
\begin{itemize}
\item[(a)] {Let $\lambda \geq 0$. 
For any $\eta>0$,} we have
\begin{equation} \label{eq:cont_propr}
\|u^{(n)} - u \|_{T,1+\beta+{\eps -\eta}} \longrightarrow 0, \qquad \text{as } n\to \infty,
\end{equation}
where  
$u^{(n)},u\in\CT{1 + \beta+\eps}$ 
are the unique mild solutions 
to \eqref{eq:back_cauchy probl_reg} and \eqref{eq:back_cauchy probl}, respectively.
\item[(b)] {Let $\ell\equiv 0$.  There exists $\bar\lambda >0$, dependent on $T$, $\beta$, $\eps$, $\| \Bc \|_{T,-\beta}$ and $\| g \|_{T,-\beta}$ such that for any $\lambda>\bar \lambda$ we have  }
\begin{equation}\label{eq:bound_grad_u_one_half}
\|  u^{(n)} \|_{T,1 + \beta + \eps}  \leq \frac{1}{2}, \qquad n\in \N.
\end{equation}
\end{itemize}
\end{theorem}

\begin{remark}\label{rem:contraction_u}
Combining Part (a) and Part (b) of Theorem \ref{thm:conv_un_back}, one trivially obtains that \eqref{eq:bound_grad_u_one_half} also holds for the unique  solution $u$ in $ \CT{1 + \beta+\eps}$ to the singular Cauchy problem \eqref{eq:back_cauchy probl} with $\lambda>\bar\lambda$ and $\ell\equiv 0$. 
\end{remark}

\begin{proof}[Proof of Theorem \ref{thm:conv_un_back}]
  We first prove Part (a). It is enough to show the claim for $\eta <  \eps $. Owing to the anisotropic Schauder's estimate of Lemma \ref{th:schauder_first}, the proof is a straightforward modification of the one of \cite[Lemma 4.17-(ii)]{issoglio2022pde} where we substitute
  $\| u - u^{(n)}\|_{\varrho,T,1+\alpha}$ with $\| u -  u^{(n)}\|_{\varrho,T,1+\beta + \eps-\eta}$, $\| \Bc - \Bc^{(n)}\|_{T,-\beta}$ with $\| \Bc -  \Bc^{(n)}\|_{T,-\beta - \eta}$, $\| g - g^{(n)}\|_{\varrho,T,-\beta}$ with $\| g -  g^{(n)}\|_{\varrho,T,-\beta - \eta}$ and  $\| \ell - \ell^{(n)}\|_{1+\alpha}$ with $\| \ell -  \ell^{(n)}\|_{1+\beta+\eps-\eta}$.

We now prove Part (b). 
Let $n\in\N$ be fixed. 
Proceeding similarly as in the proof of \cite[Proposition 4.13-(ii)]{issoglio2022pde} one obtains that $u^{(n)}$ is a solution of 
\begin{equation}\label{eq:un_duhamel}
u^{(n)}_t = - \int_t^T e^{-\lambda (s-t)} P_{s-t}\big(  g^{(n)}_s - \langle \Bc^{(n)}_s , \nabla_v  \rangle u^{(n)}_s  \big) ds, \qquad t\in [0,T] .  
\end{equation}
In fact it is enough to choose $\eta < \min(\eps,1-2\beta-\eps)$
and denote by $C$, indistinctly, any positive constant that depends, at most, on $T$, $\beta$ and $\eps$.
For any $0\leq t < s \leq T$, applying the Schauder estimate of Theorem \ref{th:schauder_first} yields
\begin{align}
\|  P_{s-t}\big(  g^{(n)}_s - \langle \Bc^{(n)}_s , \nabla_v  \rangle u^{(n)}_s  \big) \|_{1+\beta+\eps} & \leq C (s-t)^{-\frac{1 + \eps + \eta + 2\beta}{2}} \|    g^{(n)}_s - \langle \Bc^{(n)}_s , \nabla_v \rangle u^{(n)}_s  \|_{-\beta-\eta}
\intertext{(by Proposition \ref{prop:bony_prod}, since $\eps-\eta>0$)}
  & \leq C (s-t)^{-\frac{1 + \eps + \eta + 2\beta}{2}} \big( \|    g^{(n)}_s \|_{-\beta-\eta} + \|  \Bc^{(n)}_s \|_{-\beta-\eta} \| \nabla_v u^{(n)}_s  \|_{\beta+\eps} \big)\\
&  \leq C (s-t)^{-\frac{1 + \eps + \eta + 2\beta}{2}} \big( \|    g_s^{(n)} \|_{-\beta} + \|  \Bc_s^{(n)} \|_{-\beta} \| \nabla_v u^{(n)}_s  \|_{\beta+\eps} \big) \\
&  \leq C (s-t)^{-\frac{1 + \eps + \eta + 2\beta}{2}} \big( \sup_m\|    g_s^{(m)} \|_{-\beta} + \sup_m\|  \Bc_s^{(m)} \|_{-\beta} \|  u^{(n)}_s  \|_{1+\beta+\eps}\big),
\end{align}
which, combined with \eqref{eq:un_duhamel}, yields
\begin{align*}
\|  u^{(n)}  \|_{T,1+\beta+\eps} & \leq C  \big( \sup_m\|    g^{(m)} \|_{T,-\beta} +\sup_m\|  \Bc^{(m)} \|_{T,-\beta}  \|  u^{(n)}  \|_{T,1+\beta+\eps} \big)  \sup_{t\in[0,T]} \int_t^T e^{-\lambda (s-t)} (s-t)^{-\frac{1 + \eps + \eta + 2\beta}{2}}ds\\
& \leq C   \big( \sup_m\|    g^{(m)} \|_{T,-\beta} +\sup_m\|  \Bc^{(m)} \|_{T,-\beta}  \|  u^{(n)}  \|_{T,1+\beta+\eps} \big)   \int_0^T e^{-\lambda s} s^{-\frac{1 + \eps + \eta + 2\beta}{2}}ds
\end{align*}
Now, as $\eta < 1-2\beta-\eps$, the latter integral is finite and tends to $0$ as $\lambda \to + \infty$. This completes the proof.
\end{proof}

\begin{corollary}\label{cor:uniform_conv_un}
 Under the assumptions of Theorem \ref{thm:conv_un_back} we have $u^{(n)} \to u$ in $ \mathit{C}_T \C_b$ as $n\to\infty$, and in particular $u^{(n)}$ is equicontinuous in $n$ as a function on $[0,T]$ with values in $ \C_b$.
\end{corollary}
\begin{proof}
 It is a straightforward consequence of Theorem \ref{thm:conv_un_back}, Part (a).
\end{proof}

\begin{lemma}
\label{lm:pre-tighness}
Let  $(\Bc^{(n)})_{n\in \N} \subset L^\infty_T \C_B^{-\beta}$  be such that $\Bc^{(n)} \to \Bc$ in $L^\infty_T \C_B^{-\beta{-\eta}}$ for some $\eta>0$.  
Let  $u^{(n)}$ be the   mild $\R^N$-valued  solution in $L_T^\infty  \C_B^{1+\beta{+\varepsilon}}$,  for some $\varepsilon \in(0, 1-2\beta)$ not depending on $n$, to the  Cauchy problem
\begin{equation}
\label{eq:Zvonkin PDE reg}
\begin{cases}
\Kc u^{(n)}+ \langle \Bc^{(n)},\nabla_v \rangle u^{(n)}= \lambda\, u^{(n)} - \begin{pmatrix}
\Bc^{(n)}\\
0
\end{pmatrix} \\
u^{(n)}_T=0,
\end{cases}
\end{equation}
which is meant component by component, where we recall that $\Bc^{(n)}$ is $d$-dimensional and the vector 0 here is $(N-d)$-dimensional. Similarly, let $u$ be the unique solution to 
 \begin{equation} 
\label{eq:Zvonkin PDE}
\begin{cases}
\Kc u+ \langle \Bc,\nabla_v \rangle u= \lambda\, u -\begin{pmatrix}
\Bc\\
0
\end{pmatrix} \\
u_T=0.
\end{cases}
\end{equation}
There exists $\bar \lambda$  such that  for all $\lambda> \bar \lambda$ we have
\begin{equation}
\label{eq:bound_grad}
\|  \nabla_v u^{(n)} \|_{\infty} \vee \|  \nabla_v  u\|_{\infty}
\le \frac{1}{2}, \qquad n\in \N.
\end{equation}
We set  
\begin{equation}
\phi^{(n)}_t(z):= z + u^{(n)}_t(z), \quad \phi_t(z):= z + u_t(z) ,\qquad t\in [0,T],\ z\in \R^N,
\end{equation}
and notice that they take values in $\mathbb R^N$. Then
we have the following.
\begin{itemize}
\item[(i)] For any fixed $t\in [0,T]$, the function $\phi^{(n)}_t$ 
is invertible;
\item[(ii)] denoting by $\psi^{(n)}$ the inverse of $\phi^{(n)}$, the sequence $(\psi^{(n)})$ is
 equicontinuous; 
\item[(iii)] $\phi^{(n)} \to \phi$ uniformly on $[0,T] \times \R^N$; 
\item[(iv)]  $\psi^{(n)} \to \psi$ uniformly on $[0,T] \times \R^N$ where $\psi$  is the inverse of $\phi$; 
 \item[(v)]  $\sup_{t,\tilde x}  |\psi_t^{(n)}(\tilde v,\tilde x)  |\leq c (1+ |\tilde v|) $, $\forall \tilde v \in \R^d$ for a universal constant $c$ independent of $n$. 
\end{itemize}
\end{lemma}

 \begin{proof}

By Theorem \ref{thm:sol_back_CP} each component of  $u^{(n)}$ is well-defined as weak solution of \eqref{eq:Zvonkin PDE reg}. 
 Theorem \ref{thm:conv_un_back} and Remark \ref{rem:contraction_u} guarantee that the sequence of solutions $u^{(n)}$ converges in $L_T^\infty \C_B^{1+\beta+\eps}  $ for some $\eps\in(0, 1-2\beta)$
 to the unique weak solution $u$ of \eqref{eq:Zvonkin PDE}, and that $\bar \lambda$ can be chosen such that 
\eqref{eq:bound_grad} holds. 

 To prove Item  (i), we fix $t\in [0,T]$ and, for notational simplicity, we omit the variable $t$. We sometimes write $\phi^{(n)}(z) = (\phi_1^{(n)}(z), \phi_2^{(n)}(z))$, $u^{(n)}(z) = (u_1^{(n)}(z), u_2^{(n)}(z))$ and $z=(v,x)$, where the first component of these vectors is $d$-dimensional and the second is $(N-d)$-dimensional. Due to the definition of the vector $u^{(n)}(z)$ as solution of the PDE above, and by uniqueness of the PDE, we have  that $u^{(n)}_2(z) = 0$ and so 
\begin{equation}\label{eq:phi(z)}
\phi^{(n)}(z) = (\phi^{(n)}_1(z), \phi^{(n)}_2(z)) = ( v + u^{(n)}_1(z), x + u^{(n)}_2(z)) = ( v + u^{(n)}_1(z), x ).
\end{equation}
To show that $\phi^{(n)}(v,x)$ is invertible it is equivalent to find a unique solution $(v,x)$ of the system
 \begin{equation}
 \label{eq:inverse Zvonkin problem}
 \begin{cases}
 \tilde{v}=u^{(n)}_1(v,x)+v\\
 \tilde{x}=x,
 \end{cases}
 \end{equation}
 for any given $\tilde{z}=(\tilde{v},\tilde{x})\in\R^N$. If it exists,
 we denote the inverse of $\phi^{(n)}$ by $\psi^{(n)}=(\psi_1^{(n)},\psi_2^{(n)})\in \R^N$. Clearly $x:=\tilde x$ so the inverse of the second block is simply the identity, $\psi_2^{(n)} (\tilde v,\tilde x) = \tilde x$. As for the first block, by \eqref{eq:bound_grad} we have 
 \begin{equation}\label{eq:u1Lip}
\big{|} u^{(n)}_1(v,x)-u^{(n)}_1(v',x)\big{|}\leq \frac{1}{2}|v-v'|, \qquad v,v'\in\R^d, 
 \end{equation}
 uniformly in $(t,x)$, so the map 
  $\tilde{v}-u^{(n)}_1(\cdot,x)$ is a contraction, with fixed point in $\R^d$ denoted by $\psi_1^{(n)}(\tilde v, \tilde x)$. 
 It immediately follows that $(v,{x}):= (\psi_1(\tilde v, \tilde x),\psi_2(\tilde v, \tilde x))$ solves \eqref{eq:inverse Zvonkin problem} and so $\psi^{(n)}(\tilde z) = (\psi_1^{(n)}(\tilde z), \psi_2^{(n)}(\tilde z))$ is the inverse of $\phi^{(n)}(z)$.
 
We now verify Item (ii), that is  the equicontinuity of $(\psi^{(n)})$; to this aim, it is enough to consider only the first component $\psi_1^{(n)} $  because $\psi_2^{(n)}$ is simply the identity in the $\tilde x$ variable. For  $(\psi_1^{(n)})$ we have
\begin{equation}\label{eq:J}
|\psi_1^{(n)} (t,\tilde v, \tilde x) - \psi_1^{(n)} (t',\tilde v', \tilde x')| \leq  |\psi_1^{(n)} (t,\tilde v, \tilde x) - \psi_1^{(n)} (t',\tilde v, \tilde x')| + |\psi_1^{(n)} (t',\tilde v, \tilde x') - \psi_1^{(n)} (t',\tilde v', \tilde x')| =: J_1^{(n)} + J_2^{(n)}.
\end{equation}
For the term $J_1^{(n)}$ we can reduce the problem to the equicontinuity of $( u^{(n)}_1)$ in $(t,x)$ uniformly in $v$ as follows: setting
\[ v:=\psi_1^{(n)}(t,\tilde{v},\tilde{x}),\qquad v':=\psi_1^{(n)}(t',\tilde{v},\tilde{x}'),\]
 and using that $\tilde x = x$ and $\tilde x' = x'$ we get
\begin{equation}
\begin{split}
\tilde{v}&=\phi_1^{(n)}(t,v,{x})=v+u^{(n)}_1(t,v,{x})\\
\tilde{v}&=\phi_1(t',v',{x}')=v'+u^{(n)}_1(t',v',{x}').
\end{split}
\end{equation}
Then,
 \begin{align}
 |\psi_1^{(n)}(t,\tilde{v},\tilde{x})-\psi_1^{(n)}(t',\tilde{v},\tilde{x}')|
 &=|v-v'| \\
 &=\big{|}u_1^{(n)}(t,v,{x})-u_1^{(n)}(t',v', {x}')\big{|}\\
&\leq \big{|}u_1^{(n)}(t,v,{x})-u_1^{(n)}(t,v',{x})\big{|} + \big{|}u_1^{(n)}(t,v',{x})-u_1^{(n)}(t',v',{x}')\big{|}\\
&\leq \|\nabla_v u_1^{(n)}\|_\infty |v-v'|+\big{|} u_1^{(n)}(t,v',{x})-u_1^{(n)}(t',v',{x}')\big{|}\\
&\leq \frac12 | \psi_1^{(n)}(t,\tilde{v},\tilde{x})-\psi_1^{(n)}(t',\tilde{v},\tilde{x}')|+\big{|} u_1^{(n)}(t,v',{x})-u_1^{(n)}(t',v',{x}')\big{|},\label{eq:psi1}
 \end{align}
 having used Theorem \ref{thm:conv_un_back} Part (b). Moving on the LHS the first term we have
\begin{align}\label{eq:psi1i}
|\psi_1^{(n)}(t,\tilde{v},\tilde{x})-\psi_1^{(n)}(t',\tilde{v},\tilde{x}')| &\leq 2 \big{|} u_1^{(n)}(t,v',{x})-u_1^{(n)}(t',v',{x}')\big{|}\\ \nonumber
&\leq 2 \big{|} u_1^{(n)}(t,v',{x})-u_1^{(n)}(t',v',{x})\big{|} + 2 \big{|} u_1^{(n)}(t',v',{x})-u_1^{(n)}(t',v',{x}')\big{|}\\
& =: J_{11}^{(n)} + J_{12}^{(n)}. \nonumber
\end{align}
The term $J_{11}^{(n)}$ is controlled by Corollary \ref{cor:uniform_conv_un}. For the term $J_{12}^{(n)}$, since $u_1^{(n)}(t, \cdot) \in  \C^{1+ \beta+\eps}_B\subset \C^\alpha_B$ for every $\alpha \in (0,1)$, by
 Remarks \ref{rm:holder} and \ref{rm:holder2} there exists some $\nu \in (0,1)$ such that for all $t\in[0,T]$
$$  \|u_1^{(n)}(t, \cdot, \cdot) \|_{\C^\nu} \leq C \|u_1^{(n)}(t, \cdot, \cdot) \|_{\alpha}  \leq C \|u_1^{(n)}(t, \cdot, \cdot) \|_{1+\beta } \leq C  \|u_1^{(n)} \|_{T, 1+\beta } ,$$ where the latter bound is uniformly bounded in $n$ thanks to Theorem \ref{thm:conv_un_back}, Part (a). Thus,  recalling also that $\tilde x = x$ and $\tilde x'= x'$, we obtain 
$J_{12}^{(n)} \leq C \sup_n  \|u_1^{(n)} \|_{T, 1+\beta }  |\tilde x - \tilde x'|^\nu $.\\
For the term $J_2^{(n)}$ in \eqref{eq:J} we notice that 
\[
|\psi_1^{(n)} (t',\tilde v, \tilde x') - \psi_1^{(n)} (t',\tilde v', \tilde x')| \leq \|\nabla_{\tilde{v}}\psi_1 \|_\infty |\tilde v - \tilde v'|,
\]
so it is enough to show that $\nabla_{\tilde{v}} \psi_1^{(n)}$ is uniformly bounded.
As $\sup_n \|\nabla_v u_1^{(n)}\|_\infty \le \frac{1}{2}$ by \eqref{eq:bound_grad}, the Jacobian
\begin{equation}
\nabla_v \phi_1^{(n)} =I_d+\nabla_v u_1^{(n)}
\end{equation}
is nonsingular for every $(t,v,x)$.
 By the Inverse function Theorem, $\psi_1$ is of class $C^1$ in $\tilde{v}$; then
 \begin{equation}\label{eq:nabla-psi1}
 \|\nabla_{\tilde{v}}\psi_1^{(n)}\|_\infty=\|(\nabla_v \phi_1^{(n)})^{-1}\|_\infty \leq 2.
  \end{equation}
  
Item (iii) states uniform convergence of $\phi^{(n)}$ to $\phi$, and it  follows from the definition of $\phi^{(n)}$ and $\phi$ and the fact that $u_1^{(n)} \to u_1 $ uniformly by Corollary \ref{cor:uniform_conv_un}. 

Item (iv) states uniform convergence of $\psi^{(n)}$ to $\psi$. Taking into account Item (ii) and the fact that  $\psi^{(n)}_2(t, \tilde v, \tilde x) = \tilde x  = \psi_2(t, \tilde v, \tilde x)$,  it is enough to show pointwise convergence of $\psi_1^{(n)}(t, \tilde v, \tilde x)$ to $\psi_1(t, \tilde v, \tilde x)$. We define $v^{(n)}:= \psi^{(n)}_1(t, \tilde v, \tilde x)$ 
  we have by the inverse property   
  \begin{equation}\label{eq:tildev}
  \tilde v  = \phi_1^{(n)}(t, v^{(n)}, \tilde x) = v^{(n)}+ u_1^{(n)}(t, v^{(n)}, \tilde   x) .
  \end{equation} Since $(u_1^{(n)})$ is a uniformly bounded sequence, also $(v^{(n)}) $ is a bounded sequence and thus it converges to some $v^{(0)}$ up to a subsequence. 
  Since $(\phi_1^{(n)})$ converges uniformly  by Item (iii), taking the limit as $n\to \infty$ in \eqref{eq:tildev} we get $ \tilde v  = \phi_1(t, v^{(0)}, \tilde x) $, so by the inverse property again  $v^{(0)} = \psi_1(t, \tilde v, \tilde x)$ which concludes the proof. 

 Item (v) can be seen with similar ideas as in Item (ii). Again, we drop the variable $t$. Since $\psi^{(n)}(\tilde v, \tilde x) =(\psi^{(n)}_1( \tilde v, \tilde x) , \psi^{(n)}_2( \tilde v, \tilde x) )=(\psi^{(n)}_1(\tilde v, \tilde x) , \tilde x ) $, we focus on the linear growth of the first component. From bound \eqref{eq:psi1i} with $t'=t$, $\tilde x' =0$ we get
\[
|\psi^{(n)}_1( \tilde v, \tilde x)-\psi^{(n)}_1( \tilde v, 0)|  \leq  4 \sup_n\|u^{(n)}_1\|_{C_T \C_b},
\] 
and 
  together with \eqref{eq:nabla-psi1}   we have
\begin{align}
|\psi^{(n)}_1( \tilde v, \tilde x) | 
&\leq |\psi^{(n)}_1( \tilde v, \tilde x)-\psi^{(n)}_1( \tilde v, 0)| + |\psi^{(n)}_1( \tilde v, 0)-\psi^{(n)}_1( 0, 0)| + |\psi^{(n)}_1( 0,0)|\\
& \leq 4 \sup_n\|u^{(n)}_1\|_{C_T \C_b}  + 2 |\tilde v |  + \sup_n |\psi^{(n)}_1( 0,0)|\\
& \leq c(1+  |\tilde v |) , \label{eq:psi_1_lin_gro}
\end{align}
 having used the fact that $\sup_n\|u^{(n)}_1\|_{C_T \C_b}  $ is bounded  by Corollary \ref{cor:uniform_conv_un} and that $\sup_n|\psi^{(n)}_1( 0,0)|  $  is finite by Item (iv).
  \end{proof}

\section{The  singular McKean-Vlasov kinetic SDE}\label{sc:MKMP}

In this section we apply the results on the PDEs to define and show well-posedness for the singular McKean-Vlasov kinetic equation \eqref{eq:mkv_singular}.
Here the canonical space will be $C_T(\mathbb R^N)$  equipped, as usual, with the Borel sets with respect to the topology induced by the uniform convergence. In the sequel we will also make use of the standard \emph{canonical process} on the canonical space.

\subsection{The martingale problem formulation for the linear case}\label{sec:linear MP}
 Let $\mu_0$ be a probability law on $\R^N$. 
Throughout this subsection we assume that Assumption \ref{ass:kolmogorov_op} holds and $\Bc \in C_T \C^{-\beta}_B$ with $\beta \in (0,1/2)$. 
We consider the linear   case of equation \eqref{eq:mkv_singular}, which is formally
\begin{equation}\label{eq:SDE_singular}
\begin{cases}
\dd V_t = \Big (\Bc_t(Z_t) + \Bc_0 Z_t \Big) \dd t   +     \dd W_t       \\
\dd X_t = \Bc_1 Z_t \dd t,
\end{cases}
\end{equation}
with initial condition $Z_0 \sim \mu_0$, where the $\R^N$-valued  stochastic process $Z_t:= (V_t, X_t)$ is the unknown, and $B:=\begin{pmatrix}
  \Bc_0  \\ \Bc_1 
  \end{pmatrix}$  
  is given as in \eqref{B_tot}.
 The term linear refers to the fact that the corresponding Fokker-Planck equation is linear, and is reflected in the SDE by the fact that the coefficients do not depend on the law of the unknown.
  Clearly \eqref{eq:SDE_singular} is only formal, and its meaning is given below in terms of martingale problem.

The {backward Kolmogorov} operator $\mathcal L$  associated to the SDE \eqref{eq:SDE_singular} with drift $\Bc$ is given by 
\begin{equation}\label{eq:L}
\mathcal L u = \mathcal K u + \langle \Bc, \nabla_v  \rangle u,
\end{equation}
 where
$\mathcal K$ is given in \eqref{eq:kolm_const}
and $ u \in 
C^{AC}([0,T]; {\mathcal S}') \cap  L_T^\infty \C_B^{1+\beta}$.

Let $g \in \CT{\varepsilon}$ 
and  $f_T \in \C_B^{1+\beta+\varepsilon}$   for some $\varepsilon \in (0,1-2\beta)$. We denote by  $u$ the unique {weak (or, equivalently, mild)} solution of \begin{equation}\label{eq:cauchy_MP_def}
\begin{cases}
\mathcal{L} u = g\\
 u_T = f_T
 \end{cases}
 \end{equation}
 in the space $\CT{1+\beta+\varepsilon}$.
 Let  $Z$ be a stochastic process on some probability space $(\Omega, \mathcal F, \mathbb P)$, and set
\begin{equation} \label{eq:martingale}
M_t^u {= M_t^u(Z)} := u_t( Z_t) - u_0( Z_0) - \int_0^t g_s( Z_s) \dd s.
\end{equation}

\begin{definition}\label{def:solMP}
{Let $\mu_0$ be a probability law on  $\R^N$ and $t_0\in[0,T)$. Let also $\Bc \in \mathit{C}_T \C_B^{-\beta}$}
  and let the operator $\mathcal L$ with drift $\Bc$ be given  in \eqref{eq:L}.
  We denote by MP$( \Bc, \mu_0 ; t_0)$ the martingale problem formally formulated as SDE \eqref{eq:SDE_singular}, with initial condition $Z_{t_0}\sim \mu_0$. 
  
We say that a couple $(Z, \mathbb P) $ defined on some measurable space $(\Omega, \mathcal F)$ is a {\em solution to MP}$( \Bc, \mu_0; t_0)$ if $Z$ is continuous, $Z_{t_0} \sim \mu_0$ and for all $g\in C_T \mathcal S$  and all $f_T \in \mathcal S$
 we have that $(M_t^u - M^u_{t_0})_{t\in[ t_0 , T ]}$ is a (local) martingale under $\mathbb P$ with respect to the canonical filtration of $Z$, where $M^u$ is given in \eqref{eq:martingale} and $u$ is the solution of \eqref{eq:cauchy_MP_def}.
We say that  MP$( \Bc, \mu_0 ; t_0)$ {\em admits uniqueness}  if, given any pair of solutions $(Z^1, \mathbb P^1) $ and $(Z^2, \mathbb P^2) $, then the law of $Z^1$ under $\mathbb P^1$ is the same as the law of $Z^2$ under $\mathbb P^2$.

If $t_0 = 0$ we  employ the abbreviation MP$( \Bc, \mu_0 )$ in place of MP$( \Bc, \mu_0 ; 0)$.
\end{definition}

\begin{remark}\label{rm:ZP}
Let  $(\Omega, \mathcal F, \mathbb P)$ be a probability  space and $Z$ be a stochastic process on  $(\Omega, \mathcal F)$.  Let us denote by  ${\mathbb P}_Z$ the law of $Z$ under $\mathbb P$ and by $\hat Z$ the canonical process on the canonical space $C_T \R^N$.
The couple $(Z, \mathbb P)$ is a  solution to  MP$( \Bc, \mu_0)$ if and only if the couple  $(\hat Z, {\mathbb P}_Z)$  is  a solution to  MP$( \Bc, \mu_0)$. This is a direct consequence of the change of variable theorem  and the martingale property.
\end{remark}
\begin{remark}
\label{rem:domain MP}
The  solution to MP$(\Bc, \mu_0)$ can be equivalently defined as follows. Let us  define the domain
\begin{equation} 
\mathcal D_{\mathcal L}:=\{ u\in \CT{1+\beta+\varepsilon} {\text{ for some } \eps \in (0,1-2\beta)} : u \text{ is a  sol.\ to \eqref{eq:cauchy_MP_def} for some } g \in  C_T \mathcal S, f_T \in \mathcal S \}.
\end{equation}
Then $(Z, \mathbb P)$ is a solution to  MP$(\Bc, \mu_0)$ if and only if for all $u \in \mathcal D_{\mathcal L}$, $M^u$ is a local martingale under $\mathbb P$.
\end{remark}

\begin{lemma}\label{lm:Mu}
  If $(Z, \mathbb P)$ is a solution to MP$(\Bc, \mu_0)$ then the process $M^u$ defined in \eqref{eq:martingale} is a local martingale even if $u $ is solution to PDE \eqref{eq:cauchy_MP_def} with {$g \in \CT{\varepsilon}$}
  for some $\eps >0$ and  $f_T \in \C_B^{1+\beta+\eps}$, when both have compact support. 
\end{lemma}
\begin{proof}
By Lemma \ref{lem:regul_g} Parts (ii) and (iii) 
there exist two sequences $(g^{(n)} )_{n\in\N}\subset C_T \mathcal S$ and $(f_T^{(n)} )_{n\in\N} \subset \mathcal S$ such that $\|g^{(n)}-g\|_{T, \eps}\to 0$ and $\|f^{(n)}_T - f_T\|_{1+\beta+\eps} \to 0$ for some $\eps>0$. Thus by Theorem \ref{thm:conv_un_back}  Part (a)  there exists $\eps>0$ such that $\|u^{(n)}-u\|_{T, 1+\beta+\eps}\to 0$. Consequently $g^{(n)} \to g  $
and $u^{(n)} \to u$ uniformly on compacts and thus $M^{u^{(n)}} \to M^u$ in the sense of uniform convergence  in  probability (ucp). We conclude that $M^u$ is a local martingale since the space of local martingales is closed under the ucp topology.
\end{proof}

In the next result we  compare the notion of solution to the MP given in Definition \ref{def:solMP} with the classical notion of solution \`a l\`a Stroock-Varadhan.  Suppose that
{$\Bc \in \CT{\varepsilon}$,} for some $\eps>0$.
We recall that $(Z, \mathbb P)$ is a solution to  the Stroock-Varadhan Martingale Problem  with respect to $\mathcal L_{\Bc}:= \mathcal K  + \langle \Bc, \nabla_v \rangle$ meant in the classical sense, with initial condition $\mu_0$ if $Z$ is continuous,  $Z_0 \sim \mu_0$ and, for every $u\in  C^{1,2}([0,T]\times \R^N)$  (or equivalently $u\in C_c^\infty$), the process
\begin{equation}\label{eq:SV}
u_t(Z_t) - u_0(Z_0) - \int_0^t (\mathcal{L}_\Bc u)_s( Z_s) \dd s 
\end{equation}
is a (local) martingale, see e.g.\ \cite[Chapter 6]{stroock_varadhan}.
\begin{proposition}\label{prop:MP-Strook}
Let Assumption \ref{ass:kolmogorov_op} hold.
Let $(Z, \mathbb P)$ be a couple where $Z$ is a stochastic process with $Z_0 \sim \mu_0$ and $\mathbb P$ is  a probability measure on a given measurable space $(\Omega, \mathcal F)$. 
Let {$\Bc \in \CT{\varepsilon}$,}
for some $\eps>0$. 
Then  $(Z, \mathbb P)$ is a solution to MP$(\Bc, \mu_0)$ if and only if it is a solution to the classical Stroock-Varadhan martingale problem with respect to $\mathcal L_\Bc$  
with initial condition $\mu_0$. 
\end{proposition}

\begin{proof}
``$\Rightarrow$'' We have to show that for any $u\in C_c^\infty$ then
$M^{u}$ is a local martingale. It is clear that $\Kc u \in C_c^\infty $ and
$\langle \Bc, \nabla_z u\rangle \in  \CT{\varepsilon}$,
hence 
$\mathcal L u \in \CT{\varepsilon}$ for some $\eps>0$.
Moreover since $u$ has compact support and is time-homogeneous,  also $\mathcal L u$ and $u(T, \cdot )$ have compact support. The claim now follows by Lemma \ref{lm:Mu}.

``$\Leftarrow$'' Let $g\in C_T \mathcal S \subset \CT{\varepsilon}$
and {$f_T \in \mathcal S \subset  \C_B^{ 2+\eps} $}. By assumption, {$\Bc \in \CT{\varepsilon},$}
and thus, by  Proposition \ref{prop_lie_sol}, the mild solution $u$ to \eqref{eq:cauchy_MP_def} belongs to $\Cc^{2+\alpha}_{B,T}$ for any $\alpha<\eps$ and solves pointwise 
\begin{equation}
\Lc u=g\qquad (t,z)\in [0,T]\times \R^N,
\end{equation}
where the term $Yu$ that appears in $\Lc u$, through $\Kc u$, makes sense as a Lie derivative (see \eqref{eq:def Lie der}). 
Furthermore, by Lemma \ref{lem:density_intrinsic_spaces}, there exists a sequence $(u^{(n)})$ such that each $u^{(n)} \in  C^{1,2}([0,T]\times \R^N)$, and
\begin{equation}\label{eq:conv_regul_u_intrinsicBis}
\| u - u^{(n)} \|_{T},  \ \| ( g-\mathcal{L}u^{(n)}) \mathbf{1}_K \|_{T} \longrightarrow 0, \qquad \text{as } n\to +\infty,
\end{equation}
for any compact $K\subset \R^N$. Thus $ u^{(n)} \to u$ and $\mathcal{L}u^{(n)} \to g$ uniformly on compacts and so $M^{u^{(n)}} \to M^u$ ucp, which implies that  $M^{u}$ is a local martingale on $[0,T]$.
\end{proof}

\begin{corollary}\label{cor:MP-Strook} 
 Under the conditions of Proposition \ref{prop:MP-Strook}, a solution to MP$(\Bc, \mu_0)$ always exists.
\end{corollary}
\begin{proof}
Under the conditions of Proposition  \ref{prop:MP-Strook}, by \cite[Theorem 12.2.3  and Exercise 12.4.1]{stroock_varadhan} we know that there exists a  solution to the classical Stroock-Varadhan Martingale Problem with deterministic initial condition $z$ at time $s$. The solution is constructed on the canonical space and the probability is denoted by $ \mathbb P_{s,z}$.  Moreover $(s,z) \mapsto \mathbb P_{s,z}$ is measurable. Thus we can 
 superpose these laws to treat any initial condition  $\mu_0$ at time 0, by setting  $\mathbb P (B):= \int_{\R^N} \mathbb P_{0,z}(B)  \mu_0(dz) $ where $B$ is a Borel set of the canonical space. For more details  on the {\em superposition} argument, see   the proof of existence of  \cite[Theorem 4.5]{issoglio_russoMPb}.    
  Hence as a consequence of Proposition \ref{prop:MP-Strook} a solution to MP$(\Bc, \mu_0)$ always exists when
{$\Bc \in \CT{\varepsilon}$,}
for some $\eps>0$. 
\end{proof}

\begin{lemma}
\label{lm:tighness}
Let the hypotheses of Lemma \ref{lm:pre-tighness} hold and $\bar\lambda$, $\phi^{(n)}$ be as in Lemma \ref{lm:pre-tighness}. 
Let $Z$ be a stochastic process on some measurable space 
equipped with a sequence of probability measures $ \mathbb P^{(n)}$, $n\in\N$, such that the following holds.
 \begin{itemize}
\item[(a)] For any $\lambda>\bar\lambda$ we have
 \begin{align}\label{eq:phi-phi^4}
\phi^{(n)}_t(Z_t)-\phi^{(n)}_0(Z_0)=\int_0^t \lambda ( \phi^{(n)}_r(Z_r) - Z_r) dr  + \int_0^t BZ_r dr + M_t^{(n)},
 \end{align}
where $M^{(n)}$ is an $N$-dimensional martingale under $ \mathbb P^{(n)}$, with respect to the canonical filtration of $Z$, 
  with the last $N-d$ components identically zero and the first $d$ components, denoted by  $\tilde M^{(n)}$, with covariation 
  \begin{equation} \label{eq:covariation}
  [\tilde M^{(n)}, \tilde M^{(n)}]_t = \int_0^t [\nabla_v \phi_r^{(n)}(\nabla_v \phi_r^{(n)} )^\top ](Z_r) \dd r;
  \end{equation}
\item[(b)] $Z_0\sim \mu_0 $ under $ \mathbb P^{(n)}$ for any $n\in\N$.
\item[(c)] $Z_t$ under $ \mathbb P^{(n)}$  has finite 4th moment for any $t \in[0,T]$.
\end{itemize}
Then the sequence of the laws of $Z$ under $\mathbb P^{(n)}$ 
is tight.
\end{lemma}
 
\begin{proof}
Let us define the process $Y^{(n)}$ as
\begin{equation}\label{eq:Y^n}
Y^{(n)}_t:=u^{(n)}_t(Z_t)+Z_t=\phi^{(n)}_t(Z_t).
\end{equation}
We first show that the sequence of the laws of $Y^{(n)}$ under $ \mathbb P^{(n)}$ is tight.  
According to \cite[Theorem 4.10 in Chapter 2]{karatzasShreve} we need to prove that 
\begin{equation}\label{eq:KS1}
\lim_{\eta\to\infty} \sup_{n\geq 1} \mathbb  P^{(n)}(| Y^{(n)}_0 |>\eta ) =0
\end{equation}
and that for every $\varepsilon>0$
\begin{equation}\label{eq:KS2} 
\lim_{\delta\to0} \sup_{n\geq 1} \mathbb  P^{(n)} \Big ( \sup_{\substack{
s,t \in [0,T] \\|s-t|\leq \delta }} |Y^{(n)}_t-Y^{(n)}_s|>\varepsilon   \Big ) =0 .
\end{equation}
By \eqref{eq:Y^n} and recalling that $Z_0 = (V_0, X_0)$ we have 
\[
Y_0^{(n)} =  \phi^{(n)}_0(Z_0) = \begin{pmatrix}
V_0 + u^{(n)}_1(0, Z_0)\\
X_0
\end{pmatrix},
\]
so the first condition \eqref{eq:KS1} for tightness gives
\begin{align*}
\mathbb P^{(n)}(| Y^{(n)}(0) |>\eta ) &\leq  \mathbb P^{(n)}(|X_0|+|V_0|+ |u_1^{(n)}(0, Z_0) |>\eta )  \\
&\leq  \mathbb P^{(n)}(|X_0|+|V_0|+ \sup_n \|u_1^{(n)}\|_{\infty}>\eta ) ,
\end{align*} 
having used the fact that $\sup_n \| u_1^{(n)}\|_{\infty}$ is finite by Theorem \ref{thm:conv_un_back} Part (a). The right hand side of the latter inequality does not depend on $n$ because the initial condition  $(V_0, X_0 )\sim\mu_0$  is independent of $n$, so \eqref{eq:KS1} is satisfied. 

Concerning the second bound \eqref{eq:KS2} for tightness, if we can prove that 
\begin{equation}\label{eq:kolm}
  \mathbb E^{(n)}[|Y_t^{(n)} - Y_s^{(n)}|^4] \leq C |t-s|^2, \qquad { t,s\in [0,T],}
\end{equation}
holds for some positive constant $C$ independent of $n$, then we conclude exactly as in the proof of \cite[Lemma 3.12]{issoglio_russoMPb}, using
Garsia-Rodemich-Rumsey lemma (see \cite[Section 3]{BarlowYor}).
It remains to prove \eqref{eq:kolm}. To this aim, using assumption (c) we know that $Y^{(n)}_t = \phi^{(n)}_t (Z_t)$ has finite 4th moment since $ \phi^{(n)}_t  $ has linear growth. We next prove that there exists a constant $C$ depending on $T$ such that  
\begin{equation}\label{eq:Yn^2}
\mathbb E^{(n)}[|Y^{(n)}_t|^4] \leq C, \qquad t\in[0,T].
\end{equation}
 Indeed, using  \eqref{eq:phi-phi^4} and \eqref{eq:Y^n} we get 
\begin{align}\label{eq:expr_Yn} 
Y^{(n)}_t & =  Y^{(n)}_0 +  \int_0^t \lambda u^{(n)}_r(Z_r) dr  + \int_0^t BZ_r dr + M_t^{(n)}.
\end{align}
Taking the square and then the expectation, and using the results of Lemma \ref{lm:pre-tighness}, in particular that $Z_t = \psi^{(n)}_t(\phi_t^{(n)}(Z_t)) = \psi^{(n)}_t(Y^{(n)}_t)$ and  the uniform in $n$ linear growth property (v) 
of $\psi^{(n)}$,  Jensen's inequality,  BDG inequality and \eqref{eq:covariation},  one obtains, for all $t\in[0,T]$,
\begin{align*} 
\mathbb E^{(n)}[ |Y^{(n)}_t|^4 ] & \leq c\,   \mathbb E^{(n)}[ |Y^{(n)}_0|^4 ] + c\,\mathbb E^{(n)} \bigg[ \left|\int_0^t \lambda u^{(n)}_r(Z_r) dr\right|^4 \bigg] +c\, \mathbb E^{(n)}\bigg[ \left|\int_0^t BZ_r dr\right|^4 \bigg]+  c\, \mathbb E^{(n)}\bigg[ \left|M_t^{(n)}\right|^4 \bigg] \\
  & \leq c\,   \mathbb E^{(n)} [( 
 {|Z_0|} + |u_0^{(n)}( Z_0) |)^4] + c \,\lambda^4 T^4 \|u^{(n)}\|^4_{\infty} +c\, |B|^4 { t^{3} \,} \mathbb E^{(n)} \bigg[  \int_0^t  (1+ |Y_r^{(n)}|)^4 dr \bigg]\\
  &\quad + c\, { T}  (1+ \|{ \nabla_v} u^{(n)}\|_{\infty})^2 \\
\intertext{(by the assumption on the 4th moment of $\mu_0$)}
 & \leq c\, (1+ \|u^{(n)}\|_{T, 1+ \beta+\eps}^4)+ c \,  \mathbb E^{(n)} \bigg[ \int_0^t   |Y_r^{(n)}|^4 dr \bigg] \\	
 & \leq c (1+  \int_0^t \mathbb E^{(n)}  [ |Y_r^{(n)}|^4 ] dr) ,
 \end{align*}
 for some $\eps>0$, having used Theorem \ref{thm:conv_un_back} Part (b) in the last bound.  Notice that the constant $c$  is independent of $n$ (but depends on $T, B, \mu_0, \lambda$).
 Applying Gronwall's lemma we show that \eqref{eq:Yn^2} holds.
We now  notice that 
{\eqref{eq:expr_Yn} yields}
 \begin{align}\label{eq:Y-Y^4}
 Y_t^{(n)}-Y_s^{(n)}=\int_s^t  
 { \lambda u^{(n)}_r(Z_r)} dr+ \int_s^t BZ_r dr + M^{(n)}_t -M^{(n)}_s.
 \end{align}
 With similar computations as above, together with multidimensional BDG inequality (see e.g.\ \cite[Problem 3.29]{karatzasShreve}), we have
\begin{align*}
\mathbb E^{(n)} [ |Y_t^{(n)}-Y_s^{(n)}|^4 ]& \leq c\, \mathbb E^{(n)} \bigg[ \left |\int_s^t |\lambda u^{(n)}_r(Z_r)| +  |B\psi^{(n)} (Y^{(n)}_r)| dr\right|^4 \bigg]+ c\,  \mathbb E^{(n)} \bigg[ \left |M_t^{(n)}- M_s^{(n)} \right|^4\bigg] 
\intertext{(by Lemma \ref{lm:pre-tighness}-(v))}
 &\leq c\, \|u^{(n)}\|_{T, 1+\beta+\varepsilon} |t-s|^4 + c\, |t-s|^3 \mathbb E^{(n)} \bigg[ 
 \int_s^t  
 (1+ |Y^{(n)}_r|)^{4} dr 
 \bigg] \\ 
 &\quad + c\, \mathbb E^{(n)} \left[ ( \int_s^t  | \nabla_v u^{(n)}_r(Z_r) + \begin{pmatrix}
 I_d  \\
 0
 \end{pmatrix} |^2 dr )^2 \right]\\
  &\leq c \|u^{(n)}\|_{T, 1+\beta+\varepsilon} |t-s|^4 + c {|t-s|^3} \int_s^t  
 \mathbb E^{(n)} [ (1+ |Y^{(n)}_r|)^{4} ] dr +  c(1+ \|u^{(n)}\|_{T, 1+\beta+\varepsilon})^2 |t-s|^2\\
  &\leq c \|u^{(n)}\|_{T, 1+\beta+\varepsilon} |t-s|^4 + c |t-s|^{ 3}+ {c} (1+ \|u^{(n)}\|_{T, 1+\beta+\varepsilon})^2 |t-s|^2, 
\end{align*}
having used \eqref{eq:Yn^2} in the last bound. This clearly implies \eqref{eq:kolm}, as wanted. 
To conclude that the laws of $Z$ under  $\mathbb P^{(n)}$ are tight we observe that $Z = \psi^{(n)}(Y^{(n)})$ and $(\psi^{(n)})$ are  equicontinuous and have linear growth by Lemma \ref{lm:pre-tighness} Parts (ii), (v). Consequently, when restricted to any compact set, the family $(\psi^{(n)})$ is  equicontinuous and bounded, so from tightness of the laws of  $Y^{(n)} $ under $ \mathbb P^{(n)}$   the claim follows.
\end{proof}

Notice that the assumption on the finiteness of the 4th moments of $ Z_t$ in Lemma \ref{lm:tighness} is not optimal, and it could be relaxed up to requiring finite $(2+\epsilon)$-moment. However, this assumption is sufficiently general for us as we will apply this result to $Z_0 \sim\mu_0$ being equal to a Dirac delta distribution, and the finiteness of the moments for $t>0$ will follow.

\begin{lemma}\label{lm:tightnessPn}
  Let $(\Bc^{(n)})$ be a sequence in $\CT{\gamma}$
  for some $\gamma>0$ and such that $\|\Bc^{(n)}-\Bc\|_{T, -\beta-\eta} \to0$ for some $\eta >0$. 
Let $Z$ be the canonical process and  we denote by $(Z,\Pb^{(n)})$ the solution to the Stroock-Varadhan  martingale problem with respect to $\mathcal L_{\Bc^{(n)}}$ and initial condition $\mu_0$. Assume that $\mu_0$ has finite 4th moment.
 Then there exists a subsequence of $\Pb^{(n)}$ that weakly converges to some probability measure $\Pb$ and $(Z, \Pb)$ is a solution to MP$(\Bc, \mu_0)$.
\end{lemma}
\begin{remark}
Notice  that the solution $(Z, \Pb^{(n)} )$ exists by  Corollary \ref{cor:MP-Strook}.
\end{remark}
\begin{proof}[Proof of Lemma \ref{lm:tightnessPn}]
The solution $(Z, \Pb^{(n)} )$
 is also a   solution to  MP$(\Bc^{(n)}, \mu_0)$ by Proposition \ref{prop:MP-Strook}. 
   For the moment we assume the hypothesis of Lemma  \ref{lm:tighness} hold (this will be proven later), which  implies that  the sequence of $(\Pb^{(n)})$ is tight. Thus  there exists a subsequence, which we denote again with $(\Pb^{(n)})$, that weakly converges to some Borel probability measure $\Pb$ on $ C_T\R^N$. We write
\begin{equation}
\label{eq:weak converg}
\Pb^{(n)}\Rightarrow \Pb.
\end{equation} 
We now show that the couple $(Z,\Pb)$ solves the MP$(\Bc,\mu_0)$. 
Since for any $n\in \N$,  $\Pb^{(n)}\circ Z_0^{-1}= \mu_0$ the same holds for $\Pb\circ Z^{-1}_0$. To prove the martingale property, we consider $u\in \mathcal{D}_{\mathcal{L}} $ (see Remark \ref{rem:domain MP}) along with $g \in C_T \mathcal S \subset \CT{\varepsilon}$
and $f_T \in \mathcal S\subset \C_B^{1+\beta+\eps}$, for some $\varepsilon >0$, such that   $u$ is a mild solution to \eqref{eq:cauchy_MP_def}. 
We then consider the process
\begin{equation}
M^u_t:=u_t(Z_t)-u_0(Z_0)-\int_0^t g_s(Z_s)ds,
\end{equation}
and show that it is a $\Pb$-martingale.
To this aim, fixed an $n\in \N$, we consider the  solution  $u^{(n)}\in \CT{1+\beta+\eps}$ of
\begin{equation}\label{eq:un}
\begin{cases}
\Kc u^{(n)}+\langle \Bc^{(n)},\nabla_v \rangle u^{(n)} =g\\
u^{(n)}_T=f_T.
\end{cases}
\end{equation}
Such $u^{(n)}$ exists  by Theorem \ref{thm:sol_back_CP}.
By definition of solution of MP$(\Bc^{(n)},\mu_0)$ we have that 
\begin{equation}
M_t^{u^{(n)}}=u^{(n)}_t(Z_t)-u^{(n)}_0(Z_0)-\int_0^t g_s(Z_s)ds
\end{equation}
is a $\Pb^{(n)}$-martingale.
Moreover, by assumption 
$\|\Bc^{(n)}-\Bc\|_{T,-\beta-\eta}\xrightarrow{n\to\infty} 0$.
By \eqref{eq:cont_propr} in Theorem \ref{thm:conv_un_back} Part (a), it now follows that
the sequence $(u^{(n)})$ converges uniformly to $u$ and is uniformly bounded. The latter implies, the $M^{u^{(n)}}$ and $M^u$ are uniformly bounded.  By these facts and \eqref{eq:weak converg} it can be shown easily that $M^u$ is a $\Pb$-martingale  by using the definition of martingale, in particular that $\mathbb E^{\mathbb P}[ \Phi (Z_r, r\leq s) (M_t^{u} - M_s^{u})]=0$ for all bounded continuous functionals $\Phi$ on $C_s(\mathbb R^N)$ knowing that $\mathbb E^{\mathbb P^{(n)}}[ \Phi (Z_r, r\leq s) (M_t^{u^{(n)}} - M_s^{u^{(n)}})]=0$ for all bounded continuous functionals $\Phi$.

It is left to prove that we can apply Lemma \ref{lm:tighness}. Let $u^{(n)}$  and $\phi^{(n)}$ be defined as in Lemma \ref{lm:pre-tighness},  for some $\bar \lambda$ independent of $n$ and defined for the sequence $\Bc^{(n)}$.  The reader should be careful about the notation: $u^{(n)}$ in the rest of the proof should not be confused with the $u^{(n)}$ used above in the current proof, which was solution to \eqref{eq:un}.
We would like to apply It\^o's formula to $\phi_t^{(n)}(Z_t)$ under $\mathbb P^{(n)}$ for fixed $n$, but $\phi^{(n)}$ does not necessarily belong to $C^{1,2}([0,T]\times \R^N)$, hence  we perform a further smoothing procedure and to this aim we momentarily drop the superscript $(n)$, everywhere except in $\mathbb P^{(n)}$ and $\Bc^{(n)}$.
Recall that $(Z, \mathbb P^{(n)})$ solves the Stroock-Varadhan martingale problem with respect to $\mathcal L_{\Bc^{(n)}}$, therefore there exists a probability space $(\Omega, \mathbb P^{(n)})$ and a Brownian motion $W$ and a process $Z=(V, X)$ such that 
\begin{equation}\label{eq:sv}
\begin{cases}
\dd V_t = \Big (\Bc^{(n)}_t(Z_t) + \Bc_0 Z_t \Big) \dd t   +     \dd W_t       \\
\dd X_t = \Bc_1 Z_t \dd t\\
Z_0 \sim \mu_0. 
\end{cases}
\end{equation}
Indeed, equation \eqref{eq:sv} admits existence in law and furthermore $Z$ under $\mathbb P^{(n)}$ has finite 4th moment, hence assumptions (b) and (c) of Lemma \ref{lm:tighness} are satisfied. It remains to prove assumption (a).
By  Proposition \ref{prop_lie_sol},  $u = u^{(n)}$ is a strong Lie  solution of \eqref{eq:Zvonkin PDE reg} in $\Cc^{2+\alpha}_{B,T}$ for any $\alpha\in(0,1)$. 
 Let us denote by  $u^{(m)}$ a suitable  approximating sequence in $C^{1,2}([0,T]\times \R^N)$, which exists by Lemma \ref{lem:density_intrinsic_spaces} Part (a). Setting $\phi_t^{(m)}(z):= z + u^{(m)}(z) $ we compute $d\phi^{(m)}_t(Z_t)$ by applying It\^o's formula component-wise:
  we obtain
\begin{equation}\label{eq:nabla_phi_m}
\phi^{(m)}_t(Z_t) - \phi^{(m)}_0(Z_0) = \int_0^t \partial_t \phi^{(m)}_s (Z_s) ds + \int_0^t \nabla \phi^{(m)}_s (Z_s) dZ_s + \frac12 \sum_{i,j=1}^N \partial_{z_i, z_j} \phi^{(m)}_s(Z_s) d [Z^i, Z^j]_s.
\end{equation}
The second term on the RHS of \eqref{eq:nabla_phi_m} gives 
\begin{align}
\int_0^t  & \nabla \phi^{(m)}_s (Z_s) dZ_s \\
&= 
\int_0^t 
\langle 
\begin{pmatrix}
  \Bc_s^{(n)} (Z_s) \\ 
 0
\end{pmatrix}, 
\nabla \phi^{(m)}_s (Z_s) 
\rangle  d s  
+ \int_0^t \langle B Z_s, \nabla \phi^{(m)}_s (Z_s) \rangle ds + 
\begin{pmatrix}
\int_0^t \nabla_v \phi^{(m)}_s (Z_s)  d W_s \\ 
0
\end{pmatrix} \\
& = 
\int_0^t 
\langle  \Bc_s^{(n)} (Z_s), \nabla_v \phi^{(m)}_s (Z_s)  \rangle  d s  
+ \int_0^t \langle B Z_s, \nabla \phi^{(m)}_s (Z_s) \rangle ds
+ 
\begin{pmatrix}
 \int_0^t \nabla_v \phi^{(m)}_s (Z_s)  d W_s \\ 
 0
\end{pmatrix}, 
\end{align} 
where the previous identity are to be intended for each component of the vector $\phi^{(m)}$. The last term on the RHS of \eqref{eq:nabla_phi_m} can be written, again component-wise,  as  
\[
\frac12 \sum_{i,j=1}^N \partial_{z_i, z_j} \phi^{(m)}_s(Z_s) d [Z^i, Z^j]_s = \frac12 \Delta_v  \phi^{(m)}_s (Z_s) ds,
\]
so \eqref{eq:nabla_phi_m} becomes 
\begin{align}\label{eq:nabla_phi_mii}
\phi^{(m)}_t(Z_t) - \phi^{(m)}_0(Z_0) 
&= \int_0^t [ \partial_t \phi^{(m)}_s (Z_s) 
+  \langle B Z_s, \nabla \phi^{(m)}_s (Z_s) \rangle 
+\frac12 \Delta_v  \phi^{(m)}_s (Z_s) ] ds \\
 &\quad + \int_0^t \langle  \Bc_s^{(n)} (Z_s), \nabla_v \phi^{(m)}_s (Z_s)  \rangle  d s  
+\begin{pmatrix}
\int_0^t \nabla_v \phi^{(m)}_s (Z_s)  d W_s \\ 
 0
\end{pmatrix}.
\end{align}
We can rewrite this equality in terms of $u^{(m)}$ as
\begin{align}\label{eq:nabla_u_mii}
u^{(m)}_t(Z_t) &+ Z_t - u^{(m)}_0(Z_0) - Z_0  \\
&= \int_0^t [ \partial_t u^{(m)}_s (Z_s) 
+  \langle B Z_s, \nabla u^{(m)}_s (Z_s) \rangle 
+\frac12 \Delta_v  u^{(m)}_s (Z_s) ] ds  + \int_0^t   B Z_s ds \\
 & \quad + \int_0^t \langle  \Bc_s^{(n)} (Z_s), \nabla_v u^{(m)}_s (Z_s)  \rangle  d s + \int_0^t 
\begin{pmatrix} 
  \Bc_s^{(n)} (Z_s)\\
  0
  \end{pmatrix}   ds
+\begin{pmatrix}
\int_0^t \nabla_v u^{(m)}_s (Z_s)  d W_s \\ 
 0
\end{pmatrix}
+ 
\begin{pmatrix}
W_t \\ 
 0
\end{pmatrix}\\
&=\int_0^t [ Y u^{(m)}_s (Z_s) 
+\frac12 \Delta_v  u^{(m)}_s (Z_s) ] ds   + \int_0^t \langle  \Bc_s^{(n)} (Z_s), \nabla_v u^{(m)}_s (Z_s)  \rangle  d s \\
&  \quad 
 + \int_0^t   B Z_s ds 
+ \int_0^t 
\begin{pmatrix} 
  \Bc_s^{(n)} (Z_s)\\
  0
  \end{pmatrix}   ds
+\begin{pmatrix}
\int_0^t \nabla_v u^{(m)}_s (Z_s)  d W_s \\ 
 0
\end{pmatrix}
+ 
\begin{pmatrix}
W_t \\ 
 0
\end{pmatrix}
.
\end{align}
By   Lemma \ref{lem:density_intrinsic_spaces} Part (b)
we get that $u^{(m)}, Y u^{(m)}, \Delta_v  u^{(m)}$ and $\nabla_v u^{(m)} $ converge uniformly on compact sets to $u, Y u, \Delta_v  u$ and $\nabla_v u$, respectively, 
hence the ucp limit  of  \eqref{eq:nabla_u_mii} gives 
\begin{align}
u_t(Z_t) + Z_t - u_0(Z_0) - Z_0  
&=\int_0^t [ Y u_s (Z_s) 
+\frac12 \Delta_v  u_s (Z_s) ] ds   + \int_0^t \langle  \Bc_s^{(n)} (Z_s), \nabla_v u_s (Z_s)  \rangle  d s \\
&\quad
 + \int_0^t   B Z_s ds 
+ \int_0^t 
\begin{pmatrix} 
  \Bc_s^{(n)} (Z_s)\\
  0
  \end{pmatrix}   ds
+\begin{pmatrix}
\int_0^t \nabla_v u_s (Z_s)  d W_s \\ 
 0
\end{pmatrix}
+ 
\begin{pmatrix}
W_t \\ 
 0
\end{pmatrix}.
\end{align}
As $u$ is a strong Lie solution of \eqref{eq:Zvonkin PDE reg},  the PDE in the latter is satisfied everywhere on $[0,T]\times \R^N$, and thus   
\begin{align}
u_t(Z_t) + Z_t - u_0(Z_0) - Z_0  & = 
\lambda \int_0^t u_s (Z_s) ds 
 + \int_0^t   B Z_s ds 
+\begin{pmatrix}
\int_0^t \nabla_v u_s (Z_s)  d W_s \\ 
 0
\end{pmatrix}
+ 
\begin{pmatrix}
W_t \\ 
 0
\end{pmatrix},
\end{align}
 which can be seen to be equality \eqref{eq:phi-phi^4}, where the martingale $M$ is given by 
\[M_t =  \begin{pmatrix}
\int_0^t \nabla_v \phi_s (Z_s)  d W_s \\ 
 0
\end{pmatrix},
\]
having  used the fact that $\phi_t(x)= x + u_t(x)$. This concludes the proof that Lemma \ref{lm:tighness} can be applied. 
\end{proof}
\begin{proposition}\label{prop:ex_MP}
Let Assumption \ref{ass:kolmogorov_op}  hold and $\Bc \in C_T \C^{-\beta}_B$ with $\beta \in (0,1/2)$. Let $\mu_0$ be a probability law on $\R^N$ with finite 4th moment. Then  MP$(\Bc, \mu_0)$ admits existence of a solution according to Definition \ref{def:solMP}.
\end{proposition}

\begin{proof}
We define
\begin{equation}\label{eq:b1n}
\Bc_t^{(n)}: =  \Phi_n \ast \Bc_t
\end{equation}
as in \eqref{eq:b_n bis}. By Lemma \ref{lem:regul_g}, $\Bc^{(n)} \in \CT{\gamma}$
for all $\gamma>0$, 
and for any  
$\eta >0$
$$ \Vert \Bc^{(n)} - \Bc \Vert_{T,-\beta-\eta}  
  \xrightarrow[n\to+\infty] \  0 .$$
Let $Z$ be the canonical process and  we denote by $(Z,\Pb^{(n)})$ the solution to the Stroock-Varadhan  martingale problem with respect to $\mathcal L_{\Bc^{(n)}}$.   Lemma \ref{lm:tightnessPn} allows to conclude existence of a solution  $(Z,\Pb) $.
\end{proof}

We now prove existence and uniqueness of solutions to MP$(\Bc, \mu_0)$,  for a general distribution law $\mu_0$. Concerning existence, the key tool is the superposition of solutions to MP$(\Bc, \delta_z)$ for any $z\in\R^N$, with $\delta_z$ clearly satisfying the finite $4$th moment assumption required to apply Proposition \ref{prop:ex_MP}.

\begin{theorem}\label{thm:ex_un_MP}
Let Assumptions \ref{ass:kolmogorov_op} hold, $\Bc \in C_T\C_B^{-\beta}$ with $\beta\in (0,1/2)$, and $\mu_0$ be a generic distribution law on $\R^N$. Then MP$(\Bc, \mu_0)$ admits existence and uniqueness of a solution according to Definition \ref{def:solMP}.
\end{theorem}
\begin{proof}

{\em Uniqueness.} 
We want to  apply  Theorem \ref{thm:EK-extended} from the Appendix, so we need to verify that
Property M is satisfied for MP$(\Bc,\mu_0; t_0)$, for any initial 
condition $\mu_0$ and any $t_0\in[0,T)$ (see Definition \ref{def:Prop_P} in Appendix \ref{sec:MP Ethier-Kurtz} for the definition of Property M). Without restriction of generality we set $t_0=0$.
Assume that $(Z^1, \mathbb P^1)$ and  $(Z^2, \mathbb P^2)$ are two solutions with initial condition $\mu_0$. Without loss of generality, we can choose $Z^1=Z^2=Z$ where $Z$ is the canonical process. Let now 
 $g \in C_T \mathcal S$
 and $u$ be the unique solution of $\mathcal Lu = g $ with terminal condition $u_T =0$. By the definition of solution of MP$(\Bc,\mu_0)$  
\begin{equation}
\mathbb E^{\mathbb P^i} \left[ \underbrace{ u_T( Z_T)}_{=0} -  u_0(Z_0) - \int_0^T g_s(Z_s) \dd s\right] =0,\qquad \text{for } i=1,2,
\end{equation}
and $\Pb^1\circ Z^{-1}_0=\Pb^2\circ Z^{-1}_0=\mu_0$. As a result,
\begin{equation}
\Eb^{\Pb^1} \left[\int_0^T g_s(Z_s)ds\right]=\Eb^{\Pb^2} \left[\int_0^T g_s(Z_s)ds\right].
\end{equation}
By arbitrariness of 
 $g \in C_T \mathcal S$
 then $\Pb^1$ and $\Pb^2$ have the same marginals for almost all $s\in[0,T]$, and since  $Z$ is continuous they have the same marginals for all $s \in[0,T]$,  that is Property M is satisfied for MP$(\Bc,\mu_0)$, for any initial condition $\mu_0$.

{\em Existence.} 
Given any $z\in \R^N$, by Proposition \ref{prop:ex_MP} there exists a solution $(Z, \mathbb P_z)$ to MP$(\Bc, \delta_z)$ since $\delta_z$ has finite 4th moment. As the solution is unique, the map $z\mapsto \mathbb P_z$ is measurable from $\R^N$ to the space of probability measures. This can be seen  with the same argument as in \cite[Exercise 6.7.4]{stroock_varadhan}. 
Given an arbitrary probability law $\mu_0$, we define $\mathbb P$ for a Borel set $B$ of the canonical space, by $\mathbb P (B):= \int_{\R^N} \mathbb P_z (B) d\mu_0(z) $.
The couple $(Z, \mathbb P)$ is a solution to MP$(\Bc, \mu_0)$, and we recall that the details of this arguments can be seen in the proof of existence of \cite[Theorem 4.5]{issoglio_russoMPb}, using the  current definition of martingale problem.
\end{proof}

The next proposition is the extension of the convergence result stated in Lemma \ref{lm:tightnessPn}, when removing the assumption on the finiteness of the 4th moment of the initial condition. 

\begin{proposition}\label{pr:convergence}
 Let $\mu_0$ be any probability law on $\R^N$. Let $\Bc^{(n)}, \Bc$ be as in Lemma \ref{lm:tightnessPn}.
Let $(Z, \mathbb P^{(n)})$ be a solution of the Stroock-Varadhan martingale problem with respect to $\mathcal L_{\Bc^{(n)}}$ with initial condition $\mu_0$. Then $\mathbb P^{(n)}$ weakly converges to $\mathbb P$, where $(Z, \mathbb P)$ is solution to MP$(\Bc, \mu_0)$.
\end{proposition}

\begin{proof}
  
For every $n$ we proceed to the {\em disintegration} of $\mathbb P^{(n)}$ as in \cite[Fact 2 in the proof of Theorem 4.5]{issoglio_russoMPb}. We write 
\begin{equation}\label{eq:E1}
\mathbb P^{(n)} = \int \mathbb P^{(n)}_z d\mu_0 (z).
\end{equation}
We show that $\mathbb P^{(n)}_z$ solves the Stroock-Varadhan martingale problem with respect to $\mathcal L_{\Bc^{(n)}}$ and initial condition $\delta_z$; this can be done as in Fact 2 of previously mentioned
\cite[Theorem 4.5]{issoglio_russoMPb}.
By the uniqueness of the Stroock-Varadhan martingale problem, $\mathbb P^{(n)}_z$ in \eqref{eq:E1} is the unique solution of that problem. Since $\mu_0=\delta_z$ has finite 4-th moment, we apply Lemma \ref{lm:tightnessPn} to show that for all $z\in\R^N$ we have 
\[
\mathbb P^{(n)}_z \Rightarrow \mathbb Q_z,
\]
where $\mathbb Q_z$ solves the MP$(\Bc, \delta_z)$. 
The map $z \mapsto \mathbb Q_z$ is measurable thanks to the uniqueness by Theorem \ref{thm:ex_un_MP} and again an argument as \cite[Exercise 6.7.4]{stroock_varadhan}.
By superposition, $\mathbb Q: = \int_{\R^N} \mathbb Q_z d\mu_0(z)$ solves the MP$(\Bc, \mu_0)$ which equals $\mathbb P$ by uniqueness, see again Theorem  \ref{thm:ex_un_MP}. It remains to show that 
$
\mathbb P^{(n)} \Rightarrow \mathbb P.
$
Let $F: C_T\R^N \to \R$ be bounded and continuous. We get by Fubini 
\begin{equation}\label{eq:E2}
\mathbb E^{\mathbb P^{(n)}} (F(Z)) = \int_{\R^N} \mathbb E^{\mathbb P^{(n)}_z} (F(Z)) d\mu_0(z).
\end{equation}
Using the fact that $\mathbb P^{(n)}_z \Rightarrow \mathbb Q_z$
for all $z\in\R^N$, taking the limit in \eqref{eq:E2} and using dominated convergence theorem we get 
\[
\lim_{n\to\infty} \mathbb E^{\mathbb P^{(n)}} (F(Z))  = \int_{\R^N} \mathbb E^{\mathbb Q_z} (F(Z)) d\mu_0(z)= \mathbb E^{\mathbb P} (F(Z)),
\]
and the proof is concluded. 
\end{proof}

In the next result we see that if the solution $Z_t$ of the singular martingale problem  has a density $v_t$ in a suitable space, then the function $t \mapsto v_t$ satisfies the Fokker-Planck equation, as expected.  
\begin{proposition}
\label{prop:MP sol FP}
Let Assumption \ref{ass:kolmogorov_op} holds and $\Bc \in C_T \C^{-\beta}_B$ with $\beta \in (0,1/2)$. {Let $\mu_0$ be a probability measure on $\R^N$ with density $u^0\in \C_B^{\beta+\varepsilon}$, for some $\varepsilon>0$.} Let $(Z,\mathbb{P})$ be the solution to the MP$(\Bc, \mu_0)$. Then $Z_t$ admits a law density $u_t$ with $u\in \CT{\beta+\varepsilon}$. Furthermore, $u$ is a weak solution of the linear Fokker-Planck equation \eqref{eq:FP cauchy probl} where $F(u)b = \Bc$, that is for every $\varphi\in \mathcal{S}$ we have
\begin{equation}
  \< u_t|\, \varphi\>= \<u^0 |\, \varphi\>+ \int_0^t \<u_s|\, \Ac \, \varphi\> \, ds +\sum_{i=1}^d \int_0^t \big\< u_s \Bc_i|\, \partial_{z_i} \varphi\big\> ds , \qquad t\in [0,T],
\end{equation}
where $\Ac$ was defined in \eqref{eq:kolm_const}. 
\end{proposition}

\begin{proof}
Denote by $u$ and $u^{(n)}$ the unique mild, and thus weak, solutions in $\CT{\beta+\varepsilon}$ to \eqref{eq:FP cauchy probl} with $F(u)b=\Bc$ and $F(u)b=\Bc^{(n)}$, respectively, which exist by Theorem \ref{th:well_posedness_FP}. Here $(\Bc^{(n)})$ is defined as in \eqref{eq:b_n bis}. Recall that, by Theorem \ref{thm:conv_un}, $u^{(n)}$ converges uniformly to $u$ on $[0,T]\times \R^N$. 

Furthermore, once more by Theorem \ref{thm:conv_un}, $u^{(n)}$ is non-negative and $u^{(n)}\in L^1(\R^N)$. By \cite[Theorem 2.6]{figalli}, there exists a solution $(Z^{(n)}, \mathbb P^{(n)})$ to the Stroock-Varadhan martingale problem with $\mathcal L_{ \Bc^{(n)}}$ and initial condition $\mu_0$, whose law  density at time $t$ is exactly $u_t^{(n)}$. Owing to Proposition \ref{prop:MP-Strook}, the pair $(Z^{(n)},\Pb^{(n)})$ coincides with the unique solution to MP$(\Bc^{(n)}, \mu_0)$. Therefore, by Theorem \ref{thm:ex_un_MP} (uniqueness) and Proposition \ref{pr:convergence}, $u_t^{(n)}(z)d z$ converges to the law of  $Z_t$ under $\Pb$, for any $t\in [0,T]$. This implies that $u_t$ is the density of $Z_t$ under $\Pb$.
\end{proof}

\subsection{The McKean-Vlasov case}\label{sec:MKMP}
Throughout this section we assume that Assumption \ref{ass:kolmogorov_op} and \ref{ass:beta and b}  hold, and that Assumption \ref{ass:phi} holds for $\Phi$ being $F$ and $\tilde F$. 
The definition of solution to MP extends easily to the McKean-Vlasov case \eqref{eq:mkv_singular}  and \eqref{eq:initial_cond_Z}, recalled below for ease of reading:
\begin{equation}\label{eq:mkv_singularII}
\begin{cases}
\dd V_t = \Big( F\big( u_t (Z_t)  \big) b_t(Z_t) + \Bc_0 Z_t \Big) \dd t   +     \dd W_t       \\
\dd X_t = \Bc_1 Z_t \dd t \\
\mu_{Z_t}(\dd z) = u_t(z) \dd z,
\end{cases}
\end{equation}
for some initial condition $Z_0 \sim \mu_0 = u^{0}(z) \dd z$.

To formulate \eqref{eq:mkv_singularII} as a martingale problem we resort to the linear singular martingale problem MP$( \Bc, \mu_0)$ introduced in Section \ref{sec:linear MP}, with $\Bc=F(u) b $ and $u \in C_T \C_b \cap L^\infty_T\C_B^{\beta+\eps}$.  
Precisely, 
we give the following definition.

\begin{definition}\label{def:MKMP}
Let  the McKean-Vlasov Martingale Problem  formally given by \eqref{eq:mkv_singular} 
 and \eqref{eq:initial_cond_Z} be denoted by  MKMP$( F, b, u^0)$.
 We say that   MKMP$( F, b, u^0)$ admits {\em existence} if there exists a triple $(Z, \mathbb P, u) $  where $Z$ is a continuous stochastic process, $\mathbb P$ is a probability measure, and $u \in C_T \C_b \cap L^\infty_T\C_B^{\beta+\eps}$ for some $\eps>0$, such that   the law of $Z_t$ under $\mathbb P$ is $u_t$ and 
 $(Z, \mathbb P) $ is a solution to  MP$( \Bc, \mu_0)$ where  $\Bc = 
 F(u) b $.
 
 We say that MKMP$( F, b, u^0)$ admits {\em uniqueness} if given any two solutions $(Z, \mathbb P, u) $ and $(\hat Z, \hat {\mathbb P}, \hat u) $ such that $u,\hat u \in C_T \C_b \cap L^\infty_T\C_B^{\beta+\eps}$ for some $\eps>0$,  we have that $u= \hat u$ and the law of $Z$ under $\mathbb P$ is the same as the  law of $\hat Z$ under $\hat {\mathbb P}$.
\end{definition}

We now prove the main result, which was stated in Theorem \ref{thm:ex_uniq_McKeanSDE_intro} and is repeated here for completeness.

\begin{theorem}\label{thm:ex_uniq_McKeanSDE}
 Let Assumptions \ref{ass:kolmogorov_op}  and 
 \ref{ass:beta and b} hold. 
 Let also Assumption \ref{ass:phi} hold with $\Phi=F,\tilde F$. 
Then  MKMP$( F, b,u^0)$ admits uniqueness.
\end{theorem}

\begin{remark}\label{rem:regularity_ut}
Let us   denote by $u_t$ the probability density of the solution of MKMP$( F,b,u^0)$ at time $t\in [0,T]$. We now illustrate the fact that the space-regularity of the solution  decreases as $\beta$ increases, as expected. To this aim, let us suppose that $u^0$ is a bounded and smooth function.   The optimal regularity of the solution $u$ is obtained when regarding $ u^0$ as an element of  $ \C_B^{1-\beta-\eta}$ for every $\eta >0$ small enough.   
Indeed we apply Theorem \ref{th:well_posedness_FP} and Remark \ref{rem:existence_uniq} with $ \beta +\eps = 1-\beta - \eta$, and we see that $u_t \in \C_B^{1-\beta-\eta}$, hence the regularity index $1-\beta - \eta$ decreases as $\beta $ increases.
\end{remark}

\begin{proof}
{\em Existence.}
We construct a solution  $(Z, \mathbb P, u) $ with limiting arguments as follows. 
Let $(b^{(n)}) \subset L^\infty_T \C_B^\alpha$,
for all $\alpha>0$, be the sequence  defined as in \eqref{eq:b_n bis} so that $b^{(n)} \to b$ in  $L^\infty_T\C_B^{-\beta-\eta}$ for some $\eta>0$. Let $ u^{(n)}\in \CT{\beta+\eps}$ for some $\varepsilon>0$  be the corresponding mild solution to the non-linear FP Cauchy problem \eqref{eq:FP cauchy probl_regul}, which exists by Theorem \ref{thm:conv_un}, Part (a), and moreover, $u^{(n)}$ is a non-negative and bounded function which defines a non-negative finite measure $\mu_t^{(n)}$ with density $u_t^{(n)}$.
By Proposition \ref{prop:mild_weak_equiv}, $u^{(n)}$ is also a weak solution and thus a solution in the sense of distributions.
Hence we can apply \cite[Theorem 2.6]{figalli} to infer that there exists a solution $(Z^{(n)}, \mathbb P^{(n)})$ to the Stroock-Varadhan martingale problem with $\mathcal L_{ \Bc^{(n)}}$, where $\Bc^{(n)}:=F(u^{(n)})b^{(n)}$, and initial condition $\mu_0$,  
whose law  density at time $t$ is exactly $u_t^{(n)}$.  Without loss of generality, let us consider the solution $(Z, \mathbb P^{(n)})$, where $Z$ is the canonical process  on the canonical space $ C_T \R^N$.  
By the continuity of $F$ in H\"older spaces (Lemma \ref{lem:F lipsch est} with $\Phi=F$) and by continuity of the pointwise product (Proposition \ref{prop:bony_prod}) we obtain that $\Bc^{(n)}   \to \Bc:=F(u)b $  in $L^\infty_T\C_B^{-\beta-\eta}$ for some $\eta>0$, where $u$ is the mild solution to the FP Cauchy problem \eqref{eq:FP cauchy probl}.
By the continuity results on the PDE \eqref{eq:FP cauchy probl_regul}, see Theorem \ref{thm:conv_un}, Part (b), we have  $u^{(n)} \to u$ in $\CT{\beta+\eps}$
for some $\eps>0$.
 By Proposition \ref{pr:convergence},
 we know that $ \mathbb  P^{(n)} \to  \mathbb  P$, with   
 $(Z, \mathbb  P) $ being a solution to MP$(\Bc, \mu_0)$. 
Recall that
 $u_t^{(n)}$ is the law of $Z_t$ under $\mathbb P^{(n)}$.
Since, as mentioned earlier $u^{(n)} \to u$ in $\CT{\beta+\eps}$
  so $u_t$ is the law of $Z_t$ under $\mathbb P$ and
 by Definition \ref{def:MKMP}, we also have that $(Z, \mathbb  P, u) $ is
 a solution to MKMP$( F, b, u^0)$.

\noindent{\em Uniqueness.}
Let $(Z, \mathbb P, u) $ and $(\hat Z, \hat {\mathbb P}, \hat u) $ be two solutions to  MKMP$(F,b,  u^0)$. Then  $u, \hat u \in C_T \C_b \cap L^\infty_T\C_B^{\beta+\eps},$  for some $\eps > 0$, by definition. So, by Proposition \ref{prop:MP sol FP} with $\Bc = F(u)b$ and  $\Bc = F(\hat u)b$ respectively, we have that $u$ and $\hat u$ are both weak solutions to the nonlinear FP singular Cauchy problem \eqref{eq:FP cauchy probl}. By Proposition \ref{prop:mild_weak_equiv}
and  Theorem \ref{th:well_posedness_FP},
that Cauchy problem admits uniqueness in  $\C _T \C_b \cap L_T^\infty \C_B^{\beta+\eps}$;
consequently  $u = \hat u$.
By definition of solution to  MKMP$(F,b, u^0)$, this means that both $(Z, \mathbb P)$ and $(\hat Z, \hat{\mathbb P})$ are solutions to the same linear singular martingale problem, namely MP$(F(u) b, \mu_0),$ where $\mu_0$ is a law with density $u^0$. By Theorem \ref{thm:ex_un_MP}, we have uniqueness for the MP$(F(u)b, \mu_0)$ and hence the law of $Z$ under $\mathbb P $ is the same as the law of $\hat Z$ under $\hat{\mathbb P}$, which concludes the proof.
\end{proof}


\appendix

\section{Martingale Problem: Uniqueness}\label{sec:MP Ethier-Kurtz}
The purpose of this Appendix is to prove uniqueness of the Martingale Problem using  Ethier-Kurtz techniques, which allow to show uniqueness of the law knowing uniqueness of the time-marginals.

Throughout this section we assume that Assumption \ref{ass:kolmogorov_op} holds and $\Bc \in C_T \C^{-\beta}_B$ with $\beta \in (0,1/2)$.

\begin{definition}\label{def:Prop_P}
Let ${\mu_0}$ be a Borel probability measure on $\R^N$, and $t_0 \in [0,T)$.
We say that {\em Property M} holds for MP$(\Bc, \mu_0; t_0)$ if given $(Z, \mathbb P)$ solution to the MP$(\Bc, \mu_0; t_0)$ then the marginal laws are uniquely determined, that is if $(Z^1, \mathbb P^1)$ and $(Z^2, \mathbb P^2)$ are two solutions then 
\begin{equation}
 \mathbb P^1 \circ (Z^1)^{-1}_t= \mathbb P^2 \circ (Z^2)^{-1}_t,\qquad t\in [t_0,T].
\end{equation}
\end{definition}

\begin{theorem}\label{thm:EK-extended}
If Property M holds for MP$(\Bc,\mu_0; t_0)$  for every initial condition $\mu_0$ and every $t_0 \in [0,T)$, then we have uniqueness of  the MP$(\Bc, \mu_0)$ for every $\mu_0$.
\end{theorem}
\begin{proof}
By Remark \ref{rm:ZP}, it is enough to prove uniqueness on the canonical space $C_T \R^N$, hence we let $Z$ be the canonical process and consider two solutions $(Z, \mathbb P)$
and $(Z,\bar { \mathbb P})$ to MP$(\Bc,\mu_0)$. 
It is  sufficient to prove that for every $n \in \N$,  $0\leq t_0 <t_1 <\ldots <t_n < T$ and any strictly positive {bounded} Borel functions $h_0,\ldots,h_n$ we have
\begin{equation}
\label{eq:ext uniqueness}
\Eb^\Pb \left[ \prod_{i=0}^n h_i(Z_{t_i})\right]= \Eb^{\bar{\Pb}} \left[ \prod_{i=0}^n h_i(\bar{Z}_{t_i})\right] .
\end{equation}
Indeed, the same equality still holds true when $h_i =\mathbf 1_{F_i}$ for closed sets $F_i$ (see \cite[proof of Theorem 1.1.1]{stroock_varadhan} for the approximation of an indicator function of a closed set with positive functions), hence the law of $(Z_{t_0}, \ldots, Z_{t_n})$ is identified.

We proceed by induction on $n$. The base step $n=0$ holds by Property M. Let $n\geq 1$ and assume
\eqref{eq:ext uniqueness} holds  
with $n$ replaced by $n-1$. 
We can define two new probability measures on $ C_T \R^N $
\begin{equation}
\mathbb Q(H):= \frac{1}{\Eb^\Pb \left[ \prod_{i=0}^{n-1} h_i(Z_{t_i})\right]}E^\Pb \left[ \mathbf{1}_H \prod_{i=0}^{n-1} h_i(Z_{t_i})\right], 
\qquad \bar{\mathbb Q}(H):= \frac{1}{\Eb^{\bar{\Pb}} \left[ \prod_{i=0}^{n-1} h_i({Z}_{t_i})\right]}E^{\bar{\Pb}} \left[ \mathbf{1}_H \prod_{i=0}^{n-1} h_i({Z}_{t_i})\right],
\end{equation}
for every $H\in \sigma\left(Z_t,\, 0\leq t\leq T\right)$. 
We show that the law of $Z_{t_{n-1}}$ under $\mathbb Q$ is the same as the law  of $Z_{t_{n-1}}$ under $\bar {\mathbb Q}$. Indeed for any strictly positive  bounded Borel function $l$ we have, using induction step \eqref{eq:ext uniqueness}  with $n$ replaced by $n-1$,
\begin{align*}
\mathbb E^{\mathbb Q}[l (Z_{t_{n-1}})]
&= \frac{1}{\Eb^\Pb \left[ \prod_{i=0}^{n-1} h_i(Z_{t_i})\right]}E^\Pb \left[ l (Z_{t_{n-1}}) \prod_{i=0}^{n-1} h_i(Z_{t_i})\right]\\
&= \frac{1}{\Eb^{\bar\Pb} \left[ \prod_{i=0}^{n-1} h_i(Z_{t_i})\right]}E^{\bar \Pb} \left[ l (Z_{t_{n-1}}) \prod_{i=0}^{n-1} h_i(Z_{t_i})\right] =\mathbb E^{\bar{\mathbb Q}}[l (Z_{t_{n-1}})].
\end{align*}
We can thus denote by  $\mu_{n-1}:= \mathbb Q \circ Z^{-1}_{t_{n-1}} = \bar{ \mathbb Q }\circ Z^{-1}_{t_{n-1}}$. 
Next we want to show that $(Z_{\vert_{[t_{n-1}, T]}}, \mathbb Q)$ and $(Z_{\vert_{[t_{n-1}, T]}}, \bar{ \mathbb Q})$ solve  MP$(\Bc, \mu_{n-1}; t_{n-1})$.
We know that, for any $u $ such that  $\mathcal L u =g \in C_T\mathcal S$  and $u(T, \cdot) \in \mathcal S$,
\begin{equation}
 M^{u}_t - M^{u}_{t_n} = u_{t}( Z_t)-u_{t_{n-1}}( Z_{t_{n-1}})-\int_{t_{n-1}}^t   g_s ( Z_s)ds, \qquad t_{n-1}< t \leq T,
\end{equation}
is a martingale under $\mathbb P$ and under $\bar {\mathbb P}$.
Let $t_{n-1}< r< t \leq T$ and $H \in \sigma (Z_s, 0\leq s\leq r)$ be arbitrary, then 
\begin{align*}
\Eb^{\mathbb Q} &\left[ \left(M^{u}_t-M^{u}_r \right) \mathbf{1}_H\right]
= \frac{1}{\Eb^\Pb \left[ \prod_{i=0}^{n-1} h_i(Z_{t_i})\right]}E^\Pb \left[ \left(M^{u}_t -M^{u}_r \right) \mathbf{1}_H \cdot \prod_{i=0}^{n-1} h_i(Z_{t_i})\right]=0,
\end{align*}
the latter being 0 since $(Z,\mathbb P)$ is a solution to MP$(\Bc, \mu_0)$ and  $\mathbf{1}_H \cdot \prod_{i=0}^{n-1} h_i(Z_{t_i})$ is bounded measurable with respect to $\sigma (Z_s, 0\leq s\leq r)$.
Thus $( M^{u}_t - M^{u}_{t_n} )_{t\in[t_{n-1},T]}$ is a $\mathbb Q$-martingale. Analogously, 
it is also a $\bar{\mathbb Q}$-martingale. Hence both $(Z_{\vert_{[t_{n-1}, T]}}, \mathbb Q)$ and $(Z_{\vert_{[t_{n-1}, T]}}, \bar {\mathbb Q})$ are solutions to MP$(\Bc, \mu_{n-1}; t_{n-1})$. Thus by Property M we know that the marginals at time $t_n$ of the solutions are the same, and thus  
\[
\Eb^{\mathbb Q} [h_n(Z_{t_n}) ] = \Eb^{\bar {\mathbb Q}} [h_n(Z_{t_n}) ] .
\]
Employing the definitions of $\mathbb Q$ and of $\bar{\mathbb Q}$, this implies
\[
 \frac{1}{\Eb^\Pb \left[ \prod_{i=0}^{n-1} h_i(Z_{t_i})\right]} \Eb^{\mathbb P} \left [h_n(Z_{t_n}) \prod_{i=0}^{n-1} h_i(Z_{t_i})\right] =  \frac{1}{\Eb^{\bar \Pb} \left[ \prod_{i=0}^{n-1} h_i(Z_{t_i})\right]} \Eb^{\bar {\mathbb P}} \left [h_n(Z_{t_n}) \prod_{i=0}^{n-1} h_i(Z_{t_i}) \right] .
\]
Finally, by the inductive assumption $\Eb^\Pb \left[ \prod_{i=0}^{n-1} h_i(Z_{t_i})\right]=  \Eb^{\bar \Pb} \left[ \prod_{i=0}^{n-1} h_i(Z_{t_i})\right]$, we obtain \eqref{eq:ext uniqueness} as wanted.
\end{proof}

\section{Schauder Estimates (proof of Theorem \ref{th:schauder_first})}\label{app:proof_Schauder}

Throughout this section, we will denote by $C$, indistinctly, any positive constant. In each statement below, $C$ depends on a set of parameters, which will be specified in each case. 

We start by recalling some known Gaussian bounds on the the fundamental solution $\Gamma$ of $\Kc$ defined in \eqref{eq:Gamma general Langevin}, and its derivatives. 
\begin{notation}
Let $\n=(\n_1,\dots,\n_N)\in\mathbb{N}_0^N $ be a multi-index. We define the $B$-length of $\n$ as
\begin{equation}
[\n]_B:=
 \sum_{j=0}^{\rr}(2j+1)\sum_{i=\bar{d}_{j-1}+1}^{\bar{d}_j}\n_{i},
\end{equation}
where $ \bar{d}_j$ are defined  in terms of ${d}_j$, see  \eqref{eq:elements_d}.
Moreover, as usual
 $\p^\n_x=\p_{x_1}^{\n_1}\cdots \p_{x_N}^{\n_N}.$
\end{notation} 
\begin{lemma}
\label{lem:derivative estimations fund sol}
For any $\nu\in\mathbb{N}_0^N
$, there exists $C>0$, only dependent on $B$ and $\nu$
, such that
\begin{equation}
 \int_{\R^N} \big( | \partial^\nu_y \Gamma_t(z - y) | 
 + | \partial^\nu_y \Gamma_t(z - e^{Bt} y) |
 \big)\dd z \leq C 
 t^{-\frac{[\nu]_B }{2}}, \qquad t>0, \ y\in\R^N.
\end{equation}
\end{lemma}
\begin{proof}
For $[\nu]_B \leq 4$, it is a particular case of \cite[Proposition 6 and Lemma 8]{lucertini2022optimal}. The case $[\nu]_B >4$ is a straightforward extension: we omit the details for brevity.
\end{proof}
We now continue by proving a fundamental estimate on the Paley-Littlewood blocks of $\Gamma_t$. This extends a very similar result, \cite[Lemma 3.1]{zhang2021cauchy}, which is given for the kinetic case of Example \ref{ex:kinetic}.
\begin{lemma}
\label{lem:l1 norm density}
For any $l \geq 0$ and $T>0$, there is a positive constant $C=C(B,l,T)$ 
such that 
\begin{equation}
\label{eq:l1 norm density}
\| \Delta_j \Gamma_t\|_{\Leb^1}+\| \Delta_j \left(\Gamma_t \circ e^{tB} \right)\|_{\Leb^1}  \leq C  (t 4^j)^{-l} , \qquad t\in (0,T] ,\ j\in \Nzero\cup \{ -1 \}. 
\end{equation}
Note that $C$ above is independent of $j$ and $t$.
\end{lemma}
\begin{proof}
For brevity, we give a complete proof of the estimate on $\| \Delta_j \Gamma_t\|_{\Leb^1}$; analogous arguments apply to $\| \Delta_j \left(\Gamma_t \circ e^{tB} \right)\|_{\Leb^1}$. 
It is straightforward to prove that $\| \Delta_{-1} \Gamma_t \|_{\Leb^1}$ is bounded, uniformly with respect to $t\in[0,T]$. Therefore, it is enough to prove 
\eqref{eq:l1 norm density} for $j\in \Nzero$.

Fix $t\in(0,T]$ and $j\in \Nzero$.  
We first prove the uniform boundedness of the norms of the blocks. For any $z\in\R^N$ we have
\begin{align}
\Delta_j  \Gamma_t  (z) & =\big( \check\rho_j *  \Gamma_t  \big) (z) &&
\text{(by \eqref{eq: scaling propr part})} \\
& = 2^{jQ}\big(  (  \check{ \rho}_0 \circ D_{2^j}) * \Gamma_t    \big) (z) = \int_{\R^N} \check{ \rho}_0(y) \Gamma_t\big( z-2^{-j}.y \big)\dd y. \label{eq:Deltat_Gammat_rep}
\end{align}
As a result,
\begin{align}\label{eq:repr_L1_norm}
\| \Delta_j \Gamma_t\|_{\Leb^1} 
& = \int_{\R^N} \Big{|}\int_{\R^N} \check{ \rho}_0(y) \Gamma_t\left(z- 
2^{-j}.y 
\right)dy \Big{|} dz\\
 & \leq \int_{\R^N} \Big{|}\check{ \rho}_0(y)\Big{|}\underbrace{\Big(\int_{\R^N} \Gamma_t\big(z-2^{-j}.y\big)dz\Big)}_{=1}dy 
 =  \| \check{ \rho}_0 \|_{\Leb^1} .
\end{align}
Therefore, to prove \eqref{eq:l1 norm density}, it is not restrictive to assume that $t 4^{j}\geq 1$.
For 
any $n\in \N$ we set 
the operators 
\begin{equation}
\Delta_{i}^n :=(-1)^n \sum_{k=\bar{d}_{i-1}+1}^{\bar{d_{i}}} \partial^{2n}_{y_k}, \qquad i = 0, \cdots, r.
\end{equation}
Set also $p=\prod_{i=0}^r (2i+1)$ and define
\begin{equation}
\Lambda^m:= \Delta_{0}^{pm}+\Delta_{1}^{\frac{p}{3}m}\cdots + \Delta_{r}^{\frac{p}{2r+1}m}, \qquad m\in\N.
\end{equation}
By the property of the Fourier transform, the inverse of $\Lambda^m$ is defined for those $f\in \mathcal{S}$ such that supp$\fh$ does not contain the origin; in this case such operator acts as
\begin{equation}
\Lambda^{-m}f =\F^{-1}\left(\frac{ \fh(\xi)}{
 \sum_{i=0}^r \sum_{k=\bar d_{i-1} + 1}^{\bar d_{i}} |\xi_k|^{2 m\frac{ p}{2i + 1}}  }\right)=: \F^{-1}\left(\frac{ \fh(\xi)}{S_m(\xi)}\right).
\end{equation}
Now note that, by the chain rule, we have
\begin{equation}
\nabla_y  \Gamma_t\big(z- 2^{-j}.y \big) = 2^{-j}. \nabla_\zeta \Gamma_t\big(z-\zeta\big)|_{\zeta = 2^{-j}.y}, \qquad y,z\in\R^N,
\end{equation}
It is important to see that the $i$-th block of the gradient $\nabla_\zeta$ is multiplied by $2^{-j(2i+1)}$ every time that $\Gamma_t(z-2^{-j}.y)$ is differentiated with respect to a variable of the $i$-th block. By this fact it is not hard to see that
\begin{equation}\label{eq:repr_Lambda_m}
 \Lambda^m_y \big( \Gamma_t(z-2^{-j}.y) \big)= 4^{-jpm} 
  \Lambda^m_\zeta \big( \Gamma_t(z-\zeta) \big) \big{|}_{\zeta=2^{-j}.y}, \qquad y,z\in\R^N,
\end{equation}
where the indeces $y$ and {$\zeta$} in $\Lambda^m_y$ and $\Lambda^m_\zeta$, respectively, denote the variables the operator $\Lambda^m$ acts on. 

We can now conclude the proof. By \eqref{eq:Deltat_Gammat_rep} and Plancherel's Theorem, we obtain
\begin{align}
\Delta_j \Gamma_t (z)&
=\int_{\R^N} \rho_0(\xi) \overline{\F_y(\Gamma_t(z-2^{-j}.y))(\xi)}d\xi =\int_{\R^N} \frac{\rho_0(\xi)}{S_m(\xi)} \, S_m(\xi)\overline{\F_y(\Gamma_t(z-2^{-j}.y))}(\xi)d\xi
\intertext{(by applying once more Plancherel's Theorem together with the properties of the Fourier transform)}
 &=\int_{\R^N} \big( \Lambda^{-m}\check{ \rho}_0 \big)(y)\, \Lambda^m_y  \big( \Gamma_t  (z-2^{-j}.y ) \big)\, dy, \label{eq:Deltat_Gammat_rep_bit}
\end{align}
and thus, by 
\eqref{eq:repr_Lambda_m}, we have
\begin{align}
\| \Delta_j \Gamma_t\|_{\Leb^1} 
& \leq 
4^{-jpm}  \int_{\R^N} \int_{\R^N} \big| \big( \Lambda^{-m}\check{ \rho}_0 \big)(y) \big| \times \big| \Lambda^m_\zeta  \big( \Gamma_t  (z-\zeta ) \big)|_{\zeta=2^{-j}.y} \big| \, dy  dz
\intertext{(by Tonelli's Theorem and the estimate of Lemma \ref{lem:derivative estimations fund sol} {$N$ times with $\nu = \n_i = (0, \dots, 0 , \frac{2mp}{2j+1}, 0 , \dots, 0) \in \mathbb{N}_0^{i-1} \times \mathbb{N}_0 \times  \mathbb{N}_0^{N-i} = \mathbb{N}_0^N $ for $i = \bar d_{j-1}+1 , \dots, \bar d_j $ and $j=0, \dots , r$}) 
}
& \leq C\,  4^{-jpm} t^{-pm} \int_{\R^N} \big| \big( \Lambda^{-m}\check{ \rho}_0 \big)(y) \big|  dy . 
\end{align}
As $m\in\N$ is arbitrary and $t 4^j\geq 1$, this concludes the proof.
\end{proof}

The next two lemmas are also needed to prove Theorem \ref{th:schauder_first}.

\begin{lemma}\label{lem:exp}
For any $T>0$, there exists a positive $C=C(B,T)$ such that
\begin{equation}
\big| e^{t B^\top} \xi\big|_B  \leq C \bigg( \,\sum_{\substack{k,h=0,\cdots, r \\ k \geq h}} |t|^{\frac{k-h}{2h +1}} |\xi|_B^{({2k+1})/({2h+1})}     +    \sum_{\substack{k,h=0,\cdots, r \\ k < h}}  |\xi|_B^{({2k+1})/({2h+1})}     \bigg)
, \qquad \xi\in \R^N, \ t\in[-T,T].
\end{equation}
\end{lemma}
\begin{proof}
Throughout the proof, we will denote by $C$ any positive constant that depends at most on $B$ and $T$. 
We first study the block structure of the powers of $B^\top$. 
Recalling the structural hypotheses \eqref{B_tot}-\eqref{B}, we have
\begin{equation}
\label{eq:powers_B_transpose}
\big( (B^\top)^n\big)_{ij} = 0 , \qquad {\bf b}(j)-{\bf b}(i) > n,
\end{equation}
where ${\bf b}(i)$ denotes the block of variables corresponding to the index $i$, i.e.
\begin{equation}
{\bf b}(i) := \min \{  k : i \leq \bar d_k   \}, \qquad i = 1,\cdots, N.
\end{equation}
A graphical illustration of this structure is
\begin{equation}
B^\top =  \begin{pmatrix}
 \ast & \ast & 0 & \cdots & 0 & 0 \\
 \ast & \ast & \ast & 0 & \cdots&  0 \\
 \vdots & \vdots &\ddots & \ddots& \ddots&\vdots \vspace{-2pt} \\ 
 \ast & \ast & \ast & \ddots & \ast & 0 \\
 \ast & \ast & \ast & \cdots & \ast & \ast \\
 \ast & \ast & \ast & \cdots & \ast & \ast
  \end{pmatrix}, \qquad (B^\top)^2 = \begin{pmatrix}
 \ast & \ast & \ast & 0 & \cdots & 0 & 0 \\
 \ast & \ast & \ast & \ast & 0 & \cdots&  0 \\
 \vdots & \vdots &\ddots & \ddots &\ddots& \ddots&\vdots \vspace{-2pt} \\
 \ast & \ast & \ast &  \ddots & \ddots & \ast& 0 \vspace{-2pt} \\
 \ast & \ast & \ast &  \ast & \ddots & \ast & \ast\\
 \ast & \ast & \ast &  \ast & \cdots & \ast & \ast  \\
  \ast & \ast & \ast &  \ast & \cdots & \ast & \ast
 \end{pmatrix},
\end{equation}
and so on, until
\begin{equation}
\qquad (B^\top)^n =\begin{pmatrix}
 \ast & \ast & \cdots & \ast & \ast \\
 \ast & \ast & \cdots & \ast & \ast\\
 \vdots & \vdots &\ddots & \vdots & \vdots \\
 \ast & \ast & \cdots & \ast & \ast \\
 \ast & \ast & \cdots & \ast & \ast \\
\end{pmatrix}, \qquad n\geq r-1,
\end{equation}
with the blocks in the position $k,h$ representing $(d_k \times d_h)$ matrices.
As a result, for any $n\in\mathbb{N}$ and $\xi=(\xi_0,\cdots, \xi_r)\in\Rb^N$, we obtain  
\begin{align}
\label{eq:power est block}
|t^n (B^\top)^n \xi|_B & \leq  \sum_{h=0}^{r} \bigg(   
C^n\,   |t|^n \sum_{k=0}^{(h+n)\wedge r}  |\xi_k| \bigg)^{\frac{1}{2h+1}} 
\intertext{(since $|\xi_k| \leq |\xi|_B^{2k + 1} $)}
&\leq 
C^n   \sum_{h=0}^{r}  |t|^{n/(2h+1)}   \sum_{k=0}^{(h+n)\wedge r}   |\xi|_B^{(2k+1)/(2h+1)}   . 
\end{align}
For all the pairs of indices $(h,k)$ in the sums above we have
\begin{equation}
n \geq k-h ,
\end{equation}
and thus
\begin{equation}
|t|^{n/(2h+1)}
 \leq C^n |t|^{(k-h)/(2h+1)}, \qquad t\in [-T,T].
\end{equation}
This  and $|t|^{n/(2h+1)}\leq C^n$ yield
\begin{align}
|t^n (B^\top)^n \xi|_B\leq  C^n 
\bigg( \,\sum_{\substack{k,h=0,\cdots, r \\ k \geq h}} |t|^{\frac{k-h}{2h +1}} |\xi|_B^{({2k+1})/({2h+1})}     +    \sum_{\substack{k,h=0,\cdots, r \\ k < h}}  |\xi|_B^{({2k+1})/({2h+1})}     \bigg)
\end{align}
and concludes the proof.
\end{proof}
\begin{lemma}\label{lem:theta}
For any $T>0$ we have
\begin{equation}
\label{eq:inter probl}
 \big\{ i \in \Nzero\cup\{ -1 \} \, | \; \text{\emph{supp}} \rho_j \cap \emph{supp}\big( \rho_i \circ e^{tB^\top} \big) \neq \emptyset\big\} \subseteq \Theta^t_j , \qquad t\in [-T,T],\quad j\in\mathbb{N}_0 \cup \{-1 \}, 
\end{equation}
where
\begin{align}
\label{def:theta} 
\Theta^t_j &:=\bigg\{ i \in \Nzero\cup\{ -1 \}\, \big| \, 2^i   \leq   C  \sum_{\substack{k,h=0,\cdots, r \\ k \geq h}} |t|^{\frac{k-h}{2h +1}} 2^{\frac{j({2k+1})}{{2h+1}}} , \, 2^j   \leq   C  \sum_{\substack{k,h=0,\cdots, r \\ k \geq h}} |t|^{\frac{k-h}{2h +1}} 2^{\frac{i({2k+1})}{{2h+1}}}  \bigg\}, \quad j\in \N_0, \hspace{20pt} \\
\Theta^t_{-1} & := \bigg\{ i \in \Nzero\cup\{ -1 \}\, \big| \, 2^i   \leq   C  \sum_{\substack{k,h=0,\cdots, r \\ k \geq h}} |t|^{\frac{k-h}{2h +1}}   \bigg\} , 
\end{align}
with $C=C(B,T)$ positive constant. 

Furthermore, for any $\gamma\in \R\setminus\{ 0 \}$, there exists a constant $C=C(B,T,\gamma)>0$ such that
\begin{align}
\label{eq:Theta est}
 \sum_{i\in \Theta^t_j} 2^{-\gamma i}& \leq  
C\, 2^{-j \gamma} 
 \sum_{\substack{k,h=0,\cdots, r \\ k \geq h}} (|t| 4^j)^{|\gamma|({k-h})/({2h +1})}  
 , && t\in [-T,T], \quad j\in \Nzero,\\
 \sum_{i\in \Theta^t_{-1}} 2^{-\gamma i}& \leq  
C, && t\in[-T,T]. \label{eq:Theta est_bis}
\end{align}
\end{lemma}
\begin{proof}
Fix $j\in\mathbb{N}_0$. For any $i\in\mathbb{N}_0 \cup \{ -1 \}$, the second property in \eqref{eq:def_rho_minus_one} and the property \eqref{rmk:point (b) support partition} of the partition yield
\begin{equation}
\xi \in \text{supp} \rho_j \cap \text{supp} \big(\rho_i \circ e^{tB^\top}\big) \Longrightarrow
\begin{cases}
 |\xi|_B\leq 2^{j+1} \\
  | e^{tB^\top} \xi |_B \geq 2^{i-1} - \frac{1}{4}
\end{cases} ,
\end{equation}
and also
\begin{equation}
 \zeta \in \text{supp} \big(\rho_j \circ e^{-tB^\top} \big)\cap \text{supp} \rho_i \Longrightarrow
\begin{cases}
 |e^{-tB^\top} \zeta|_B \geq 2^{j-1}
\\
| \zeta |_B \leq 2^{i+1}
\end{cases} .
\end{equation}
By applying Lemma \ref{lem:exp} to the first system of inequalities, we have 
\begin{align}
2^{i-1} - \frac{1}{4} \leq |e^{t B^\top}\xi|_B & \leq C \bigg( \,\sum_{\substack{k,h=0,\cdots, r \\ k \geq h}} |t|^{\frac{k-h}{2h +1}} |\xi|_B^{({2k+1})/({2h+1})}     +    \sum_{\substack{k,h=0,\cdots, r \\ k < h}}  |\xi|_B^{({2k+1})/({2h+1})} \bigg)\\
& \leq C \bigg( \,\sum_{\substack{k,h=0,\cdots, r \\ k \geq h}} |t|^{\frac{k-h}{2h +1}} 2^{(j+1)({2k+1})/({2h+1})}     +    \sum_{\substack{k,h=0,\cdots, r \\ k < h}}  2^{(j+1)({2k+1})/({2h+1})}     \bigg) \\
\intertext{(since $2^{j+1}\geq 1$ {then $k<h\Rightarrow 2^{(j+1)({2k+1})/({2h+1})}\leq 2^{j+1}$})}
& \leq C \,\sum_{\substack{k,h=0,\cdots, r \\ k \geq h}} |t|^{\frac{k-h}{2h +1}} 2^{(j+1)({2k+1})/({2h+1})}     .
\end{align}
Analogously, the second system leads to
\begin{align}
2^{j-1}\leq C \,\sum_{\substack{k,h=0,\cdots, r \\ k \geq h}} |t|^{\frac{k-h}{2h +1}} 2^{(i+1)({2k+1})/({2h+1})}  , 
\end{align}
which, combined with the previous inequality, yields \eqref{eq:inter probl} for $j\in \mathbb{N}_{0}$. Furthermore, 
for any $i\in\mathbb{N}_0 \cup \{ -1 \}$, by \eqref{eq:def_rho_minus_one} we have
\begin{equation}
\xi \in \text{supp} \rho_{-1} \cap \text{supp} \big(\rho_i \circ e^{tB^\top}\big) \Longrightarrow
\begin{cases}
 |\xi|_B\leq 1 \\
  | e^{tB^\top} \xi |_B \geq 2^{i-1} - \frac{1}{4}
\end{cases} .
\end{equation}
Proceeding like above yields  \eqref{eq:inter probl} for $j=-1$.

Prove now \eqref{eq:Theta est}. We only show it when $\gamma>0$, the case $\gamma < 0$ being easier.
Let $i\in\Theta_t^j$. By the second inequality in \eqref{def:theta} we have
\begin{equation}\label{eq:ineq_theta}
2^{-i}\leq C 2^{-j} \sum_{\substack{k,h=0,\cdots, r \\ k \geq h}} |t|^{\frac{k-h}{2h +1}} \underbrace{2^{\frac{i({2k+1})}{{2h+1}}-i}}_{=4^{i(k-h)/(2h+1)}}.
\end{equation} 
Now, if $i\leq j$, then \eqref{eq:ineq_theta} yields
\begin{equation}
2^{-i} \leq C 2^{-j} \sum_{\substack{k,h=0,\cdots, r \\ k \geq h}} |t|^{\frac{k-h}{2h +1}} 4^{j(k-h)/(2h+1)} :=D_j.
\end{equation}
For $i>j$ we have $2^{-i}<2^{-j}$, and by definition of $\Theta^t_j$ we get
\begin{equation}
2^{-i}<2^{-j}\leq D_j,
\end{equation}
that is $i\geq -\log_2 D_j$ for every $i\in \Theta^t_j$.
This latter implies
\begin{equation}
\sum_{i\in \Theta^t_j} 2^{-\gamma i} \leq \sum_{i=\lfloor -\log_2 D_j \rfloor}^{\infty} 2^{-\gamma i} = \frac{2^{\lfloor \log_2 D_j \rfloor \gamma} }{1-2^{-\gamma}}  \leq \frac{2^{\beta} D_j^\gamma}{1-2^{-\gamma}}, 
\end{equation}
which yields \eqref{eq:Theta est} for $\gamma>0$. Finally, \eqref{eq:Theta est_bis} is trivial.
\end{proof}
We are now in the position to prove Theorem \ref{th:schauder_first}.
\begin{proof}[Proof of Theorem \ref{th:schauder_first}]
We fix $t\in ( 0, T]$ and write $P_t$ and $P'_t$ in a convenient way, suitable to exploit the properties of the Fourier transform. 
We have 
\begin{equation}
\Gamma_t * ( f \circ e^{-tB} ) (z)= \int_{\R^N}  \Gamma_t(z-y) f(e^{-tB}y) dy\\
=|\det e^{tB}| \int_{\R^N}\Gamma_t\big(z - e^{tB}y\big)  f(y)  d y , \qquad f\in \mathcal{S},
\end{equation}
\begin{equation}
\big((\Gamma_t\circ e^{tB}) * ( f \circ e^{tB} ) \big) (y)= \int_{\R^N}  \Gamma_t\big(e^{tB}(z-y)\big) f(e^{tB}z) dz\\
=\frac{1}{|\det e^{tB}|} \int_{\R^N}\Gamma_t\big(z - e^{tB}y\big)  f(z)  d z , \qquad f\in \mathcal{S},
\end{equation}
and thus
\begin{equation}
{P'_t f} = \frac{1}{|\det e^{tB}|} \, \Gamma_t * ( f \circ e^{-tB} ), \qquad {P_t f} = |\det e^{tB}| \, (\Gamma_t\circ e^{tB}) * ( f \circ e^{tB} ),\qquad f\in \mathcal{S},
\end{equation}
which extends to $f\in \mathcal{S}'$ by duality. 
As the following steps apply to $P_t$ and similarly to $P'_t$, for the sake of brevity, we prove only the part of \eqref{eq:schauder 1}  related to {$P'_t$}.
For any $j\geq -1$ we have obtained that
\begin{equation}
\Delta_j ({P'_t f}) =\frac{1}{|\det e^{tB}|}\,\big(\Delta_j \Gamma_t \big) * ( f \circ e^{-tB} )  .
\end{equation}
Furthermore, by the second part of Remark \ref{rem:paley_little}, we have
\begin{equation}
f \circ e^{-tB} = \sum_{i=-1}^{+\infty} ( \Delta_i f ) \circ e^{-tB} ,
\end{equation}
and thus
\begin{equation}
\Delta_j ( {P'_t f} ) = \frac{1}{|\det e^{tB}|} \sum_{i=-1}^{\infty}   \big(\Delta_j  \Gamma_t \big)* \big( (\Delta_i f)  \circ e^{-tB} \big) .
\end{equation}
Each term of this summation is nonzero if and only if
\begin{equation}
 \F\Big(  \big(\Delta_j  \Gamma_t  \big)* \big(  (\Delta_i f)  \circ e^{-tB} \big)  \Big)
 = |\det e^{tB}| \, \rho_j
 \widehat{\Gamma}\,
  \big( \rho_i \circ e^{tB^\top}\big) 
   \big( \fh\circ e^{tB^\top}\big) 
\end{equation}
is non-null. Therefore, Lemma \ref{lem:theta} yields
\begin{equation}
\Delta_j (  P'_tf )=\frac{1}{|\det e^{tB}|}\sum_{i\in \Theta^t_j}  \big(\Delta_j  \Gamma_t \big)* \big( ( \Delta_i f ) \circ e^{-tB} \big)  , 
\end{equation}
and by Young's inequality we obtain
\begin{equation}
\|\Delta_j (  P'_tf ) \|_{\Leb^\infty} \leq \frac{1}{|\det e^{tB}|}\| \Delta_j  \Gamma_t \|_{\Leb^1} \sum_{i\in \Theta^t_j}  \underbrace{\|  ( \Delta_i f ) \circ e^{-tB}  \|_{\Leb^\infty}}_{=\|\Delta_i f\|_{\Leb^\infty}} \leq  \frac{1}{|\det e^{tB}|}\| \Delta_j  \Gamma_t \|_{\Leb^1} \|f\|_\gamma \sum_{i\in \Theta^t_j} 2^{-i\gamma}. 
\end{equation}
Assume now $\gamma\neq 0$ and $j\in\N_0$: \eqref{eq:Theta est} yields 
\begin{align}
\|\Delta_j ( P'_t f ) \|_{\Leb^\infty} 
& \leq  C \| \Delta_j \Gamma_t \|_{\Leb^1} \|f\|_\gamma \, 2^{-j \gamma} 
\sum_{\substack{k,h=0,\cdots, r \\ k \geq h}} (t 4^j)^{|\gamma|({k-h})/({2h +1})}.
\end{align}
Now, for each $0\leq h \leq k \leq r$, we apply Lemma \ref{lem:l1 norm density} with 
\begin{equation}
l=\frac{\alpha}{2}+ \frac{k-h}{2h+1}|\gamma|
\end{equation}
and obtain
\begin{equation}
(t 4^j)^{|\gamma|({k-h})/({2h +1})} \| \Delta_j  \Gamma_t \|_{\Leb^1}
\leq C (t 4^{j})^{-l}  (t 4^{j})^{|\gamma|\frac{k-h}{2h +1}}=C t^{-\frac{\alpha}{2}} 2^{-j \alpha} .
\end{equation}
This proves
\begin{equation}\label{eq:schauder_est_proof}
 \|\Delta_j (P'_t f)\|_{\Leb^{\infty}}\leq  C\, t^{-\frac{\alpha}{2}} 2^{-j(\alpha+\gamma)} \|f\|_\gamma, \qquad j\in\mathbb{N}_0.
\end{equation}
Following analogous arguments, the same estimate can be proved for $j=-1$. We omit the details for brevity. Finally, the case $\gamma=0$ follows by the interpolation Lemma 2.10 in \cite{hao2024singular}. Precisely, employing the notation in the latter reference, estimate \eqref{eq:schauder_est_proof} for $\gamma = 0$ can be obtained by setting: $\mathbb{X}={L}^\infty(\R^N)$, $\beta_0=-1$, $\beta_1=1$, $T_j=\Delta_j P'_t$, $C_{0,j}=C_{-1}2^{-j\alpha}t^{-\frac{\alpha}{2}}$ and $C_{1,j}=C_{1}2^{-j\alpha}t^{-\frac{\alpha}{2}}$, with $C_{-1}, C_{1}$ denoting the constant $ C$ appearing in \eqref{eq:schauder_est_proof} for $\gamma=-1,1$ respectively.
\end{proof}

{\bf ACKNOWLEDGEMENTS.}
 The research of Elena Issoglio was partially supported by PRIN-PNRR2022 (P20224TM7Z) CUP: D53D23018780001 and from EU - Next Generation
EU - PRIN2022 (202277N5H9) CUP: D53D23005670006. 
  The research of Stefano Pagliarani was partially supported by the INdAM - GNAMPA project CUP\_E53C22001930001 and by the
PRIN22 project CUP E53C23001670001. 
The research of Francesco Russo was supported by the
ANR-22-CE40-0015-01
project (SDAIM).

\bibliographystyle{siam}
\bibliography{../../BIBLIO_FILE/Bibtex-Final.bib}
\end{document}